\pdfoutput=1
\RequirePackage{ifpdf}
\ifpdf 
\documentclass[pdftex]{sigma}
\else
\documentclass{sigma}
\fi

\usepackage{mathabx}

\usepackage[all,2cell]{xy} \UseAllTwocells \SilentMatrices

\newtheorem{thm}{Theorem}[section]
\newtheorem{cor}[thm]{Corollary}
\newtheorem{lem}[thm]{Lemma}
\newtheorem{prop}[thm]{Proposition}
\newtheorem{conj}[thm]{Conjecture}
{\theoremstyle{definition}
\newtheorem{defn}[thm]{Definition}
\newtheorem{Example}[thm]{Example}
\newtheorem{rem}[thm]{Remark}
}

\numberwithin{equation}{section}

\newcommand{\To}{\longrightarrow}
\newcommand*{\Longhookrightarrow}{\ensuremath{\lhook\joinrel\relbar\joinrel\rightarrow}}
\newcommand{\Z}{\mathbb Z}
\newcommand{\Q}{\mathbb Q}
\newcommand{\C}{\mathbb C}

\newcommand{\R}{\mathbb R}

\newcommand{\Pro}{\mathbb P}

\newcommand{\GC}{\mathcal{GC}}
\newcommand{\q}{/\!/}

\newcommand{\tr}{\mathrm{tr}}
\newcommand{\id}{\mathrm{id}}

\newcommand{\can}{\mathrm{can}}

\newcommand{\mm}{\mathfrak{m}}

\newcommand{\GL}{\mathrm{GL}}

\newcommand{\mc}{\mu}

\newcommand{\0}{\color{blue}{\mathsf{0}}}

\begin{document}
\allowdisplaybreaks

\newcommand{\arXivNumber}{2101.04419}

\renewcommand{\thefootnote}{}

\renewcommand{\PaperNumber}{103}

\FirstPageHeading

\ShortArticleName{Invariant Differential Forms on Complexes of Graphs and Feynman Integrals}

\ArticleName{Invariant Differential Forms on Complexes of Graphs\\ and Feynman Integrals\footnote{This paper is a~contribution to the Special Issue on Algebraic Structures in Perturbative Quantum Field Theory in honor of Dirk Kreimer for his 60th birthday. The~full collection is available at \href{https://www.emis.de/journals/SIGMA/Kreimer.html}{https://www.emis.de/journals/SIGMA/Kreimer.html}}}

\Author{Francis BROWN}

\AuthorNameForHeading{F.~Brown}

\Address{All Souls College, University of Oxford, Oxford, OX1 4AL, UK}
\Email{\href{mailto:francis.brown@all-souls.ox.ac.uk}{francis.brown@all-souls.ox.ac.uk}}

\ArticleDates{Received March 04, 2021, in final form November 14, 2021; Published online November 23, 2021}

\Abstract{We study differential forms on an algebraic compactification of a moduli space of metric graphs. Canonical examples of such forms are obtained by pulling back invariant differentials along a tropical Torelli map. The invariant differential forms in question generate the stable real cohomology of the general linear group, as shown by Borel. By~integrating such invariant forms over the space of metrics on a graph, we define canonical period integrals associated to graphs, which we prove are always finite and take the form of generalised Feynman integrals. Furthermore, canonical integrals can be used to detect the non-vanishing of homology classes in the commutative graph complex. This theory leads to insights about the structure of the cohomology of the commutative graph complex, and new connections between graph complexes, motivic Galois groups and quantum field theory.}


\Classification{18G85; 11F75; 11M32; 81Q30}

\renewcommand{\thefootnote}{\arabic{footnote}}
\setcounter{footnote}{0}

\section{Homology of the commutative graph complex}

We consider the graph complex introduced by Kontsevich in \cite{Kontsevich93}, which he refers to as the odd, commutative graph complex. It is denoted by $\GC_N$ in \cite{WillwacherGRT}, where $N$ is any fixed even integer. We review the definitions and some known results about its homology.
\subsection{Definitions}
Let $G$ be a connected graph. Let $V_G$, $E_G$ denote its set of vertices, and edges, and denote by
\begin{gather*}
h_G\colon\ \text{the number of loops, or genus, of } G,
 \\
e_G = |E_G|\colon\ \text{the number of edges of } G ,
\\
\deg_N G = e_G - N h_G\colon\ \text{which will be called the degree of } G .
\end{gather*}
In the case $N=0$ the degree coincides with the number of edges. In the case $N=2$,
the degree is minus what is sometimes called the ``superficial degree of divergence'' in the physics literature. An \emph{orientation} of $G$ is an element
\begin{gather*}
\eta \in \big(\textstyle{\bigwedge^{\!e_G}} \Z^{E_G} \big)^{\times} .
\end{gather*}
If the edges of $G$ are denoted by $e_1,\dots, e_n$, where $n=e_G$, then an orientation
is equal to either $e_1\wedge \dots \wedge e_n$ or its negative. Thus an orientation is simply an ordering of the edges of $G$ up to the action of even permutations.

The notation
 $G/\gamma$ will denote the graph obtained by contracting all the edges of a~sub\-graph~$\gamma$ of $G$
(defined by a subset of the set of edges of $G$). It is defined by removing every edge of $\gamma$, in any order, and identifying its endpoints.
It is convenient to use a different notation for the operation:
\begin{gather*}
G\q \gamma = \begin{cases}
G/\gamma& \text{if}\ h_{\gamma}=0,
\\
\varnothing & \text{if}\ h_{\gamma}>0.
\end{cases}
\end{gather*}
In other words, the contraction $G\q \gamma$ is the empty graph if $\gamma$ contains a loop.

Let $\GC_N$ denote the $\Q$-vector space generated by pairs $(G,\eta)$, where $G$ is a connected graph and $\eta$ an orientation, such that: $G$ has no tadpoles (edges bounding on a single vertex) and no vertices of degree $\leq 2$, modulo the equivalence relations
\begin{gather}
(G, - \eta) = -(G, \eta), \nonumber
\\
 (G, \eta) = (G', \sigma (\eta)),\label{GequivalenceRelations}
 \end{gather}
where $\sigma$ is any isomorphism $\sigma\colon G \overset{\sim}{\rightarrow} G'$. Denote the equivalence class of $(G,\eta)$ by $[G,\eta]$. The differential in $\GC_N$ is defined by
\begin{gather*}
{\rm d} [ G, e_1 \wedge \dots \wedge e_n ] = \sum_{i=1}^n (-1)^i [G\q e_i, e_1 \wedge \dots \wedge \widehat{e_i} \wedge \dots \wedge e_n ].
\end{gather*}
No tadpoles can arise in the right-hand side because graphs with double edges vanish in $\GC_N$ by~\eqref{GequivalenceRelations}. One checks that the differential is well-defined and satisfies ${\rm d}^2=0$. Furthermore, it preserves the loop number~$h$, and decreases the degree $\deg_N$ by $1$.

\begin{defn} The \emph{graph homology} is defined to be the vector space:
\begin{gather*}
H(\GC_N) = \frac{\ker {\rm d}}{\operatorname{Im} {\rm d}}.
\end{gather*}
 It is graded by homological degree (denoted $H_n(\GC_N)$), where $n=\deg_N G$ is the degree of $G$:
 \begin{gather*}
 H(\GC_N ) = \bigoplus_{n \in \Z} H_n(\GC_N),
 \end{gather*}
and also by the number of loops $H_n(\GC_N) = \bigoplus_{h\geq0} H_n(\GC_N)^{(h)}$. It is therefore bigraded.
 \end{defn}
 The graph complexes $\GC_N$ for all even $N$ are mutually isomorphic, so modifying $N$ merely changes the grading by degree.
 In this paper, the grading by loops plays a secondary role, and we work essentially with $\GC_0$ for the most part. However, for the purposes of the introduction we will discuss the case of $\GC_2$ because it makes the comparison with results in the literature more explicit and because the figures below take up considerably less space on the page.

\subsection{Examples}
Any graph admitting an automorphism which acts on its set of edges by an odd permutation vanishes in $\GC_N$ by \eqref{GequivalenceRelations}. In particular, a graph which contains
 a doubled edge is zero. It follows that any graph with the property that every edge is contained in a triangle is closed in the graph complex, since contracting an edge of a triangle leads to a doubled edge.

 Consider the wheel with $n$ spokes depicted in Figure \ref{figWheel}.
 \begin{figure}[h]
\centering
{\includegraphics[width=4.8cm]{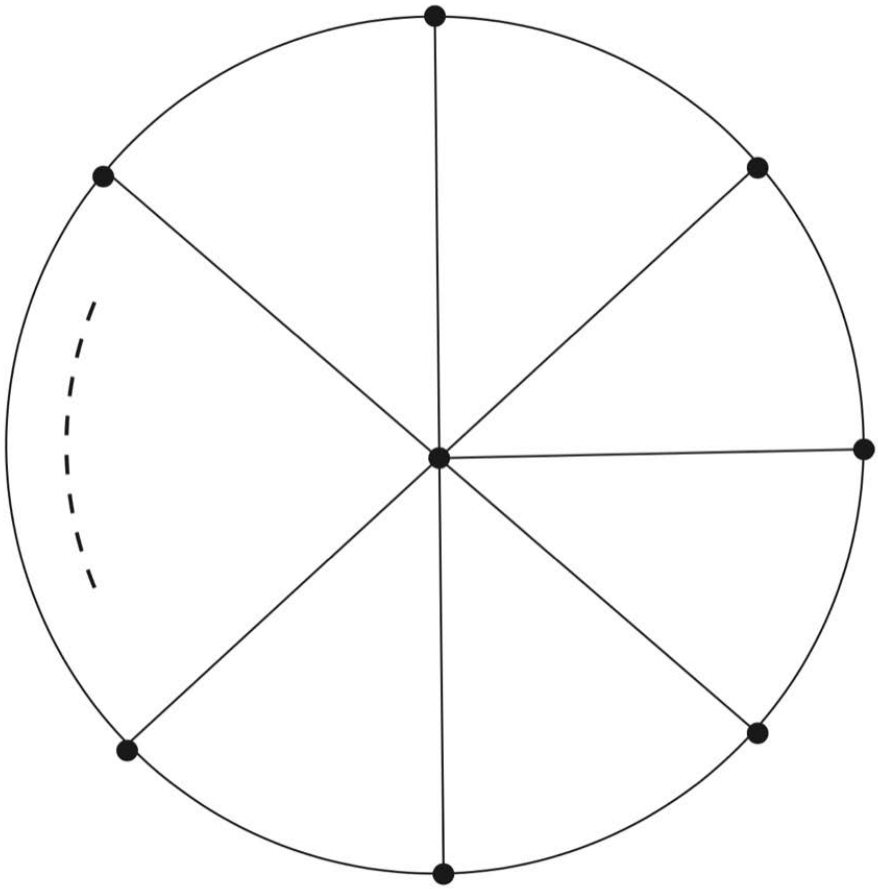}}
 \caption{The wheel with $n$ spokes $W_n$.} \label{figWheel} \end{figure}
Since every edge lies in a triangle, ${\rm d}[W_n]=0$ (here and henceforth, a choice of orientation will be implicit in the notation for a graph and will be omitted). Since the even wheels $W_{2k}$ admit an odd automorphism, they vanish in the graph complex. One knows (e.g., by \cite{SpectralSequenceGC2}) that the odd wheel classes $[W_{2n+1}] $ are non-zero in homology:
\begin{gather*}
[W_{2n+1} ] \in H_0(\GC_2)
\end{gather*}
for all $n \geq 1$. The graph $W_{2n+1}$ has $2n+1$ loops, and $4n+2$ edges.

\subsection{Known results}
Table \ref{table:knownresults} depicts computer calculations of graph homology for $N=2$ in low degrees. At the time of writing, little is known explicitly in homological degrees $\geq 1$ beyond $11$ loops.

\begin{table}[h]
\caption{Dimensions of $H_n(\GC_2)$ at low loop order \cite{GraphComplexComputations}. The (red) classes in $H_0(\GC_2)$ with 3, 5, 7, 9 loops are generated by the wheels $W_3$, $W_5$, $W_7$, $W_9$. Other classes in this diagram are presumably only representable as linear combinations of graphs.}
\label{table:knownresults}
\begin{center}
\begin{tabular}{c|cccccccccc}
$H_8$ &&&&&&&&&& $\0$ \\
$H_7$ &&&&&&&&& $\0$& 1 \\
$H_6$ &&&&&&&& $\0$ & 0& 0 \\
$H_5$ &&&&&&& $\0$ & 0 & 0& 0 \\
$H_4$ &&&&&& $\0$ & 0 & 0& 0 &0\\
$H_3$ &&&&& $\0$ & 1 & 0& 1 &1 & 2 \\
$H_2$ &&&& $\0$ & 0&0 & 0 & 0 & 0 & 0 \\
$H_1$ &&& $\0$ & 0 & 0 & 0 & 0 & 0 & 0 & 0 \\
$H_0$ && $\0$ & \color{red}{1} & 0&\color{red}{1} &0 & \color{red}{1} & 1& \color{red}{1} & 1 \\
\hline
$h_G$ & $1$ & $2$ & $3$ & $4$ & $5$ & $6$ & $7$ & $8$& $9$& $10$
\end{tabular}
\end{center}
\end{table}
All trivalent (3-regular) graphs lie along the diagonal line $e_G = 3(h_G-1)$. All graphs above this line (blue entries and above) satisfy $e_G \geq 3h_G-2$ and vanish in $\GC_2$ since they have a~2-valent vertex.
\medskip

One knows that:
\begin{enumerate}\itemsep=0pt
\item The homology groups $H_n(\GC_2)$ vanish in negative degrees $n<0$ in loop degree $\geq 1$ (shown in \cite{WillwacherGRT} and interpreted geometrically in \cite{CGP}).

\item Willwacher showed \cite{WillwacherGRT} that there is an isomorphism of coalgebras (see below for the definition of the coalgebra structure on graph homology)
\begin{gather} \label{WillwacherGRT} H_0(\GC_2) \cong \mathfrak{grt}^{\vee},
\end{gather}
where $\mathfrak{grt}$ denotes the Grothendieck--Teichm\"uller Lie algebra introduced by Drinfeld in~\cite{Drinfeld}.
It is explicitly defined by generators and relations \cite{Furusho}, but little is known about its structure. A conjecture of Deligne, proved in \cite{BrMTZ}, implies that it contains the graded Lie algebra of the motivic Galois group of mixed Tate motives over the integers
$\mathcal{MT}(\Z)$:
\begin{gather} \label{MTZinjection}
\mathbb{L} (\sigma_3, \sigma_5, \dots ) \cong \mathrm{Lie} \big(G^{\mathrm{mot}}_{\mathcal{MT}(\Z)}\big) \Longhookrightarrow \mathfrak{grt}.
\end{gather}
The latter Lie algebra is isomorphic to the free graded Lie algebra $\mathbb{L} (\sigma_3, \sigma_5, \dots )$ with one generator $\sigma_{2n+1}$ in every odd degree $-(2n+1)$, for $n\geq 1$. These generators are not canonical for $n\geq 5$, but are known to pair non-trivially with the wheel graphs $W_{2n+1}$ via \eqref{WillwacherGRT}. Note that the isomorphism~\eqref{WillwacherGRT} is combinatorial -- there is presently no known geometric action of the motivic Lie algebra on graph homology.
\end{enumerate}
From (2) one infers the existence of a graph homology class $\xi_{3,5} \in H_0(\GC_2)$ at $8$ loops, dual to $[\sigma_3, \sigma_5]$; and a class
$\xi_{3,7} \in H_0(\GC_2)$ at 10 loops dual to $[\sigma_{3},\sigma_{7}]$. In degree $11$, an $11$-loop class $\xi_{3,3,5}\in H_0(\GC_2)$ dual to $[\sigma_3, [\sigma_3, \sigma_5]]$ appears. It is only well-defined up to addition of a rational multiple of $[W_{11}]$.

\begin{rem} 
Drinfeld asked the question of whether \eqref{MTZinjection} is an isomorphism. The graded Lie coalgebra dual to $ \mathrm{Lie} \big(G^{\mathrm{mot}}_{\mathcal{MT}(\Z)}\big)$ is isomorphic to the Lie coalgebra of motivic multiple zeta values modulo the motivic version of $\zeta(2)$ and modulo products. The latter space carries many additional structures, including a depth filtration and an intimate relation to modular forms. These two additional structures are not presently understood on the level of graph homology, to our knowledge.
\end{rem}

\subsection{Further structures} \label{sect: FurtherStructures}
In addition to the differential ${\rm d}$, we consider two more operations on graphs. They do not preserve~$\GC_N$, so in order to incorporate them, one must relax the definitions of the graph complex. Instead of doing this, we observe that these operations will only appear via an integration formula \eqref{IntroStokes}, in which all terms corresponding to graphs which lie outside $\GC_N$, i.e., which have a vertex of degree $\leq 2$ or a tadpole, automatically vanish by Proposition~\ref{prop: vanishing}.

The first additional structure is a ``second'' differential which deletes edges:
\begin{gather} \label{deltadef}
\delta [ G , e_1\wedge \dots \wedge e_n] = \sum_{i=1}^n (-1)^i \big[ G\backslash e_i,, \ e_1\wedge \dots \wedge \widehat{e_i} \wedge \dots \wedge e_n\big],
\end{gather}
where $G \backslash e_i$ is the graph $G$ with the same vertex set but with the edge $e_i$ deleted. One checks again that $\delta$ is well-defined on graph isomorphism classes and satisfies $\delta^2=0$ and ${\rm d} \delta+ \delta {\rm d} =0$. It has degree $N-1$.
Note that deleting an edge can generate 2-valent vertices, and so $\delta$ does not preserve the graph complex $\GC_N$.
It does, however, preserve the complex $\GC_N^{\geq 2}$ of graphs with no vertices of degree $\leq 1$, and it is
observed in \cite{SpectralSequenceGC2} that the graph complex $\GC_0^{\geq 2}$ has trivial homology with respect to $\delta$, since adjoining an edge in all possible ways defines a homology inverse.
Consequently, one shows that there exists an infinite family of non-trivial higher degree classes in $H_n(\GC_2)$, $n>0$, via a spectral sequence argument \cite{SpectralSequenceGC2}. The existence of these classes unfortunately uses \eqref{MTZinjection} in an essential way.

The second additional structure is the Connes--Kreimer coproduct \cite{ConnesKreimer}:
\begin{gather} \label{DeltaCK}
\Delta G = \sum_{\gamma \subset G} \gamma \otimes G/\gamma ,
\end{gather}
where $\gamma$ ranges over core (1-particle irreducible, or bridgeless) subgraphs of $G$. It defines a~coassociative coproduct which is compatible with both differentials. However, once again it does not preserve the graph complex -- for example, $G/\gamma$ may contain tadpoles, and if $G$ has a~bridge, then $G/\gamma$ may have a vertex of degree one. By antisymmetrizing the coproduct one obtains a~cobracket dual
 to the Connes--Kreimer Lie bracket \cite{ConnesKreimer}, which is given
by a signed sum of all vertex insertions of one graph into another. See \cite[Section~6.9]{FeynmanCat} for another interpretation. It~nduces a Lie algebra structure on graph cohomology.

We shall provide a geometric interpretation of both $\eqref{deltadef}$ and $\eqref{DeltaCK}$ via the boundary structure of a compactification of the space of metric graphs.
\subsection{Comments and questions}
Recently Chan, Galatius and Payne proved in \cite[Theorems 1 and 2]{CGP} that for all $g\geq 2$, the highest non-zero weight-graded piece of the cohomology of $\mathcal{M}_g$, the moduli space of curves of genus $g$ (which by Deligne \cite{delignehodge2} carries a canonical mixed Hodge structure) satisfies
\begin{gather} \label{CGP}
{\rm g r}^W_{6g-6} H^{4g-6-n} (\mathcal{M}_g;\Q) \overset{\sim}{\longrightarrow} H_{n}( \GC_2)^{(g)}.
\end{gather}
Using known results about the graph complex they deduced new information about the cohomology of $\mathcal{M}_g$. The existence of the
wheel class $[W_3]$, for example,
 corresponds to the fact, first proved by Looijenga \cite{Looijenga}, that
$H^6(\mathcal{M}_3;\Q)\cong \Q(-6)$, a pure Tate mixed Hodge structure of weight 12.

\begin{rem} \label{remPuzzlezeta(3)} The following puzzle was a principal motivation for this project. Simply put, \eqref{WillwacherGRT} and \eqref{MTZinjection} suggest that the motivic Galois group $G^{\mathrm{mot}}_{\mathcal{MT}(\Z)}$, and hence its Lie algebra, should act naturally on $H_0(\GC_2)$. The point is that not every graded Lie algebra which is structurally isomorphic to a free Lie algebra of the form $\mathbb{L}(\sigma_3, \sigma_5,\dots )$, is necessarily naturally motivic, i.e., admits a natural action by $\mathrm{Lie} \big(G^{\mathrm{mot}}_{\mathcal{MT}(\Z)}\big)$.

If the motivic Galois group were to act naturally upon $H_0(\GC_2)$, then by the Tannakian formalism, the latter would be endowed with the structure of a mixed Tate motive over $\Z$, and hence we would expect the left-hand side of \eqref{CGP}, or certainly the part which corresponds to $H_0(\GC_2)$, to correspond naturally to a mixed Tate motive over the integers. It would involve non-trivial extensions of pure Tate objects, whose extension classes are detected by periods which are multiple zeta values. However, the object on the left-hand side of \eqref{CGP} is by definition only a pure motive: in fact, a direct sum of copies of Tate motives $\Q(3-3g)$.

For example, the very meaning of the element $\sigma_3$ is that it corresponds (or rather, is dual) to an extension class
\begin{gather} \label{introExtzeta3}
0 \To \Q \To \mathcal{E} \To \Q(-3) \To 0,
\end{gather}
where $\mathcal{E}$ is a mixed Tate motive.
The non-triviality of this extension is detected by its period, which is proportional to $\zeta(3)$.
In this paper we shall naturally associate an extension of Tate motives of the form \eqref{introExtzeta3} to the class $[W_3]$ whose period is indeed a multiple of $\zeta(3)$ (in fact $60 \zeta(3)$) and conjecture that the same applies to all the odd wheel classes.
It seems that, up to Tate twisting, the left-hand side of \eqref{CGP} sees only one piece of the associated weight-graded object $\mathrm{gr}_{\bullet}^W \mathcal{E} = \Q\oplus \Q(-3)$, which is split.
\end{rem}

In the light of the previous remark, it may be reasonable to expect that the cohomology of the graph complex in its entirety has the structure of a non-trivial mixed motive.

\medskip

The previous discussion thus raises the following questions:
\begin{enumerate}\itemsep=0pt
\item How should one interpret higher degree graph homology classes?
\item How is the graph complex related to mixed motives and periods?
\end{enumerate}

In this paper we shall use the theory of invariant forms on locally symmetric spaces to define (motivic) periods associated to graphs. This leads to a conjectural interpretation of infinitely many higher degree classes in the graph complex.

\section{Overview of contents}

This section provides some commentary and background motivation for the main contents of the paper. The reader may wish to return to the present section periodically while reading the rest of the paper.

The main thrust of this paper is to study differential forms on a geometric incarnation of the graph complex.
For this, we consider a certain moduli space of metric graphs, which is related to both the moduli space of tropical curves \cite{BMV} and Culler and Vogtmann's Outer space \cite{CullerVogtmann}, and then go on to explain how to construct differential forms upon this space.

A possible point of confusion is the different use of the word ``marking'' in the literature, which can refer to three different things. A ``marked graph'' commonly means a graph with external half-edges, corresponding to the moduli space of curves with \emph{marked} points. However, we shall not consider any such graphs in this paper, and will therefore not use the term. In~\cite{BMV}, a ``marking'' refers to what we shall call a weighting on vertices; finally, in the context of Outer space \cite{CullerVogtmann}, ``marking'' refers to an ordered set of generators in the fundamental group of a graph, which we shall call a ``framing'' in order to avoid conflict with the other notions.

\subsection{Metric graphs} \label{sect: MetricGraphs} All graphs will be finite, and connected in the following discussion. A metric graph $G$ is one in which every edge $e$ is assigned a length $\ell_e \in \R_{>0}$. The lengths are normalised so that their total sum $\sum_{e\in E_G} \ell_e $ equals $1$. The metrics on $G$ define an open Euclidean simplex of dimension $e_G-1$
\begin{gather*}
\sigma_G = \bigg\{(\ell_e)_e \in \R_{>0}^{E_G}\colon \sum_{e\in E_G} \ell_e=1\bigg\} .
\end{gather*}
 Let $\overline{\sigma}_G \subset \R_{\geq 0}^{E_G}$ denote the closed simplex where all lengths are positive or zero.
 Contraction of an edge $e \in E_G$ corresponds to the natural inclusion
 \begin{gather*}
 \iota\colon\ \sigma_{G/e} \ \Longhookrightarrow \ \overline{\sigma}_G,
 \end{gather*}
 where $\sigma_{G/ e}$ is identified with the open face defined by $\ell_e=0$.
 An edge contraction is called admissible if $e$ has distinct end points and therefore $G/e= G\q e$.

 The group of automorphisms $\operatorname{Aut}(G)$ acts via permutation of the edges and vertices of $G$, and acts by linear transformations on $\sigma_G$, and its closure $\overline{\sigma}_G$.

\subsection{Differential forms} A first definition of a smooth differential form of degree $k$ and genus $g$ is
 the data of a collection $\{\omega_G\}_G$ of differential forms
 \begin{gather*}
 \omega_G\colon\ \text{a smooth $k$-form on $\sigma_G$ for every graph $G$ with } h_G =g,
 \end{gather*}
 which are functorial and compatible with each other: in other words, $\pi^* \omega_{G'} = \omega_{G}$, where by abuse of notation, $\pi\colon \sigma_G \rightarrow \sigma_{G'}$ denotes the linear isomorphism on cells induced by any isomorphism $\pi\colon G \overset{\sim}{\rightarrow} G'$; and
 for every admissible edge contraction of $G$, the form $\omega_G$ extends smoothly to the open face $\iota (\sigma_{G/\!/ e}) \subset \overline{\sigma}_G$ and its restriction satisfies
 \begin{gather*}
 \iota^* \omega_G = \omega_{G/\!/ e} \qquad ( = \omega_{G/e}).
 \end{gather*}
 It is important to note that the forms $\omega_G$ all have the same degree, independent of $G$ or $g$.
 The differential is defined in the usual manner: ${\rm d} \{\omega_G\}_G = \{{\rm d} \omega_G\}_G$; as is the exterior product
 $\{\omega\}_G \wedge \{\eta\}_G = \{\omega \wedge \eta\}_G$.
 This leads to a simple definition of a de Rham complex of smooth forms. We briefly discuss geometric interpretations in the next section before turning to a definition of \emph{algebraic} differential forms.

 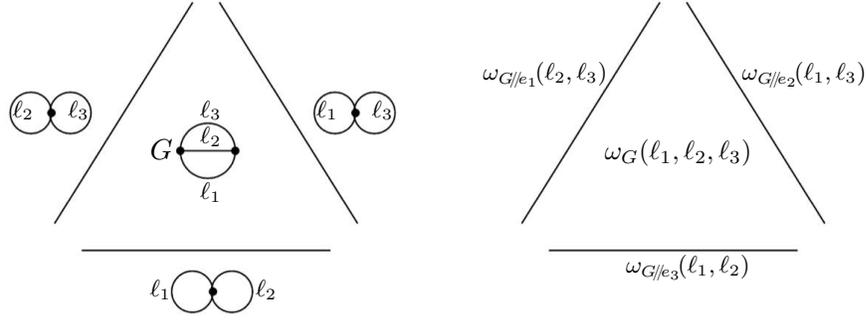
\begin{figure}[h]
\centering
\includegraphics[width=11cm]{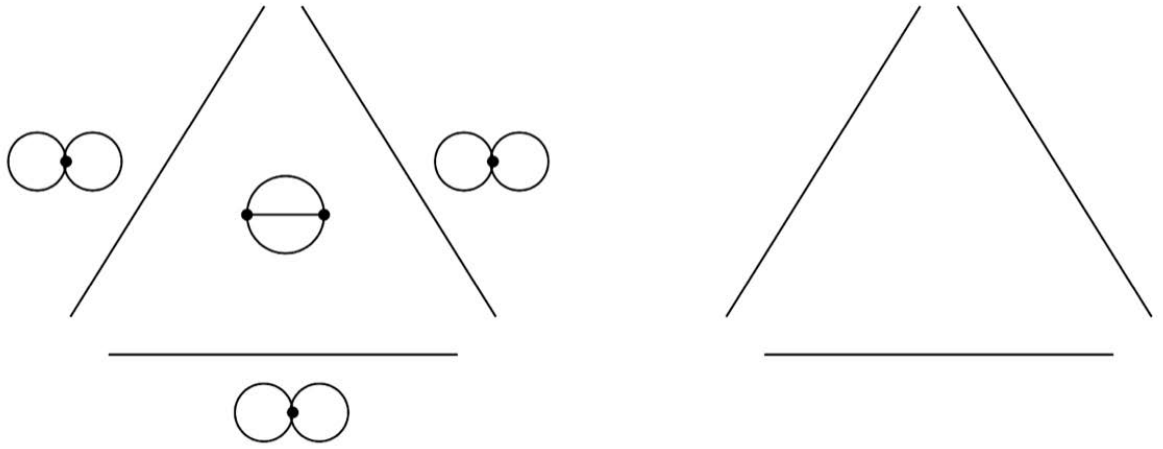}
\put(-258,60){$G$}
\put(-239,45){\footnotesize $\ell_1$}
\put(-239,66){\footnotesize $\ell_2$}
\put(-239,76){\footnotesize $\ell_3$}
\put(-310,76){\footnotesize $\ell_2$}
\put(-289,76){\footnotesize $\ell_3$}
\put(-258,8){\footnotesize $\ell_1$}
\put(-218,8){\footnotesize $\ell_2$}
\put(-195,76){\footnotesize $\ell_1$}
\put(-174,76){\footnotesize $\ell_3$}
\put(-86,60){\small $\omega_G(\ell_1,\ell_2,\ell_3)$}
\put(-132,90){\footnotesize $\omega_{G\! /\!/ \!e_1}\! (\ell_2,\ell_3)$}
\put(-34,90){\footnotesize $\omega_{G\!/\!/\!e_2} \!(\ell_1,\ell_3)$}
\put(-78,17){\footnotesize $\omega_{G\!/\!/\!e_3} \!(\ell_1,\ell_2)$}
\caption{{\it Left}: The cell $\sigma_G$ corresponding to a sunrise graph $G$ with three edges. It is the open simplex $\ell_1+\ell_2+\ell_3=1$ in $\R_{>0}^3$. Each open facet $\ell_i=0$ of its closure is identified with $\sigma_{G \!/\!/\! e_i}$, where~$G\q e_i$ is the graph obtained by contracting the edge $e_i$.
The corners, which arise from contracting loops, are omitted.
{\it Right}: A differential form $\omega_G$ on $\sigma_G$ which extends smoothly to the open facets $\ell_i=0$, restricted to which, $\omega_G$ coincides with $\omega_{G\!/\!/\! e_i}$. The form $\omega_G$ must be invariant under all permutations of $\ell_1$, $\ell_2$, $\ell_3$
since $\operatorname{Aut}(G)$ is isomorphic to the symmetric group on three letters.}
\label{figuresunrisesimplex}
\end{figure}

\subsection{Geometric digression} For the convenience of the interested reader, we relate the rather informal discussion above to the moduli space of tropical curves, and Outer space. The following is not required for the rest of the paper.

\subsubsection{Moduli space of tropical curves} \label{Sect: ModTrop}
A weighted graph $(G,w)$ is a graph $G$ which has a weight function $w\colon V_G \rightarrow \mathbb{Z}_{\geq 0}$ on its set of vertices. A graph with no weightings will usually be regarded as the graph $(G,0)$, where $0$ denotes the zero weight function.
 The genus of a connected weighted graph is
 \begin{gather*}
 g = h_G + \sum_{v\in V_G} w(v).
 \end{gather*}
The cell associated to a weighted, metric graph is the set
\begin{gather*}
C(G,w) = \left\{ ( \ell_e)_e \in \R_{>0}^{E_G} \right\}
\end{gather*}
of all possible edge lengths. It does not depend on $w$.
A tropical curve \cite{BMV} is defined to be a weighted metric graph $(G,w)$ which is stable: in other words the degree (valency) of every vertex of weight zero is $\geq 3$, and every vertex of weight $1$ has degree $\geq 1$.
 The automorphism group $\operatorname{Aut}(G,w)$ is the subgroup of the full group of automorphisms $\operatorname{Aut}(G)$ which preserves the weight function. It acts linearly upon the cell $C(G,w)$ and upon its closure $\overline{C(G,w)} \cong \R^{E_G}_{\geq 0}$.

 A specialisation (contraction) of a tropical curve with respect to an edge $e$ is the tropical curve obtained by contracting $e$. If the edge $e$ has two distinct endpoints of weights $w_1$, $w_2$, then the new vertex obtained after contracting the edge $e$ has weight $w_1+w_2$; if the edge $e$ is a loop (or tadpole) with a single endpoint of weight $w$, then after contraction it leads to a vertex of weight $w+1$. The former contractions were considered admissible in the previous paragraphs; the latter not.

 The moduli space of tropical curves \cite{BMV} of genus $g$ is the topological space
 \begin{gather*}
 M^{\tr}_g = \bigsqcup_{(G,w)} \big( \overline{C(G,w)}/\operatorname{Aut}(G,w)\big)_{/\sim},
 \end{gather*}
 where the disjoint union is over all stable weighted graphs of genus $g$. In this definition, the spaces are endowed with the quotient topology, and $\sim$ is the equivalence relation given by common specialisations of weighted metric graphs.
 Alternatively, one may define $M^{\tr}_g$ as a~colimit \cite{CGP}, by identifying the boundaries of each closed cell $\overline{C(G,w)}$ with the cells of their specialisations.

 The simplices $\sigma_G $ we considered above may be embedded in $C(G,0)$, and may also be identified with $C(G,0)/\R_{>0}$, where $\R_{>0}$ acts by scalar multiplication on the edge lengths. In this manner, consider the open subspace
 \begin{gather*}
 \big(M^{\tr}_g\big)_{w=0} \subset M^{\tr}_g
 \end{gather*}
 defined to be the complement in $M^{\tr}_g$ of the images of all cells $C(G,w)$ (or their closures, it does not matter) which involve a non-trivial weighting function $w\neq 0$, or equivalently, of graphs $(G,w)$ whose total weight $w(G) = \sum_{v\in V_g} w(v) $ is positive.

 A collection of smooth differential forms $\{\omega_G\}$ of degree $k$ and genus $g$
 may thus be interpreted as a differential $k$-form on the quotient of the locus $\big(M^{\tr}_g\big)_{w=0}$ by $\R^{>0}$.

 \subsubsection{Outer space}
 Outer space is constructed from connected metric graphs $(G,\ell)$ which have no vertices of degree $\leq 2$, and which are equipped with a homotopy equivalence from the ``rose'' graph $R_g$ which has one vertex and $g$ edges:
 \begin{gather*}
 f\colon\ R_g \To G,
 \end{gather*}
 where $g= h_G$. Such a map $f$ is called a ``marking'' in \cite{CullerVogtmann}; we shall call it a framing to avoid confusion for the reasons mentioned earlier. The metric $\ell\colon E_G \rightarrow \R_{>0}$ is normalised so that $\sum_{e \in E_G} \ell(e) =1$.
 The map $f$ induces an isomorphism
 \begin{gather*}
 f_*\colon\ \Z^g \cong H_1(R_g;\Z) \ \overset{\sim}{\To} \ H_1(G;\Z)
 \end{gather*}
 and hence defines a basis of the homology group $H_1(G;\Z)$.
 An isomorphism of framed graphs $(G,f) \cong (G',f')$ is an isomorphism $\pi\colon G \overset{\sim}{\rightarrow} G'$ such that $f'$ is homotopy equivalent to $\pi f$. The contraction of an edge in $(G,f)$ is the framed graph $(G/e, f')$, where $f'$ is the composition of $f$ with the quotient $G \rightarrow G/e$. It is admissible if $e$ has distinct endpoints.

 Outer space $\mathcal{O}_g$ is defined \cite{CullerVogtmann} by gluing together simplices $\sigma_{(G,f)}$ along the maps $\iota\colon \sigma_{(G,f)/e} \hookrightarrow \overline{\sigma}_{(G,f)}$ for admissible edge contractions, modulo the action of isomorphisms of framed graphs.
 Therefore the images of open cells in $\mathcal{O}_g$ correspond to isomorphism classes of framed graphs $(G,f)$.
 It is important to note that since only admissible edge contractions are allowed, the closure of an open cell $\sigma_G$ in Outer space is not necessarily compact (not all faces of $\overline{\sigma}_G$ are admitted\footnote{If one does admit all such faces, i.e., uses the closed simplices $\overline{\sigma}_{(G,f)}$ in place of $\sigma_{(G,f)}$, then one obtains the simplicial closure $\mathcal{O}^*_g$, whose quotient $\mathcal{O}^*_g/\mathrm{Out}(F_g)$ is isomorphic to the link of the vertex of $M^{\mathrm{tr}}_{g}$.}).
The group $\mathrm{Out}(F_g)$ of outer automorphisms of the free group on $g$ generators acts properly on the space $\mathcal{O}_g$, and its quotient $\mathcal{O}_g / \mathrm{Out}(F_g)$ is the quotient of $\big(M_g^\mathrm{tr}\big)_{w=0}$ by $\R^{>0}$.

 A collection of smooth differential forms $\{\omega_G\}_G$ of degree $k$ and genus $g$
 may thus be interpreted as an $\mathrm{Out}(F_g)$-invariant differential form on Outer space $\mathcal{O}_g$, or viewed as a form on the quotient of Outer space $\mathcal{O}_g$ by $\mathrm{Out}(F_g)$. These interpretations are not to be taken too literally, since $\mathcal{O}_g$ is not even a manifold.

\subsection{Algebraic differential forms}
In order to provide a connection with the theory of periods and motives, we require a notion of \emph{algebraic} differential forms.
Since neither the moduli space of tropical curves, nor Outer space, is even remotely close to being an algebraic variety, this must be achieved by passing to an algebraic model. In order to do this,
 the first step is to identify the simplex $\sigma_G$ of Section~\ref{sect: MetricGraphs} with the open real coordinate simplex in projective space
 \begin{gather*}
 \sigma_G \subset \Pro^{E_G-1} (\R).
 \end{gather*}
 The coordinates on the projective space will be denoted by $\alpha_e$ for all $e \in E_G$.
 The inclusion of faces $\iota\colon \sigma_{G/ \!/ e} \hookrightarrow \overline{\sigma}_G$ is induced by the inclusion of the coordinate hyperplane $\alpha_e=0$:
 \begin{gather} \label{iotae}
 \iota_e \colon \ \Pro^{E_{G\q e}-1} \To \Pro^{E_G-1},
 \end{gather}
 which is a morphism of algebraic varieties.
 Furthermore, every isomorphism $\pi\colon G \overset{\sim}{\rightarrow} G'$ induces an algebraic isomorphism of projective spaces
 $\pi\colon \Pro^{E_{G}-1} \overset{\sim}{\rightarrow} \Pro^{E_{G'}-1} $ which permutes the set of coordinate hyperplanes $V(\alpha_e)= \{\alpha_e=0\}$ for $e\in E_G$.

 We can then define an \emph{algebraic differential form} of degree $k$ and genus $g$ to be a collection $\{\omega_G\}_G$
 of projectively-invariant meromorphic differential $k$-forms on the
 spaces $\Pro^{E_G-1}$ for all $G$ with $h_G=g$, which are smooth on $\sigma_G$, and which are
 \begin{gather}
({\it compatible})\colon\ \iota_e^* \omega_G = \omega_{G\q e} \text{ for every admissible edge contraction, and } \nonumber
\\
 (\text{\it equivariant})\colon\ \pi^* \omega_{G'} = \omega_{G} \text{ for every isomorphism } \pi\colon G \overset{\sim}{\rightarrow} G'.\label{CompatibleEquivariant}
 \end{gather}
 A projectively-invariant differential form is one which is homogeneous of degree zero and annihilated by contraction with the Euler vector-field.
 A form $\omega_G$ is allowed to have poles anywhere \emph{away} from the open real locus $\sigma_G$.

Now, if $\deg (\omega_G) = e_G-1= \dim \sigma_{G}$, we would like to consider the integral
 \begin{gather*}
 I_G (\omega) = \int_{\sigma_G} \omega_G .
 \end{gather*}
 It makes sense by projectivity of the form $\omega_G$.
However, if the form $\omega_G$ blows up in an uncontrolled manner along the boundary faces of the closure $\overline{\sigma}_G$ (see Figure~\ref{figuresunrisesimplex}) then there is nothing to guarantee that the integral is finite.

\subsection{Tropical Torelli map and invariant forms} In order to construct families of algebraic forms, consider the ``tropical Torelli'' map \cite{Baker, CaporasoViviani, ChanTorelli, Nagnibeda},
 from the moduli space of tropical curves to the moduli space of tropical Abelian varieties:
 \begin{gather} \label{Intro: TropTorelli}
 M_g^{\tr} \To A_g^{\tr}.
 \end{gather}
 It associates, in particular, to a stable metric graph $(G,0)$ with zero weight function the class of a graph Laplacian matrix $\Lambda_G$.
 The space $A_g^{\tr} = \Omega^{\mathrm{rt}}/\GL_g(\Z)$ is the quotient of the space $\Omega^{\mathrm{rt}}$ of positive semi-definite quadratic forms with rational null space by the general linear group. The graph Laplacian matrix $\Lambda_G$ is a positive semi-definite symmetric $g\times g$ matrix whose entries are linear combinations of edge lengths of $G$, and depends on a choice of basis of $H_1(G;\Z)$; nevertheless, its class in $A_g^{\tr}$ is well-defined.

 A basic idea of this paper is to write down differential forms on the space of positive definite symmetric matrices which are left and right invariant under the action of $\GL_g(\Z)$ and pull-them back along the tropical Torelli map \eqref{Intro: TropTorelli}. For all $k\geq 1$, consider the forms
 \begin{gather*}
 \beta_X^{4k+1} = \tr \big(\big(X^{-1} {\rm d}X\big)^{4k+1}\big)
 \end{gather*}
 for any invertible symmetric matrix $X$, which were shown by Borel~\cite{Borel} to generate the stable cohomology of the general linear group.
 Note that since they involve inverting $X$, they are smooth only on the sublocus given by positive definite symmetric matrices, and thus have singularities along $A_g^{\mathrm{tr}}$ at infinity.

 Concretely, then, this means that to any connected graph $G$, we write down a graph Laplacian matrix $\Lambda_G$ and define for all $k\geq 1$,
 \begin{gather} \label{omegacan}
 \omega^{4k+1}_G = \tr \big(\big(\Lambda_G^{-1} {\rm d} \Lambda_G\big)^{4k+1}\big).
 \end{gather}
It does not depend on the choices which go into defining $\Lambda_G$, namely a choice of basis for $H_1(G;\Z)$.
 The determinant $\Psi_G = \det \Lambda_G$ is the Kirchhoff graph polynomial.

\begin{thm} For all $k\geq 1$, the $\omega^{4k+1}_G$ are projective forms on
$\Pro^{E_G-1} \backslash X_G$,
where $X_G = V(\Psi_G)$ is known as the graph hypersurface. They satisfy the compatiblity and equivariance properties~\eqref{CompatibleEquivariant}.
They have the following shape:
\begin{gather} \label{intro: omegashape}
\omega^{4k+1}_G = \frac{N_G}{\Psi^{k+1}_G},
\end{gather}
where $N_G$ is a polynomial in the parameters $\alpha_e$ and their differentials $d
\alpha_e$, with coefficients in~$\Q$. The form $\omega_G$ has a pole along $X_G$ of order at most $k+1$.
\end{thm}

Since the graph polynomial $\Psi_G$ is positive on the simplex $\sigma_G$, the family $\{\omega_G\}_G$ satisfies the conditions required of an algebraic differential Section~\ref{sect: algdiffblowup}, and has many other properties. The statement about the order of the poles is the content of Theorem \ref{cor: omegaprojective} and is the result of many cancellations between numerator and denominator in the definition.

Note that the $\omega_G$ are defined for every $g\geq 1$. A priori they may be viewed, for any such $g$, as differential forms on the quotient of the open set $(M^{\tr}_g)_{w=0}$
 of $0$-weighted graphs by $\R^{>0}$ via Section~\ref{Sect: ModTrop}, but in some cases they extend to a strictly larger locus inside $M^{\tr}_g$.

\subsection{Canonical algebra of differential forms}
We define the \emph{canonical algebra} of differential forms to be the exterior algebra on the forms \eqref{omegacan}
\begin{gather*} 
\Omega^{\bullet}_{\can} = \bigwedge \bigg( \bigoplus_{k\geq 1} \Q \omega^{4k+1}\bigg) .
\end{gather*}
It is a graded Hopf algebra for the coproduct $\Delta_{\can}\big(\omega^{4k+1}\big) = \omega^{4k+1} \otimes 1 + 1 \otimes \omega^{4k+1}$ with respect to which the
 generators $\omega^{4k+1}$ are primitive.
 Given any form $\omega \in \Omega_{\can}^k$ of degree $k$, which we call a \emph{canonical form}, we obtain an integral
\begin{gather} \label{IGdefn} I_G(\omega) = \int_{\sigma_G} \omega_G \end{gather}
for every graph $G$ with $k+1$ edges. One of our main results (Theorem~\ref{thm: nopolesonD}) implies

\begin{thm} The integral $I_G(\omega)$ is always finite.
\end{thm}

From the particular shape of the integrand \eqref{intro: omegashape}, one deduces that the integral $I_G(\omega)$ is what is known as a generalised Feynman integral (or ``Feynman period'') in quantum field theory. The previous theorem is
 in stark contrast with the usual situation for Feynman integrals, which are often highly divergent.

\begin{Example} Let $G= W_3$ be the wheel with three spokes, and let $\omega^5$ be the first non-trivial canonical form \eqref{omegacan}. Then
\begin{gather*}
I_{W_3} \big(\omega^5 \big) = 60 \zeta(3)
\end{gather*}
in accordance with Remark~\ref{remPuzzlezeta(3)}. Further examples are given in Section~\ref{sect: Examples}.
\end{Example}

The integrals \eqref{IGdefn} only depend on the isomorphism class of $G$ in the graph complex $\GC_N$. From this we deduce a pairing between the component of edge-degree $k$ and the space of canonical forms of degree $k-1$:
\begin{gather*} 
I \colon\ (\GC_N)_{k} \otimes_{\Q} \Omega^{k-1}_{\can} \To \C.
\end{gather*}
 This pairing can in principle be used to prove the non-vanishing of homology classes.

 \subsection{Bordification and blow-up} \label{sect: algdiffblowup} In order to prove the convergence of the integrals $I_G(\omega)$
 one can construct an algebraic compactification of the space of metric graphs, and use it to study the behaviour of the forms $\omega_G$ at infinity.
 This can be done by repeatedly blowing up intersections of coordinate hyperplanes $L_{\gamma}= V(\{\alpha_e, e \in E(\gamma)\})$ in projective space in increasing order of dimension, where $\gamma$ ranges over a specific family $\mathcal{B}_G$ of subgraphs of $G$. This leads to a projective algebraic variety
 \begin{gather} \label{PGblowup}
 \pi_G\colon\ P^G \To \Pro^{E_G-1}.
 \end{gather}
 One way to do this is to perform blow-ups corresponding to all core\footnote{A core graph, also called 1-particle irreducible, is one whose loop number decreases on cutting any edge, or equivalently, which has no bridges.} subgraphs $\mathcal{B}_G= \mathcal{B}^{\mathrm{core}}_G$~\cite{BEK}, another is to simply to blow up subspaces corresponding to all subgraphs. The required conditions on $\mathcal{B}_G$ are spelled out in \cite[Section~5.1]{Cosmic}. In either case,
 the exceptional divisor corresponding to a subgraph $\gamma \in \mathcal{B}_G$ is canonically isomorphic to a product $P^{\gamma}\times P^{G/\gamma}$ , and gives rise to a ``face map'' \begin{gather} \label{introiotagamma}
 \iota_{\gamma}\colon \ P^{\gamma} \times P^{G/\gamma} \To P^G.
 \end{gather}
 Note that the map $\iota_e\colon P^{G\q e} \rightarrow P^{G}$ coming from \eqref{iotae} may also be written in the form \eqref{introiotagamma} in the case when $\gamma =e$ is a single edge (with distinct endpoints), since $P^e= \mathop{\rm Spec} \Q $ is a point. Another interesting case is when $\gamma = G\backslash e$, for then $P^{G/\gamma}$ also reduces to a point.
 In general, the face maps \eqref{introiotagamma} provide extra structure which relate metric graphs of different genera.

The closure $\widetilde{\sigma}_G$ of the inverse image $\pi_G^{-1}(\sigma_G)$ inside $P^G(\R)$ defines a compact polytope with corners (or ``Feynman polytope''), which is essentially the basic building block of the bordification of Outer space constructed in \cite{BordificationOuterSpace}. Via \eqref{introiotagamma} its faces are isomorphic to products of $\widetilde{\sigma}_{\gamma}$, where~$\gamma$ are minors of $G$. See Figure \ref{figuresunriseblowup} for an illustration.

 Now consider the pull-backs of canonical forms
 \begin{gather*}
 \widetilde{\omega}_G=\pi_G^* \omega_G .
 \end{gather*}
 They are meromorphic differential forms on $P^G$
which may \emph{a priori} have poles along exceptional divisors. However, in Theorem~\ref{thm: nopolesonD} we show that this is not so: any primitive form satisfies \begin{gather} \label{introiotastargamma} \iota^*_{\gamma} \widetilde{\omega}^{4k+1}_G = \widetilde{\omega}^{4k+1}_{\gamma} \wedge 1 + 1 \wedge \widetilde{\omega}^{4k+1}_{G/\gamma}.
 \end{gather}
 The corresponding formula for general $\omega \in \Omega_{\can}$
 is obtained by taking exterior products and is expressible using the coalgebra structure on $\Omega_{\can}$.

Formula \eqref{introiotastargamma} implies that $\widetilde{\omega}$ has no poles on the compactification $\widetilde{\sigma}_G$ of the simplex $\sigma_G$, and therefore that the following integral is finite
 \begin{gather*}
 I_G (\omega) = \int_{\sigma_G} \omega_G = \int_{\widetilde{\sigma}_G} \widetilde{\omega}_G < \infty ,
 \end{gather*}
where $G$ is any connected graph such that $e_G = \deg(\widetilde{\omega})+1$.

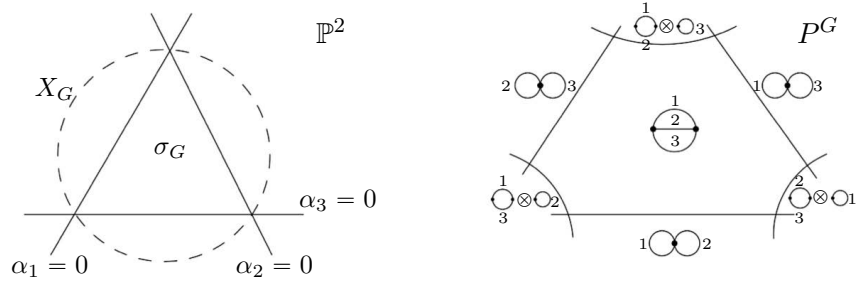
\begin{figure}[h]
\centering
\includegraphics[width=11cm]{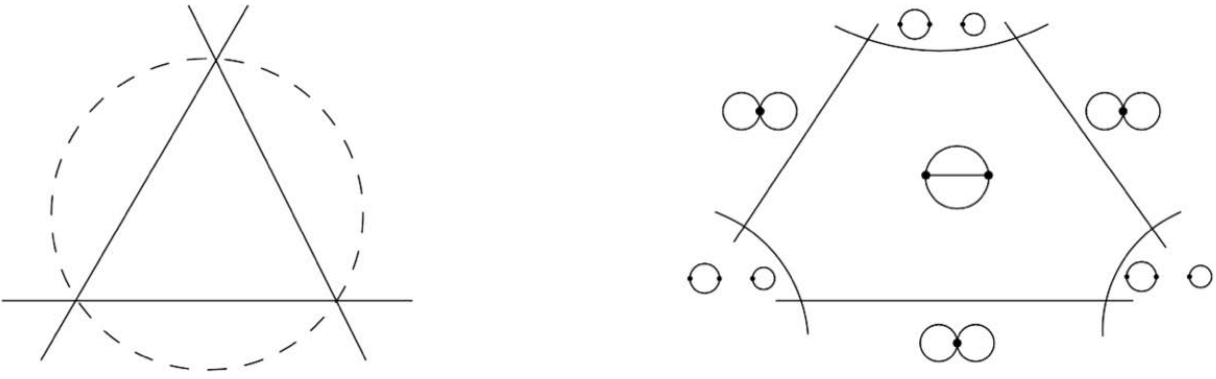}
\put(-202,85){$\Pro^2$}
\put(-20,85){$P^G$}
\put(-312,65){ \small $X_G$}
\put(-267,43){ \small $\sigma_G$}
\put(-212,23){ \small $\alpha_3=0$}
\put(-236,-2){ \small $\alpha_2=0$}
\put(-321,-2){ \small $\alpha_1=0$}
\put(-126,24){\tiny $\otimes$}
\put(-113.5,23){\tiny $2$}
\put(-133,31){\tiny $1$}
\put(-133,17){\tiny $3$}
\put(-14,25){\tiny $\otimes$}
\put(-21,17){\tiny $3$}
\put(-1,24){\tiny $1$}
\put(-21,31){\tiny $2$}
\put(-72.25,90){\tiny $\otimes$}
\put(-79,96.5){\tiny $1$}
\put(-79,82.5){\tiny $2$}
\put(-58.5,88.5){\tiny $3$}
\put(-80,7){\tiny $1$}
\put(-56,7){\tiny $2$}
\put(-132,67){\tiny $2$}
\put(-107,67){\tiny $3$}
\put(-37,67){\tiny $1$}
\put(-14,67){\tiny $3$}
\put(-68,61){\tiny $1$}
\put(-68,54){\tiny $2$}
\put(-68,46){\tiny $3$}
\caption{{\it Left}: The cell $\sigma_G$ for the sunrise graph may be identified with the open coordinate simplex $\{(\alpha_1:\alpha_2:\alpha_3)\colon \alpha_i>0\}$ in projective space $\Pro^2$. The dotted circle indicates the graph hypersurface~$X_G$, which meets its corners. {\it Right}: After blowing up the three corners $\alpha_1=\alpha_2=0$, $\alpha_1=\alpha_3=0$ and $\alpha_2=\alpha_3=0$, we obtain a space $P^G\rightarrow \Pro^2$, in which the total inverse image of the coordinate hyperplanes form a hexagon (the strict transform of the graph hypersurface $X_G= V(\Psi_G)$ is not shown). The exceptional divisors are isomorphic to products $P^{\gamma} \times P^{G / \gamma}$ corresponding to a subgraph $\gamma$ and the quotient $ G/ \gamma$. }
\label{figuresunriseblowup}
\end{figure}

\subsection{Stokes' formula}
 Equation \eqref{introiotastargamma} is an extra property of canonical forms ``at infinity'' over and above the compatibility and equivariance properties \eqref{CompatibleEquivariant}. It can be exploited to prove relations between canonical integrals for graphs with different loop numbers.
 For a canonical form $\omega \in \Omega^k_{\can}$ of degree $k$, write its coproduct in Sweedler notation:
\begin{gather*}
\Delta_{\can} \omega = \omega\otimes 1 + 1 \otimes \omega + \sum_i \omega'_i \otimes \omega''_i.
\end{gather*}
Then we prove that
\begin{gather} \label{IntroStokes}
0 = \underbrace{ \sum_{e\in E_G} \int_{\sigma_{G/e}} \omega_{G/e} }_{\rm d} + \underbrace{ \sum_{e\in E_G} \int_{\sigma_{G\backslash e}} \omega_{G\backslash e} }_{\delta} + \underbrace{\sum_{\gamma\subset G} \sum_i \int_{\sigma_{\gamma}} (\omega'_i)_{\gamma} \int_{\sigma_{G/\gamma}} (\omega''_i)_{G/\gamma} }_{\Delta'},
\end{gather}
where the sum is over core subgraphs $\gamma \subset G$ such that $\deg \omega'_i = e_{\gamma}-1$. The terms in the formula~\eqref{IntroStokes} reflect the structure of the boundary faces of the polytope $\widetilde{\sigma}_G$.
After taking into account the orientations on graphs which are consistent with the orientations of simplices $\sigma_G$, the three braced terms in this expression can be interpreted as: the differential in the graph complex~$d$; the differential \eqref{deltadef}; and the reduced version of the Connes--Kreimer coproduct~\eqref{DeltaCK}.

Thus, by extending the notation $I$ appropriately, we may rewrite \eqref{IntroStokes} equivalently as
\begin{gather*}
 0 = I_{{\rm d}G}(\omega) + I_{\delta G} (\omega) + I_{\Delta' G} (\Delta_{\can}'(\omega)),
 \end{gather*}
where $\Delta_{\can}'=\Delta_{\can} - 1 \otimes \id - \id \otimes 1$ is the reduced coproduct associated to $\Delta_{\can}$.

\begin{rem}The formula \eqref{IntroStokes} allows one in principle to detect homology classes.
A simple example is given in Corollary~\ref{cor: IGnonzero}, which states that the conjectural non-vanishing of the canonical integrals associated to wheels $W_{2n+1}$ gives another proof of the fact that the classes $[W_{2n+1}]$ are non-zero in $H_0(\GC_2)$.
Another situation in which non-vanishing of a canonical integral implies non-vanishing of a homology class is given in Corollary~\ref{cor: detectclassv2}.
\end{rem}

\subsection{Relation to motivic periods}
The integrals considered above may be lifted to ``motivic'' periods.
Concretely, define for any $\omega \in \Omega^{k}_{\can}$ and any graph $G$ with $k+1$ edges, a motivic period, defined by an equivalence class
\begin{gather*}
I^{\mm}_G(\omega) = \big[ \mathrm{mot}_G, \big[\widetilde{\sigma}_G\big], \big[\widetilde{\omega}_G\big]\big]^{\mm},
\end{gather*}
where $\mathrm{mot}_G$ is a relative cohomology ``motive'' of $G$, which is defined using the geometry of the blow up \eqref{PGblowup}, and $\widetilde{\omega}_G = \pi_G^* \omega_G$. Applying the period homomorphism allows one to recover the integral \eqref{IGdefn}, $I_G(\omega)= \operatorname{per} I^{\mm}_G(\omega)$.
We show that the formula \eqref{IntroStokes} is motivic, i.e., holds for the objects $I_G^{\mm}(\omega)$. In this manner, one can assign motivic periods to graphs, which provides a~connection between the homology of the graph complex and motivic Galois groups.

\subsection{A conjecture for graph cohomology}
The calculations of Section~\ref{sect: Examples} lead us to expect, for every increasing sequence of integers
\begin{gather*}
1\leq k_1< k_2 < \dots < k_r
\end{gather*}
 the existence of an element $X \in \GC_N$ satisfying ${\rm d}X=\delta X= 0 $ such that
\begin{gather*}
I_X\big(\omega^{4k_1+1} \wedge \dots \wedge \omega^{4k_r+1} \big) = \prod_{i=1}^r \zeta(2k_i+1) .
\end{gather*}
A similar statement should hold for motivic periods.
By the types of argument outlined above, this suggests the existence of (at least one) non-trivial graph homology class which pairs non-trivially with every canonical form, and whose canonical integral is a product of odd zeta values.
Dually, this suggests the existence of a non-canonical injective map from $\Omega^{\bullet}_{\can}$ into the cohomology of the graph complex.
Since graph cohomology is a Lie algebra one is led to the following conjecture.

\begin{conj} \label{conj: Mainconj} There is a non-canonical injective map of graded Lie algebras from the free Lie algebra on $\Omega^{\bullet}_{\can}$ into graph cohomology:
\begin{gather} \label{FreeOmegatoGraphCohom} \mathbb{L} \left( \Omega^{\bullet}_{\can} \right) \To \bigoplus_{n\in \Z} H^{n} (\GC_2)
\end{gather}
such that its restriction to the Lie subalgebra generated by primitive elements
maps to the Lie subalgebra of cohomology in degre zero:
\begin{gather} \label{FreePrimOmega}
\mathbb{L} \bigg(\bigoplus_{k\geq 1}\omega^{4k+1} \Q \bigg) \To H^{0} (\GC_2).
\end{gather}
All other elements map to higher degree cohomology $\bigoplus_{n>0} H^{n} (\GC_2)$.
Furthermore, we expect that the exterior product of $m$ primitive forms $\omega^{4k+1}$ occurs in even cohomological degree if $m$ is odd, and odd cohomological degree if $m$ is even.

The grading on the left-hand side of \eqref{FreeOmegatoGraphCohom} is by the degree of differential forms; on the right, it is by edge number only, so in fact the conjecture \eqref{FreeOmegatoGraphCohom} is more naturally expressed using $\GC_0$ rather than $\GC_2$.
\end{conj}

Information about the loop number (or equivalently, about the cohomological grading, if one rephrases the conjecture in terms of the cohomology of $\GC_N$ for some $N\neq 0$) is mostly lost in this conjecture. It is possible that some of the information can be recovered by replacing these gradings with a suitable filtration. Indeed, vanishing properties such as Proposition~\ref{prop: betavanishes} places some mild additional constraints on the loop order where canonical forms could occur in the cohomology of the graph complex, which we omitted for simplicity.

\begin{rem} The previous conjecture is slightly artificial because the natural integration pairing \eqref{IGdefn} gives rise to irrational numbers and is thus not defined over $\Q$, and because a canonical form $\omega$ could conceivably pair with several closed elements $X \in \GC_2$ representing independent graph homology classes, and giving distinct periods. Indeed, we do not expect there to be a~canonical candidate for a map \eqref{FreeOmegatoGraphCohom} since its restriction \eqref{FreePrimOmega} would give rise to an injection~\eqref{MTZinjection} of the free Lie algebra on generators of every odd degree into the motivic Lie algebra, which is \emph{a priori} not canonical (it depends on a choice of basis of motivic multiple zeta values).
\end{rem}

In order to help with the visualisation of the conjecture, or rather its equivalent formulation for $\GC_2$, Table \ref{tableForms} depicts the possible location of classes in low degrees. The table was generated using the examples of Section~\ref{sect: Examples}, the argument of Section~\ref{sect: precip}, and known results about graph cohomology.

Note that the Lie algebra $\mathbb{L}(\Omega^\bullet_{\can})$ carries extra structures not obviously apparent on graph cohomology: for example, the map $\big[\omega^{4k_1+1}, \omega^{4k_2+1}\big] \mapsto \omega^{4k_1+1} \wedge \omega^{4k_2+1}$ and its generalisations appear to be related to the differential in the spectral sequence of~\cite{SpectralSequenceGC2}.

\begin{table}[htp]
\caption{A mostly conjectural picture to illustrate the alignment between conjecture~\eqref{FreeOmegatoGraphCohom} and the known dimensions for graph cohomology groups. It is consistent with computations \cite{WZEulerCharacteristic} for the Euler characteristics of the graph complex.
}
\label{tableForms}
\begin{center}\renewcommand{\arraystretch}{1.1}
\begin{tabular}{c|ccccccccccc}
$H^9$ &&&&&&&&& & $\0$ \\
$H^8$ &&&&&&&&& $\0$ & 0 \\
$H^7$ &&&&&&&& $\0$& $\substack{\omega^9\wedge \omega^{17}}$ & 0 \\
$H^6$ &&&&&&& $\0$ & 0& 0 & $\substack{\omega^5\wedge \omega^{9} \wedge \omega^{13}}$ \\
$H^5$ &&&&&& $\0$ & 0 & 0& 0 & 0 \\
$H^4$ &&&&& $\0$ & 0 & 0& 0 &0 & 0\\
$H^3$ &&&& $\0$ & $ \substack{\omega^5\wedge\omega^{9}}$ & 0& $ \substack{\omega^5\wedge\omega^{13}}$ & $\substack{[\omega^5\!, \, \omega^5\wedge \omega^9]} $ & $ \substack{ \omega^5 \wedge \omega^{17} \\ \omega^9 \wedge \omega^{13}} $ & $\substack{[\omega^5 \! ,\, \omega^5 \wedge \omega^{13}] \\ [ \omega^9\!, \, \omega^5 \wedge \omega^9]}$ \\
$H^2$ &&& $\0$ & 0&0 & 0 & 0 & 0 & 0& 0 \\
$H^1$ && $\0$ & 0 & 0 & 0 & 0 & 0 & 0 & 0 & 0 \\
$H^0$ & $\0$ & $\,\,\, \omega^5 \,\,\, $ & 0& $\,\,\, \omega^9 \,\,\, $ &0 & $\,\, \omega^{13} \,\, $ & $\substack{[\omega^5\!, \, \omega^9]}$ & $\omega^{17}$ & $\substack{[\omega^{5}\!, \,\omega^{13}]}$ & $\substack{[\omega^5 \!,\, [ \omega^5\! ,\, \omega^{9}]] \\ \omega^{21} \vspace{.5ex}}$ \\
\hline
$h_G$ & $2$ & $3$ & $4$ & $5$ & $6$ & $7$ & $8$& $9$& $10$& $11$
\end{tabular}
\end{center}
\end{table}

\subsection{Questions}

An obvious question is whether \eqref{FreeOmegatoGraphCohom}
 is an isomorphism. This is probably false since $H(\GC_N)$ is expected to be too large.
There exists a formula for the Euler characteristic of the graph complex \cite{WZEulerCharacteristic} but its asymptotics are unknown to our knowledge. However,
M. Borinsky has recently informed us of a more compact formula \cite{BorinskyEuler} for the Euler characterstic which strongly suggests super-exponential growth. This was anticipated in \cite[Section~7.2]{Kontsevich93} based on virtual Euler characteristic computations (see also \cite{ EulerOutFn,GetzlerKapranov}). Since the free Lie algebra $\mathbb{L}\left(\Omega^{\bullet}_{\can} \right)$ grows exponentially with respect to the degree, the cokernel of any map of the form \eqref{FreeOmegatoGraphCohom} will be huge.

One explanation for this fact could be the possible existence of more general families of differential forms $\{\omega_G\}_G$ of genus $g$ which lie outside the canonical algebra $\Omega_{\can}$. A possible source might be unstable classes in the cohomology of the general linear group $\GL_g(\Z)$ which are not expressible using invariant forms $\beta_X^{4k+1}$. Another possible explanation is that the canonical forms $\omega \in \Omega^{k}_{\can}$ could pair non-trivially with several different graph homology classes. Some possible evidence in this direction is the fact that the classes of graph hypersurfaces in the Grothendieck ring are of general type \cite{BelkaleBrosnan}. One knows, furthermore, that modular motives can arise in the middle cohomology degree \cite{BrownDoryn, BrownSchnetz}, which is the case of relevance here. In such cases, the Feynman residues are related to modular forms and are conjecturally not multiple zeta values. By contrast, all presently known examples of canonical integrals (see Section~\ref{sect: Examples}) are multiple zeta values, so it would be very interesting to know if canonical integrals differ or not from Feynman residues in this regard. Section \ref{sect: QuestionGalois} discusses the possible relations between Feynman residues, canonical integrals, and motivic Galois groups.

Although our constructions provide a connection between graph homology and motivic Galois groups, it is not yet clear whether one can deduce a natural geometric action of the motivic Galois group $G^{\mathrm{mot}}_{\mathcal{MT}(\Z)
}$ on $H^0(\GC_2)$ as \eqref{WillwacherGRT} and \eqref{MTZinjection} might suggest. The wheel graphs may be a first step in this direction, since computations suggest their canonical motivic integrals are proportional to motivic odd zeta values, which are dual to the generators $\sigma_{2n+1}$ of the motivic Lie algebra.

Finally, many of the constructions in this paper are valid more generally for certain classes of regular matroids, which warrants further investigation. Indeed, linear combinations of matroids whose edge contractions are graphs may provide a possible source, and explanation for, non-trivial homology classes in $\GC_2$.

 \subsection{Related work} We draw the reader's attention to the recent work of Berghoff and Kreimer \cite{BerghoffKreimer} in which they study properties of Feynman differential forms with respect to combinatorial operations on Outer space. A key difference with the present paper is the fact that the forms they consider have different degrees on the image of each cell. Nevertheless, it raises the interesting possibility of constructing forms (in the sense defined here) on moduli spaces of graphs with external legs whose denominator involves both the first and second Symanzik polynomials.

In a different direction, Kontsevich has suggested a possible relationship between the homology of the graph complex with a ``derived'' Grothendieck--Teichm\"uller Lie algebra \cite{KontsevichBourbaki} defined from the moduli spaces $\mathcal{M}_{0,n}$ of curves of genus 0, but we do not know how it relates to the constructions in this paper.
The work of Alm \cite{Alm} is possibly also related, in which he introduces ``Stokes relations'' between multiple zeta values expressed as integrals over $\mathcal{M}_{0,n}$.
\section{Graph polynomial and Laplacian matrix}
We recall the definition of the graph polynomial and its relation to various definitions of Laplacian and incidence matrices. We also discuss a generalisation to matroids.

\subsection{Graph polynomial}
Let $G$ be a connected graph with $h_G$ loops.
Choose an orientation of every edge of $G$. The definitions to follow will ultimately not depend on this, or any other choices. There is an exact sequence
\begin{gather}\label{ShortExactHG} 0 \To H_1(G;\Z) \overset{\mathcal{H}_G}{\To} \Z^{E_G} \overset{\partial}{\To} \Z^{V_G}\To \Z \To 0,
\end{gather}
where the boundary map $\partial$ satisfies
$\partial(e) = t_e - s_e$
for any oriented edge $e$ whose source is $s_e \in V_G$ and whose target is $t_e \in V_G$. Denote the second map in \eqref{ShortExactHG} by
\begin{gather*}
\mathcal{H}_G \in \operatorname{Hom}\big(H_1(G;\Z), \Z^{E_G}\big) .
\end{gather*}

\begin{defn} Assign to every edge $e$ in $G$ a variable $x_e$, and let $\Z[x_e]$ denote the polynomial ring in the variables $x_e$, for $e\in E_G$.

Define a symmetric bilinear form on the space of edges
\begin{align*}
\Z^{E_G} \times \Z^{E_G} & \To \Z[x_e],
\\
\langle e, e'\rangle &= \ \ \delta_{e,e'} x_e,
\end{align*}
where $\delta_{e,e'}$ denotes the Kronecker delta function.
Via the map $\mathcal{H}_G$ it induces a quadratic form on $H_1(G;\Z)$, which can in turn be expressed as a linear map between $H_1(G;\Z)$ and its dual. Therefore
let us denote by
\begin{gather*}
D_G\colon\ \Z^{E_G} \To \operatorname{Hom}\big(\Z^{E_G}, \Z[x_e]\big)
\end{gather*}
the linear map which satisfies $D_G (e) = x_e e^{\!\vee}$, for all $e\in E_G$, where $\{e^{\!\vee}\}$ denotes the dual basis to $E_G$.
Composing with $\mathcal{H}_G$ defines a linear map:
\begin{gather*}
\Lambda_G = \mathcal{H}^T_G D_G \mathcal{H}_G \colon \ H_1(G; \Z) \To \operatorname{Hom}\big(H_1(G;\Z),\Z[x_e]\big) .
\end{gather*}
 \end{defn}
 The determinant of a bilinear form over the integers is an intrinsic invariant, since, in any representation as a symmetric matrix with respect to an integer basis, changing the basis multiplies the determinant by an element in $(\Z^{\times})^2=1$.
 \begin{defn}
 Define the \emph{graph polynomial} to be
 \begin{gather*}
 \Psi_G = \det \Lambda_G \in \Z[x_e].
 \end{gather*}
 \end{defn}
 The graph polynomial is also known as the first Symanzik polynomial, and was first discovered by Kirchhoff. It plays a central role in quantum field theory, and its combinatorial properties have been studied intensively. We shall argue that one should equally study combinatorial properties of the whole graph Laplacian matrix, and its invariant differentials, defined in the next section.

 \begin{thm}[{dual matrix tree theorem}] \label{thm: MatrixTree}
 The graph polynomial is equal to
 \begin{gather*}
 \Psi_G = \sum_{T \subset G} \prod_{e\notin T} x_e,
 \end{gather*}
 where the sum is over all spanning trees $T\subset G.$ Since a non-empty connected graph has a~spanning tree, it follows that $\Psi_G \neq 0$.
 \end{thm}

 If $G$ is not connected but has connected components $G_1,\dots, G_n$, then $\Lambda_G$ is the direct sum of the $\Lambda_{G_i}$ and one has $\Psi_G = \prod_{i=1}^n \Psi_{G_i}$.

\begin{Example} \label{example: W3}
If one chooses a basis of $H_1(G;\Z)$ consisting of cycles $c_1,\dots, c_h$ and if the edges of $G$ are labelled $1,\dots, N$, then $\mathcal{H}_G$ is represented by the \emph{edge-cycle incidence matrix} of $G$: the entry $(\mathcal{H}_G)_{e,c}$ corresponding to an edge $e$ and cycle $c$ is the number of times (counted with orientations) that $e$ appears in $c$.

Let $G$ be the wheel with 3 spokes, with inner edges oriented outwards from the center and outer edges oriented counter-clockwise. A basis for homology is given by the cycles consisting of edges $\{1,5,6\}$, $\{2,4,6\}$, $\{3,5,4\}$:
 \begin{figure}[h]
 \centering
\quad {\includegraphics[width=3cm]{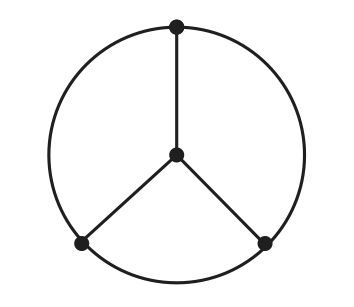}}
\put(-80,50){$1$}\put(-12,50){$3$}\put(-50,-5){$2$}
\put(-50,45){$5$}\put(-62,27){$6$}
\put(-30,27){$4$}
\end{figure}

\noindent
With respect to these bases,
\begin{gather*}
\mathcal{H}^T_G =\begin{pmatrix} 1 & 0 & 0 & 0 & 1 & -1 \\
0 & 1 & 0 & -1 & 0 & 1\\
0 & 0 & 1 & 1 & -1 & 0
\end{pmatrix}.
\end{gather*}
Therefore the graph Laplacian is respresented by the $3\times 3$ matrix
\begin{gather*}
\Lambda_G= \mathcal{H}^T_G D_G \mathcal{H}_G =
 \begin{pmatrix}
x_1 +x_5 +x_6 & - x_6 & -x_5 \\
-x_6 & x_2+ x_4+x_6 & - x_4 \\
-x_5 & -x_4 & x_3+x_4+x_5
\end{pmatrix}.
\end{gather*}
Its determinant is
\begin{align*}
\Psi_G = {}&x_{1} x_{2} x_{3}+ x_1 x_2 x_{4}+x_1x_2 x_{5} +x_1 x_{3}x_{4}+x_1x_{3}x_{6}+x_1x_{4}x_{5}+x_1x_{4}x_{6}+x_1x_{5}x_{6}
\\
& +x_{2} x_{3}x_{5}+x_2 x_{3}x_{6}+x_2 x_{4}x_{5}+x_2 x_{4}x_{6}+x_2 x_{5}x_{6} +x_{3}x_{4}x_{5}+x_{3}x_{4}x_{6}+x_{3}x_{5}x_{6}.
\end{align*}
\end{Example}

\subsection{Dual Laplacian} It is more common to express the graph polynomial using the incidence matrix between edges and vertices as opposed to between cycles and edges. The exact sequence \eqref{ShortExactHG} gives rise to a~sequence
\begin{gather} \label{SEHGv2}
0 \To H_1(G;\Z) \To \Z^{E}\overset{\partial}{\To} \mathrm{Im}(\partial) \To 0.
\end{gather}
The inverse bilinear form $D_G^{-1}$ on $\big(\Z^E\big)^{\vee}=\operatorname{Hom}\big(\Z^E,\Z\big)$ \big(taking values in $\Z\big[x_e^{-1}\big]$\big) restricts to a~bilinear form on the dual $\mathrm{Im}(\partial)^{\vee}= \operatorname{Hom} (\mathrm{Im}(\partial), \Z)$ which we denote by
\begin{gather*} 
L_G = \partial D_G^{-1} \partial^{T} \in \operatorname{Hom} \big( \mathrm{Im}(\partial)^{\vee}, \mathrm{Im}(\partial) \otimes_{\Z} \Z\big[x_e^{-1}\big]\big) .
\end{gather*}
The determinant $\det(L_G)$ is well-defined and is related to the graph polynomial by Lemma~\ref{lem: DGdirectsum} below.
It is usual in the literature to compute $L_G$ as follows.
Since the map $\Z^{V_G} \rightarrow \Z$ in \eqref{ShortExactHG} is given by the sum of all components, the choice of any vertex $w \in V_G$ defines a splitting $ \Z \rightarrow \Z^{V_G}$ by sending $1$ to the element $(0,\dots, 0,1,0,\dots ,0)$, where the non-zero entry lies in the component indexed by $w$.
Set $V'_G = V_G \backslash \{w\}$ and hence $\Z^{V_G}= \Z^{V'_G} \oplus \Z$.
Since $\mathrm{Im}(\partial) \subset \Z^{V_G}$ is given by the subspace of vectors whose coordinates sum to zero,
 the projection $ \Z^{V_G}\rightarrow \Z^{V'_G}$ induces an isomorphism
\begin{gather*}
 \mathrm{Im}(\partial) \cong \Z^{V'_G},
 \end{gather*}
and hence \eqref{ShortExactHG} can be expressed as a short exact sequence
\begin{gather} \label{reducedShortExact}
0 \To H_1(G; \Z) \To \Z^{E_G} \overset{\varepsilon_G}{\To} \Z^{V'_G}\To 0,
\end{gather}
where $\varepsilon_G$ is the composition of $\partial$ with the projection $ \Z^{V_G}\rightarrow \Z^{V'_G}$. With respect to the natural bases, $\varepsilon_G$ can be represented by the $\big(V'_G\times E_G\big)$ matrix
\begin{gather*}
(\varepsilon_G)_{v,e} = \begin{cases}
1 &\text{if} \ v=t(e), \\
-1 &\text{if} \ v=s(e), \\
0 &\text{otherwise},
 \end{cases}
 \end{gather*}
where $s(e)$, $t(e)$ denote the source and targets of $e$. This is nothing other than the edge-vertex incidence matrix of $G$ in which the row corresponding to the vertex $w$ has been removed. Thus~$L_G$ is represented by the matrix
 \begin{gather} \label{LGdefn}
 L_G = \varepsilon_G D^{-1}_G \varepsilon_G^T .
 \end{gather}

 \begin{lem} \label{lem: DGdirectsum}
 There is a unique splitting of \eqref{reducedShortExact} over the field $\Q(x_e, e\in E_G)$, which is orthogonal with respect to the bilinear form $D_G$. There is a basis which is adapted to this splitting in which the matrix $D_G$ is equal to
\begin{gather*} 
D_G = \left(
\begin{array}{c|c}
 \Lambda_G & 0 \\ \hline
 0 & L_G^{-1^{\vphantom{1}}} \\
\end{array}
\right)\!.
\end{gather*}
It follows that $\det(\Lambda_G) \det(L_G)^{-1} = \prod_{e \in E_G} x_e$ and hence
\begin{gather*}
\Psi_G = \det(L_G) \prod_{e \in E_G} x_e.
\end{gather*}
\end{lem}

\begin{proof} Let $K=\Q(x_e,e\in E_G)$.
Consider the short exact sequence:
\begin{gather*}
0 \To H_1(G; K) \overset{\mathcal{H}_G}{\To} K^{E_G} \overset{\varepsilon_G}{\To} K^{V'_G} \To 0.
\end{gather*}
Let $f_G\colon K^{V'_G} \rightarrow K^{E_G}$ denote the unique splitting whose image is orthogonal to $H_1(G;K)$. In~other words, $\varepsilon_G f_G$ is the identity map on $K^{V'_G}$ and the decomposition
\begin{gather} \label{inproof:orthogdecomp}
(\mathcal{H}_G, f_G) \colon\ H_1(G; K) \oplus K^{V_G'} \overset{\sim}{\To} K^{E_G}
\end{gather}
is orthogonal with respect to $D_G$.
The isomorphism $D_G\colon K^{E_G} \cong \big(K^{E_G}\big)^{\vee}$
can be represented, via \eqref{inproof:orthogdecomp}, as a block diagonal matrix of the following form:
\begin{gather*}
 D_G = \left(\begin{array}{c|c}
 \mathcal{H}_G^T D_G \mathcal{H}_G & 0 \\ \hline
 0 & f_G^{T^{\vphantom{1}}} D_G f_G \\
\end{array}
\right) = \left(
\begin{array}{c|c}
 \Lambda_G & 0 \\ \hline
 0 & f_G^{T^{\vphantom{1}}} D_G f_G \\
\end{array}\right)\!.
\end{gather*}
Since $f_G \varepsilon_G\colon K^{E_G} \rightarrow K^{V'_G}$, viewed as an element in $\mathrm{End}(K^{E_G})$, is the idempotent which projects onto the second factor of \eqref{inproof:orthogdecomp}, it follows that the composition
$f_G^T D_G f_G \varepsilon_G D_G^{-1} \varepsilon_G^T\colon \big(K^{V_G'}\big)^{\vee} \rightarrow \big(K^{V_G'}\big)^{\vee}$ equals $f_G^T\varepsilon^T_G =(\varepsilon_G f_G)^T$, which is simply the identity. Therefore we can replace $f_G^T D_G f_G$ in the previous matrix by $\big(\varepsilon_G D_G^{-1} \varepsilon_G^T\big)^{-1} = L_G^{-1}$.
\end{proof}

\begin{Example} \label{Example: Kn}
Let $K_n$ be the complete graph with $n$ vertices numbered $1,\dots, n$.
The $(n-1)\times (n-1)$ matrix $L_{K_n} $ corresponding to removing the final vertex has entries $\left(L_{K_n} \right)_{ij} =y_{ij}$, where for
 all $1\leq i < j \leq n$,
\begin{gather*}
y_{ij} =y_{ji}= -x_e^{-1}
\end{gather*}
whenever $e$ is the edge between vertices $i$ and $j$, and
\begin{gather*}
y_{ii} = \sum_{e {\text{ meets }} i} x_e^{-1} = - \sum_{k\neq i } y_{ik},
\end{gather*}
where the sum is over all edges $e$ which meet vertex $i$.
For $n=3$,
\begin{gather*}
L_{K_3} = \begin{pmatrix}
 -y_{12}-y_{13} & y_{12} \\
y_{21} & -y_{21}-y_{23}
 \end{pmatrix}\!.
 \end{gather*}
A general $L_{K_n}$ is equivalent to the generic symmetric matrix of rank $n-1$.
\end{Example}

 \subsection{Matroids} \label{sect: Matroids} The previous discussion can be extended to a certain class of matroids \cite{Matroids}. The main application will be to exploit the fact that regular matroids, as opposed to graphs, are closed under the operation of taking duals. This will be used to simplify several proofs, but is not essential to the rest of the paper.

 First of all, observe more generally that the definitions above are valid for any exact sequence of finite-dimensional vector spaces over $\Q$ of the form
 \begin{gather} \tag{$S$}
 0 \To H \To \Q^E \To V \To 0,
 \end{gather}
 where $E$ is a finite set.
 One can define a Laplacian as before:
 \begin{gather*}
 \Lambda_S \in \operatorname{Hom}\big(H, H^{\vee} \otimes_{\Q} \Q[x_e, e \in E] \big)
 \end{gather*}
 which defines a symmetric bilinear form on $H$.
 If one chooses a basis $B$ of $H$, and denotes by~$\mathcal{H}$ the matrix of $H\rightarrow \Q^{E}$ in this basis, then the bilinear form $\Lambda$ is
 represented by the matrix
 $\Lambda_B = \mathcal{H}^T D \mathcal{H}$, where $D$ is the diagonal matrix with entries $x_e$ in the row and column indexed by $e\in E$. Changing basis via a matrix $P \in \mathrm{GL}(H)$ corresponds to the transformation
 \begin{gather} \label{LambdaTransformationMatroids}
 \Lambda_{B'} = P^T \Lambda_B P
 \end{gather}
 from which it follows that $\Psi_S= \det(\Lambda_S)\in \Q[x_e, e\in E]$ is well-defined up to an element of~$(\Q^{\times})^2$. Similarly, we can define a dual Laplacian
 \begin{gather*}
 L_S \in \operatorname{Hom}\big(V^{\vee}, V \otimes_{\Q} \Q\big[x_e^{-1}, e\in E \big] \big)
 \end{gather*}
associated to $S$, and its determinant is likewise well-defined up to an element of $(\Q^{\times})^2$. By~identifying $\Q^E$ with its dual, we can write the dual sequence
\begin{gather} \tag{$S^{\vee}$}
 0 \To V^{\vee} \To \Q^E \To H^{\vee} \To 0 .
 \end{gather}

\begin{lem} \label{lem: dualityandLG}
We have
\begin{gather*}
 \Lambda_{S^{\vee}} = i^* L_{S},
 \end{gather*}
 where $i\colon \Q[x_e, e\in E] \rightarrow \Q\big[x_{e}^{-1}, e\in E\big]$ satisfies $i(x_e) = x_e^{-1}$. Therefore
 \begin{gather*}
 \big( \Psi_{S^{\vee}}(x_e)\big)^{-1} \Psi_{S} \big(x_{e}^{-1}\big) \prod_{e \in E} x_e \in (\Q^{\times})^2.
 \end{gather*}
 \end{lem}

 \begin{proof} The first part follows from the definitions and $D^{-1} = i^* D$. The second part is a~consequence of
 Lemma~\ref{lem: DGdirectsum}.
\end{proof}

In particular, we may write the statement of Lemma~\ref{lem: DGdirectsum} in the form
 \begin{gather} \label{DdecompSdualS}
 D = \Lambda_S \oplus i^* \Lambda_{S^{\vee}}^{-1},
 \end{gather}
 where $D$ denotes the bilinear form on $\Q^E$ considered above.

 \begin{rem}
 Let $M$ be a regular matroid with edge set $E$. A choice of realisation of the matroid defines a surjective map
 $\Q^E \rightarrow V$, where $V$ is a finite-dimensional vector space over~$\Q$. If $H$ denotes its kernel, we obtain a short exact sequence $(M)$
 $0 \rightarrow H \rightarrow \Q^E \rightarrow V \rightarrow 0$.
 When~$M$ is the matroid associated to a graph $G$, it is the exact sequence \eqref{SEHGv2} tensored with~$\Q$.
 The~matroid polynomial is defined to be
 \begin{gather*}
 \Psi_M = \sum_{B} \prod_{e \notin B} x_e,
 \end{gather*}
 where $B$ ranges over the set of bases in $M$. A matroid version of the matrix tree theorem \cite{DSWmatroid, Maurer}
 states that $\Psi_M$ is proportional to $\det(\Lambda_M)$, up to a non-zero element in $(\Q^{\times})^2$.
 It is well-known that the dual matroid $M^{\vee}$ to $M$ can be represented by the exact sequence dual to $(M)$.
Since the coefficients of monomials in the matroid polynomial are $0$ or $1$, it follows from Lemma~\ref{lem: dualityandLG} that
\begin{gather*}
\Psi_{M^{\vee}}(x_e) = \Psi_{M} \big(x_{e}^{-1}\big) \prod_{e \in E} x_e .
\end{gather*}
 In particular, when $G$ is a planar graph, and $G^{\vee}$ a planar dual, one deduces the well-known relationship $\Psi_{G^{\vee}}(x_e) = \Psi_{G} \big(x_{e}^{-1}\big) \prod_{e \in E} x_e$.
\end{rem}

\subsection{Graph matrix} \label{sect: GraphMatrix}
A third way to express the graph polynomial as a matrix determinant arises naturally in the context of Feynman integrals via the Schwinger trick.
It is defined for an exact sequence $(S)$ as follows. Denote the map $\Q^E \rightarrow V$ by $\varepsilon$, its dual $V^{\vee} \rightarrow \big(\Q^E\big)^{\vee}$ by $\varepsilon^T$, and consider the map \begin{align*}
\Q^E \oplus V^{\vee} &\To \big(\big(\Q^E\big)^{\vee} \oplus V \big)\otimes_{\Q} \Q[x_e, e\in E],
\\
( f, v) &\mapsto \ \ \big(Df - \varepsilon^{T}(v), \varepsilon(f)\big),
 \end{align*}
 where $D$ was defined earlier. It defines a bilinear form on $\Q^E\oplus V^{\vee}$ taking values in $\Q[x_e, e\in E]$, whose restriction to the subspace $V^{\vee}$ is identically zero.

 In the case when the exact sequence $(S)$ arises from a graph, we call
 the following square matrix of rank
 $(E_G+V_G-1) \times (E_G+V_G-1)$
\begin{gather*}
M_G = \left(
\begin{array}{c|c}
 D_G & - \varepsilon^T_G \\ \hline
 \varepsilon_G & 0 \\
\end{array}
\right)
\end{gather*}
a (choice of) graph matrix. Here, $\varepsilon_G$ is a reduced incidence matrix, which, we recall, depends on a choice of deleted vertex $v$ (and choice of bases).

\begin{lem} \label{lemMGLBU}
We can write $M_G = LBU$, where
\begin{gather*}
L= \left(\begin{array}{c|c}
 I & 0 \\ \hline
 \varepsilon_G D_G^{-1^{\vphantom{1}}} & I
\end{array}\right)\!, \qquad
B = \left(\begin{array}{c|c}
 D_G & 0 \\ \hline
 0 & L_G
\end{array}\right)\!, \qquad
U = \left(\begin{array}{c|c}
 I & - D^{-1}_G \varepsilon^T_G \\ \hline
 0 & I
\end{array}\right)
\end{gather*}
and $I$ are identity matrices of the appropriate rank. In particular, $\det(M_G) = \Psi_G$.
\end{lem}

\begin{proof}
The decomposition $M_G = LBU$ is straightforward. We deduce that
$\det(M_G) = \det(LBU)=\det(B)= \det(D_G) \det(L_G) $ and apply Lemma~\ref{lem: DGdirectsum}.
\end{proof}

\subsection{Variants of graph polynomials}
The following polynomials are instances of what we called ``Dodgson polynomials'' in \cite{PeriodsFeynman}.
\begin{defn} \label{defn: Dodgson}
Let us denote by
\begin{gather*}
\Psi^{I,J}_G = \det(M_G(I,J)),
\end{gather*}
where $M_G(I,J)$ denotes the minor of $M_G$ with rows $I$ and columns $J$ removed, where $I$, $J$ are subsets of $E_G$ such that $|I|=|J|$.
 We write $\Psi_G^{ij}$ instead of $\Psi_G^{\{i\}, \{j\}}$. \end{defn}

For general $I$, $J$, the polynomial $\Psi^{I,J}_G$ depends on the choice of graph matrix $M_G$ by a possible sign.
Since $M_G$ is symmetric, $\Psi_G^{ij}= \Psi_G^{ji}$
 and can be expressed as sums over spanning forests which include or avoid the edges $i$,$j$. In particular:
\begin{gather*} 
\Psi_G^{ii} = \Psi_{G \backslash i } = \frac{\partial}{\partial x_i} \Psi_G.
\end{gather*}

\section{Maurer--Cartan differential forms and invariant traces}

Let $R=\bigoplus_{n\geq 0 } R^n$ be a graded-commutative unitary differential graded algebra over $\Q$ whose differential ${\rm d}\colon R^n \rightarrow R^{n+1}$ has degree $+1$. In particular, for any homogeneous elements $a$, $b$ one has $a. b = (-1)^{\deg(a) \deg(b)} b.a$.

\subsection{Definition of the invariant trace} \label{sect: defbetabasic}

\begin{defn}
For any invertible $(k\times k)$ matrix $X \in \GL_k\big(R^0\big)$, let
\begin{gather*}
\mc_X = X^{-1} {\rm d}X \in M_{k \times k} \big(R^1\big) .
\end{gather*}
For any $n\geq 0$ consider the elements
\begin{gather*}
\beta^n_X = \tr \big(\big(X^{-1} {\rm d}X\big)^n\big) \in R^n .
\end{gather*}
\end{defn}
Denote by $I_k \in \GL_k\big(R^0\big)$ the identity matrix of rank~$k$.
 \begin{lem} \label{lemma: MC}
 The matrix $\mc_X$ satisfies the Maurer--Cartan equation
 \begin{gather*}
 {\rm d} \mc_X + \mc_X \mc_X = 0 .
 \end{gather*}
 From this it follows that $ {\rm d}\big(\mc_X^{2n}\big)=0$ and ${\rm d}\big(\mc_X^{2n-1}\big) = - \mc_X^{2n}$ for all $n \geq 1$.
 \end{lem}
 \begin{proof}
 Since $X. X^{-1} =I_k$ we deduce that $X {\rm d}\big(X^{-1}\big) + {\rm d}X. X^{-1} =0$. It follows that ${\rm d}\big(X^{-1}\big) = -X^{-1} {\rm d}X. X^{-1}$, and therefore ${\rm d} \mc_X = {\rm d}\big(X^{-1}\big) {\rm d}X= - \mc_X^2$. Now
 \begin{gather*}
 {\rm d} \mc_X^2 = {\rm d}\mc_X . \mc_X - \mc_X {\rm d}\mc_X = - \mc_X^3 + \mc_X^3=0.
 \end{gather*}
 From this it follows that all even powers $\mc^{2n}_X = \big(\mc_X^2\big)^n$ are closed under ${\rm d}$, including the case $n=0$, since $\mc_X^0$ is the identity. This in turn implies that for any $n\geq 1$, we have ${\rm d} \big(\mc_X . \mc_X^{2n-2}\big) = {\rm d}\mc_X. \mc_X^{2n-2} = -\mc_X^2 \mc_X^{2n-2} =- \mc_X^{2n}$ as required.
 \end{proof}

 The following properties of $\beta^n_X$ are well-known.

 \begin{lem} \label{lem: BasicPropertiesBeta}
 The elements $\beta_X$ satisfy the following properties for all $n\geq 1{:}$
\begin{enumerate}\itemsep=0pt
\item[$(i)$] $\beta_X^n = \tr \big( \big({\rm d}X. X^{-1}\big)^n\big)$,
\item[$(ii)$] $\beta_{X^{-1}}^n = (-1)^n \beta_{X}^n$,
\item[$(iii)$] $\beta_{X^T}^n = (-1)^{\frac{n(n-1)}{2}} \beta_{X}^n$,
\item[$(iv)$] $\beta_{X}^{2n} =0$,
\item[$(v)$] ${\rm d}\beta_{X}^{2n+1} =0$,
\item[$(vi)$] $\beta_{X_1 \oplus X_2}^{n} =\beta_{X_1}^{n} + \beta_{X_2}^n$.
\end{enumerate}
 The map $X\mapsto \beta^n_X$ is invariant under left or right multiplication by any constant invertible matrix $A \in \GL_k\big(R^0\big)$. In other words,
 \begin{gather*}
 \beta^n_X= \beta^n_{AX} = \beta^n_{XA} \qquad \text{if} \quad {\rm d}A=0 .
 \end{gather*}
 \end{lem}
\begin{proof}
Property $(i)$ follows from cyclicity of the trace. From this follows $(ii)$ since $\mc_{X^{-1}} = - {\rm d}X. X^{-1}$ via the computation in the proof of Lemma~\ref{lemma: MC}. To deduce $(iii)$, note that $(\mc_X)^T = {\rm d} \big(X^T\big) \big(X^T\big)^{-1}$. Therefore we check that:
\begin{gather*}
\beta^n_{X^T} \overset{(i)}{=} \tr \big(\big({\rm d}X^T. \big(X^T\big)^{-1}\big)^n \big) =\tr \big(\big(\mc_X^T\big)^n\big) .
\end{gather*}
Since transposition is an anti-homomorphism, $\big(\mc_X^n\big)^T = (-1)^{\frac{n(n-1)}{2}} \big(\mc_X^T\big)^n$ since $\mu_X$ has degree~$1$, and the sign is that of the permutation which reverses the order of a sequence of $n$ objects. We~therefore obtain
\begin{gather*}
 \beta^n_{X^T} = (-1)^{\frac{n(n-1)}{2}} \tr \big((\mc^n_X)^T\big) =(-1)^{\frac{n(n-1)}{2}} \beta_X^n .
 \end{gather*}
Property $(iv)$ uses the cyclicity of the trace and graded-commutativity:
\begin{gather*}
\tr\big(\mc_X^{2n}\big) = \tr\big( \mc_X^{2n-1} \mc_X\big)= \tr\big( (-1)^{2n-1} \mc_X \mc_X^{2n-1}\big) = (-1)^{2n-1} \tr\big(\mc_X^{2n}\big) .
\end{gather*}
Property $(v)$ follows from the fact that ${\rm d} \big(\mc_X^{2n+1}\big) = -\mc_X^{2n+2}$ by Lemma~\ref{lemma: MC}, which has vanishing trace by $(iv)$. Since the trace is linear it clearly commutes with the differential ${\rm d}$. Property~$(vi)$ is immediate from the definitions, where $X_1\oplus X_2$ is the block diagonal matrix with two non-zero blocks~$X_1$,~$X_2$ on the diagonal.
For the last statement, consider any two invertible matrices $A,B \in \GL_k\big(R^0\big)$, which are constant, i.e., ${\rm d}A={\rm d}B=0$. We have
\begin{gather*}
\mc^n_{AXB} = \big((A X B)^{-1} A \,{\rm d}X. B\big)^n = (B^{-1} \big(X^{-1} {\rm d}X\big) B)^n = B^{-1} \mc_X^n B,
\end{gather*}
from which it follows that $\beta_{AXB}^n= \beta_{X}^n$ by the cyclic invariance of the trace.
\end{proof}

The following lemma is a projective invariance property for $\beta_X^{2n+1}$ for $n\geq 1$.
\begin{lem} \label{lem: projinv}
Let $\lambda \in \big(R^0\big)^{\times}$ be invertible of degree zero. Then
\begin{gather*}
\beta^{2n+1}_{\lambda X} = \beta^{2n+1}_{X} \qquad \text{for all} \quad n\geq 1.
\end{gather*}
For $n=0$ however, one has $\beta^1_{\lambda X} = \beta^1_{X} + k \lambda^{-1} {\rm d}\lambda$, where $k$ is the rank of $X$.
\end{lem}
\begin{proof} Writing $\lambda X = X. \lambda I_k$, we have
\begin{gather*}
\mc_{\lambda X}=\lambda^{-1}\mc_X\lambda + \mc_{\lambda I_k}=\mc_X + (\lambda^{-1} {\rm d}\lambda)I_k.
\end{gather*}
Taking the trace proves the last statement.
Since $({\rm d} \lambda)^2=0$ and $I_k$ is central, we deduce that
$ \mc^{2m}_{\lambda X} = \mc_X^{2m}$ and $ \mc^{2m+1}_{\lambda X} = \mc_X^{2m+1} + \mc_X^{2m} \big(\lambda^{-1} {\rm d}\lambda \big) $
for all $m\geq 0$. Taking the trace gives $\beta^{2m+1}_{\lambda X} = \beta_X^{2m+1} +\tr \big(\mc_X^{2m}\big) \lambda^{-1} {\rm d}\lambda $. One concludes using Lemma~\ref{lem: BasicPropertiesBeta}$(iv)$.
\end{proof}

The following proposition has important consequences.
 \begin{prop} \label{prop: betavanishes}
 Let $X$ be an invertible $n\times n$ matrix. Then
\begin{gather*}
\beta_X^{m} =0 \qquad \text{for all}\quad m\geq 2n.
\end{gather*}
\end{prop}

\begin{proof}
It suffices to prove the stronger statement:
\begin{gather} \label{mu2nvanishes}
\mc^{2n}_X= 0 .
\end{gather}
For this, we adapt an argument due to Rosset \cite{Rosset76}, final paragraph.
The matrix $\mu_X^2$ has entries in the commutative ring $R^{\mathrm{even}} = \bigoplus_{n\geq 0} R^{2n}$, and therefore by a well-known result in linear algebra, $\big(\mu_X^2\big)^{n}=0$ holds if $\tr\big(\mu_X^{2m}\big)=0$ for all $m\geq 1$. The latter statement follows from Lemma~\ref{lem: BasicPropertiesBeta}$(iv)$. The linear algebra result referred to above follows from the Cayley--Hamilton theorem, namely, that a matrix $M$ over a commutative ring satisfies its characteristic polynomial equation, and the fact that the coefficients in the characteristic polynomial can be expressed in terms of traces of powers of $M$, which follows from Newton's identities on symmetric functions.
\end{proof}

\begin{rem} In order to connect more directly with the presentation in \cite{Rosset76}, note that
the entries of $\mc^{2n}_X $ lie in the subspace $\textstyle{\bigwedge^{\!\!2n}} R^1 \subset R^{2n}$ generated by exterior products of elements of degree $1$. Let $\{e_i\}$, where $e_i \in R^1$, denote a $\Q$-basis for the vector space generated by the entries of $\mu_X$. We may write $\mc^{2n}_X$ as a finite sum
\begin{gather*}
\mc^{2n}_X = \sum_I \mu_I e_I,
\end{gather*}
where for a set of indices $I=\{i_1,\dots, i_{2n}\}$, $e_I = e_{i_1} \wedge \dots \wedge e_{i_{2n}}$, and where $\mu_I \in \Q$. Equation~\eqref{mu2nvanishes} is equivalent to $\mu_I=0$ for all $I$. Therefore \eqref{mu2nvanishes} reduces to the case where $R$ is the exterior algebra on the $\Q$-vector space with basis $e_1,\dots, e_{2n}$, and
\begin{gather*}
\mu_X = M_1 e_1 + \dots + M_{2n} e_{2n},
\end{gather*}
where $M_i\in M_{n\times n}(\Q)$ are $n\times n$ matrices with rational coefficients.
The statement $\mu_X^{2n}=0$ is proven by Rosset in \cite{Rosset76}, final paragraph. It is equivalent to the Amitsur--Levitzki theorem for the ring $\Q$, which in this case states that
\begin{gather*}
\sum_{\sigma \in \Sigma_{2n}} \mathrm{sgn}(\sigma) M_{\sigma(1)}\cdots M_{\sigma(2n)}=0.
\end{gather*}
\end{rem}
For historical background on invariant forms and their role in the development of Hopf algebras, see Cartier's survey paper
\cite[Section~2.1]{Cartier} and references therein.

\subsection{Invariant classes}
For any invertible matrix $X$ with coefficients in $R^0$, we obtain closed elements
\begin{gather*}
\beta^{2n+1}_X \in R^{2n+1} \qquad \text{for all}\quad n\geq 0
\end{gather*}
and hence potentially non-trivial cohomology classes for all $n\geq 1$:
\begin{gather*}
\big[\beta^{2n+1}_X \big] \in H^{2n+1}(R).
\end{gather*}
If, however, $X=X^T$ is symmetric, then $\beta_X^{4n+3}$ vanishes for all $n$ by property $(iii)$, and hence only the following subset could possibly give rise to non-trivial classes:
\begin{gather*}
\beta^{4n+1}_X \in R^{4n+1} \qquad \text{for all}\quad n\geq 0.
\end{gather*}
Since $\beta^1_X$ is not invariant under multiplication $X \mapsto \lambda X$ in general (see Lemma~\ref{lem: projinv}), we obtain a more restricted list of ``projectively-invariant''classes:
\begin{gather*}
\beta^5_X,\quad \beta^9_X,\quad \beta^{13}_X ,\quad \dots .
\end{gather*}

\begin{Example} \label{examples: smallbetas} Consider the generic two-by-two matrix
\begin{gather*} X= \begin{pmatrix}
a_{1} & a_{3}\\
a_{4} & a_{2}
\end{pmatrix}
\end{gather*}
with coefficients in the field $R^0 =\Q(a_{1},\dots, a_4)$, and set $R^n = \Omega^n_{R^0/\Q}$. Then
\begin{gather*}
\beta^1_X = \frac{ a_1 {\rm d}a_2 +a_2 {\rm d}a_1 -a_3 {\rm d}a_4 -a_4 {\rm d}a_3 }{a_1a_2-a_3a_4} =
{\rm d} \log ( \det(X))
\end{gather*}
and $\beta^3_X$ is given by the expression
\begin{gather*}
 \beta^3_X = 3 \frac{ \sum_{i=1}^4 (-1)^i a_i \, {\rm d}a_1 \cdots \widehat{{\rm d}a_i} \cdots {\rm d}a_4 }{ (a_1a_2-a_3a_4)^2}.
 \end{gather*}
All higher $\beta^{2n+1}_X$ vanish for reasons of degree.

Now consider the generic three-by-three symmetric matrix:
\begin{gather*}
X= \begin{pmatrix}
a_{1} & a_{4} & a_5\\
a_{4} & a_{2} & a_6 \\
a_5 & a_6 & a_3
\end{pmatrix}
\end{gather*}
with coefficients in the field $R^0 =\Q(a_1,\dots, a_6)$, and let $R^n = \Omega^n_{R_0/\Q}$. Then
\begin{gather*}
\det(X)= a_1a_2a_3-a_1a_6^2-a_2a_5^2-a_3a_4^2+2 a_4a_5a_6.
\end{gather*}
One has $\beta^1_X = {\rm d} \log (\det (X))$,
$\beta^3_X=0$ and we verify that
\begin{gather*}
\beta^5_X = -10 \frac{ \sum_{i=1}^6 (-1)^i a_i \, {\rm d}a_1 \cdots \widehat{{\rm d}a_i} \cdots {\rm d}a_6 }{ (\det(X))^2}.
\end{gather*}
Once again, all higher elements $\beta^{2n+1}_X$ vanish. For larger matrices, the number of terms occurring in an $\beta^{2n+1}_X$ grows rapidly.
\end{Example}

In general, the forms $\beta^{2n+1}_X$ for $n\geq 1$ define interesting cohomology classes on the complement of hypersurfaces in projective space which are defined by the vanishing locus of $\det(X)$. We shall mostly be concerned with symmetric matrices.

\subsection{Hopf algebra structure and stable cohomology of the general linear group}
Let $G= \GL_g(\R)$ be the general linear group of rank $g$ and let $K\leq G$ be a maximal compact subgroup. The symmetric space $X=K\backslash G$ may be identified with the space of positive definite real symmetric matrices of rank $g$. Each $\beta^{4k+1}$ for $k\geq 1$ defines a closed $\GL_g(\Z)$-invariant differential form on $X$ and hence
a class in the cohomology of the orbifold $X / \GL_g(\Z)$:
\begin{gather*}
 \big[\beta^{4k+1}\big] \in H^{4k+1} (X / \GL_g(\Z);\R) \cong H^{4k+1}(\GL_g(\Z);\R)
 \end{gather*}
which is compatible with the natural maps $\GL_g \rightarrow \GL_{g+1}$. Borel famously proved in \cite{Borel} that the invariant forms generate the stable real cohomology:
\begin{gather*}
 H^{\bullet}(\GL(\Z);\R) = \underset{\leftarrow}{\lim} \, H^{\bullet}(\GL_g(\Z);\R),
 \end{gather*}
which is consequently isomorphic to the graded exterior algebra on the classes $\beta^{4k+1}$, for all $k\geq 1$.
Taking the limits as $m,n \rightarrow \infty$ of the map
\begin{gather*}
(X_1, X_2) \mapsto X_1 \oplus X_2\colon\ \mathrm{GL}_{m} \times \mathrm{GL}_n \rightarrow \mathrm{GL}_{m+n}
\end{gather*}
 induces a comultiplication on $H^{\bullet}(\GL(\Z);\R)$.
Since $\beta^{4k+1}_{X_1\oplus X_2} = \beta^{4k+1}_{X_1} + \beta^{4k+1}_{X_2}$, it is induced by the coproduct with respect to which the classes $\beta^{4k+1}$ are primitive:
\begin{gather} \label{betaprimitive} \Delta \beta^{4k+1} = \beta^{4k+1} \otimes 1 + 1 \otimes \beta^{4k+1} .
\end{gather}
Borel deduced that the rank of the rational algebraic $K$-theory of the integers $K_i(\Z)\otimes \Q$ for $i\geq 2$ is one if $i\equiv 1 \mod 4$, and $0$ otherwise. Note that for every $k\geq 1$, the Lie algebra element $\sigma_{2k+1}$ mentioned in the introduction, or rather its class modulo commutators, is dual to a generator of $K_{4k+1}(\Z) \otimes \Q$.

\section{Further properties of invariant forms}
{\sloppy
The following, somewhat technical, section proves some additional formulae for invariant forms~$\beta^n_X$ by using matrix factorisations of $X$.

}

\subsection{Decomposition into block-matrix form} \label{sect: DecompBlockMatrix}
In order to obtain more precise information about the elements $\beta^{2n+1}_X$, it is convenient to fix a~decomposition of $X$ into block-matrix form. We shall either:
\begin{enumerate}\itemsep=0pt
\item Let $R^{\bullet}$ be the ring of K\"ahler differentials $\Omega^{\bullet}_{R^0/\Q}$, where $R^0=\Q(a_{ij})_{1\leq i,j\leq k}$,
and write $X=(a_{ij})_{ij}$ for the generic $(k\times k)$ matrix with entries in $R^0$.

\item As above except that $R^0 = \Q(a_{\{i,j\}})_{1\leq i\leq j\leq k}$,
and $X=(a_{\{i,j\}})_{ij}$ is the generic symmetric $(k\times k)$ matrix with entries in $R^0$.
\end{enumerate}

\noindent
In either situation,
we may view
$ X \in \GL_k\big(R^0\big) $ as an endomorphism of the $R^0$-vector space $V= \bigoplus_{i=1}^k R^0$. Let us fix a decomposition
\begin{gather*}
V= V_1 \oplus \dots \oplus V_n,
\end{gather*}
where each $V_i$ is a direct sum of copies of $R^0$. It follows from the theory of Schur complements\footnote{Namely, the following identity for block matrices, where $A$, $D$ are square matrices \begin{gather*}
\begin{pmatrix} A& B \\ C & D \end{pmatrix} = \begin{pmatrix} I & 0 \\ CA^{-1} & I \end{pmatrix} \begin{pmatrix} A & 0 \\ 0 & D- C A^{-1}B \end{pmatrix} \begin{pmatrix} I & A^{-1}B \\ 0 & I \end{pmatrix} \end{gather*}
which holds whenever the matrix $A$ is invertible. It can be applied repeatedly to any decomposition of $V$ as a~direct sum of two subspaces. }
 and genericity of $X$ that it can be written uniquely in the form
\begin{gather*} 
X = L B U,
\end{gather*}
where $B= \bigoplus_{i=1}^n B_i$ is block-diagonal, $L-I$ is strictly block lower-triangular, and $U-I$ is strictly block upper-triangular with entries in $R^0$. From this we deduce that
 \begin{align*} \nonumber
 U \mu_X U^{-1} & = U (LBU)^{-1} {\rm d}(LBU) U^{-1}
 \\
 & = \mathcal{L} + \mathcal{B} + \mathcal{U}, \nonumber
 \end{align*}
 where
 \begin{gather} \label{CurlyLBUdef}
 \mathcal{L} = B^{-1} \big( L^{-1} {\rm d}L\big) B , \qquad
 \mathcal{B} = B^{-1} {\rm d}B, \qquad \mathcal{U} = {\rm d}U. U^{-1}
 \end{gather}
 are strictly block lower-triangular, block diagonal, and strictly block upper-triangular respectively. By the cyclic invariance of the trace, we conclude that
 \begin{gather} \label{betaastraceLBU}
 \beta^n_X = \tr \big( U \mc^n_X U^{-1} \big) = \tr
 \big((\mathcal{L} + \mathcal{B} + \mathcal{U})^n\big).
 \end{gather}
 This formula can lead to more efficient ways of computing the $\beta^n_X$ than using the definition, since many terms in an expansion of $(\mathcal{L} + \mathcal{B} + \mathcal{U} )^n$ have vanishing trace.

\subsection[Decomposition of type (m,1)]{Decomposition of type $\boldsymbol{(m,1)}$}
Consider the special case
\begin{gather*}
V= V_1 \oplus V_2,
\end{gather*}
where $V_1 = \big(R^0\big)^{\oplus m}$ and $V_2 = R^0$ is one-dimensional.
We have
\begin{gather*}
L= \left(\begin{array}{c|c}
 I_m & \\ \hline
 \underline{\ell} & 1 \\
\end{array}\right)\!, \qquad
B= \left(\begin{array}{c|c}
 B_1 & \\ \hline
 & b \\
\end{array}\right)\!, \qquad
U= \left(\begin{array}{c|c}
 I_m & \underline{u}^T \\
 \hline
 & 1 \\
\end{array}\right)\!,
\end{gather*}
where $\underline{\ell}=(\ell_1 \cdots \ell_m)$ and $\underline{u}= (u_1\cdots u_m)$ are $(1 \times m)$ matrices and all blank entries denote zero matrices.
By solving $X=LBU$ for $\underline{\ell}$, $\underline{u}$, $B$, we find that
\begin{gather}
B_1 = X(m+1,m+1), \nonumber
\\
b = \det(X)/ \det(X(m+1,m+1)),\label{B1asXminor}
\end{gather}
where $X(m+1,m+1)$ denotes the $(m\times m)$ minor of $X$ obtained by deleting row $m+1$ and column $m+1$. It is invertible, hence in $\GL_{m}(R^0)$, by assumption of genericity.
We find that
\begin{gather*}
\mathcal{L} = \left(
\begin{array}{c|c}
 & \quad \\ \hline
 b^{-1^{\vphantom{1}}} {\rm d} \underline{\ell}. B_1 &
\end{array}
\right)\!,
 \qquad \mathcal{B} = \left(
\begin{array}{c|c}
 \mc_{B_1} & \\ \hline
 & b^{-1^{\vphantom{1}}} {\rm d}b
\end{array}
\right)\!, \qquad
\mathcal{U} = \left(
\begin{array}{c|c}
\quad & {\rm d}\underline{u}^T \\
 \hline
 & \\
\end{array}
\right)\!,
\end{gather*}
where all blank entries are zero. We have $\mathcal{L} \mathcal{B}^i \mathcal{L}= \mathcal{U} \mathcal{B}^i \mathcal{U}=0$ for all $i\geq 0$.
Since $\big(b^{-1} {\rm d}b\big)^2=0$, $\mathcal{B}^2$ is zero except in the top-left corner and so $\mathcal{B}^2 \mathcal{L} = \mathcal{U} \mathcal{B}^2=0$. It follows that $\beta^n_X$ is a linear combination of traces of words in $\mathcal{L}$, $\mathcal{B}$, $\mathcal{U}$ of the form
\begin{gather*}
 \mathcal{B}^{i_0} \mathcal{L} \mathcal{B}^{i_1} \mathcal{U} \mathcal{B}^{i_2} \mathcal{L} \mathcal{B}^{i_3} \mathcal{U} \cdots, \qquad \text{where} \quad i_0,i_1,\ldots \geq 0,
 \end{gather*}
 and where the matrices $\mathcal{L}$ and $\mathcal{U}$ alternate and are interspersed with a power of $\mathcal{B}$; or a similar expression in which $\mathcal{L}$, $\mathcal{U}$ are interchanged. By cyclicity of the trace, the latter reduces to the former; furthermore, the
number of $\mathcal{L}$'s and $\mathcal{U}$'s in such a word must be equal in order for the trace to be non-zero.
We can also assume $i_{2k} \in \{0, 1\}$ for all $k$ since $\mathcal{B}^2 \mathcal{L} = \mathcal{U} \mathcal{B}^2=0$. In~summary, $\beta^n_X$ is a linear combination of traces of products of block-diagonal matrices:
 \begin{gather*}
 \mathcal{B}\qquad \text{and} \qquad \mathcal{L} \mathcal{B}^i \mathcal{U} \qquad \text{for}\quad i \geq 0.
 \end{gather*}
 Write
\begin{gather*}
\mathcal{L}\mathcal{B}^i \mathcal{U} = \left(
\begin{array}{c|c}
 0 & 0 \\ \hline
 0 & \nu_i \\
\end{array}
\right)\!,
\end{gather*}
where for all $i\geq 0$, we define
\begin{gather} \label{nuidefn}
\nu_i = b^{-1} \big({\rm d} \underline{\ell} B_1 \big(B_1^{-1} {\rm d}B_1 \big)^i {\rm d} \underline{u}^T \big) \in R^{i+2}.
\end{gather}
By equation \eqref{betaastraceLBU}, we deduce that for all $n\geq 2$,
\begin{gather} \label{betaXasB1andnus}
\beta_{X}^n = \beta_{B_1}^n + \big(\text{a linear combination of exterior products of } \nu_i,\ b^{-1} {\rm d}b \big).
\end{gather}

\begin{lem} \label{lem: nuvanishes} If $X$ is symmetric, $\nu_i=0$ and
 $\mathcal{L} \mathcal{B}^i \mathcal{U} =0$ whenever $i\equiv 0,1 \pmod 4$.
\end{lem}

\begin{proof} Since $X$ is symmetric, it follows that $B_1$ is also symmetric, and $\underline{\ell}= \underline{u}$. By the definition~\eqref{nuidefn}, we can write:
\begin{gather*}
 b \nu_i = {\rm d} \underline{\ell} \big({\rm d}B_1 B_1^{-1} \cdots B_1^{-1} {\rm d}B_1 \big) {\rm d} \underline{\ell}^T,
 \end{gather*}
 where the term in brackets in the middle has degree $i$. Since transposition is an anti-ho\-mo\-mor\-phism, we find that
\begin{gather*}
 (b \nu_i)^T = \big({\rm d}\underline{\ell} \big({\rm d}B_1 B_1^{-1} \cdots B_1^{-1} {\rm d}B_1 \big) {\rm d} \underline{\ell}^T \big)^T = (-1)^{\frac{(i+2)(i+1)}{2}} b \nu_i.
 \end{gather*}
 Since
$b \nu_i$ is a $(1\times 1)$ matrix and equals its own transpose, it must be equal to zero whenever the sign in the right-hand side is negative, i.e., if $i\equiv 0,1 \pmod 4$.
\end{proof}

We deduce the optimal power of $\det(X)$ in the denominator of the forms $\beta^{n}_X$.
 \begin{thm} \label{thm: optimaldenom} For
 any invertible matrix $X$ we have
\begin{gather*}
\beta^1_X = {\rm d} \log (\det(X))
\end{gather*}
and
\begin{gather} \label{denomBetaGeneral}
\beta^{2n+1}_X \in \frac{1}{\det(X)^{n+1}} \Omega^{2n+1}_{\Q[ a_{i,j}]/\Q}.
\end{gather}
If, furthermore, $X$ is symmetric then the power of the determinant in the denominator drops by another factor of two. Indeed, in this case we have
\begin{gather} \label{denomBetaSymmetric}
\beta^{4n+1}_X \in \frac{1}{\det(X)^{n+1}} \Omega^{4n+1}_{\Q[ a_{\{i,j\}}]/\Q},
\end{gather}
i.e., $\beta^{4n+1}_X$ is a polynomial form in $a_{\{i,j\}}$, ${\rm d}a_{\{i,j\}}$,
 divided by $\det(X)^{n+1}$.
\end{thm}
\begin{proof}
The theorem is first proven for generic matrices (Section~\ref{sect: DecompBlockMatrix}, situation (1) in the general case, and situation (2) for the case when $X$ is symmetric). The statements for an arbitrary invertible matrix follow by specialisation. In other words, we first prove the identity \eqref{denomBetaGeneral} (resp.~\eqref{denomBetaSymmetric}) on the algebraic variety of generic (resp.~generic symmetric) matrices which is an open subvariety of the space of all invertible matrices. Since the identities are algebraic, they remain valid on its Zariski closure, where strict minors of $X$ (but not its determinant), are allowed to vanish.
The first statement can be proven by induction on the rank of $X$. It is clear for matrices of rank $1$. Using \eqref{B1asXminor} we have
\begin{gather*}
\beta^1_X = \beta^1_{B_1} + {\rm d} \log b.
\end{gather*}
Since $B_1 $ has smaller rank than $X$, the induction hypothesis gives
\begin{gather*}
\beta^1_X = {\rm d} \log (\det(X(m+1,m+1))) + {\rm d}\log b \overset{\eqref{B1asXminor}}{=} {\rm d}\log (\det(X)).
\end{gather*}
 It is immediate from the definition of the invariant trace $\beta^{2n+1}_X$ of $X$ that it only has denominator $\det(X)$, i.e., its entries lie in
\begin{gather*}
\Q[a_{ij}, {\rm d} a_{ij} , \det(X)^{-1}] .
\end{gather*}
 Let $v_{\det(X)}$ denote the valuation on $R$ defined by the negative of the order of poles in $\det(X)$. It is known, for both generic symmetric and generic non-symmetric matrices, that $\det(X)$ is irreducible. From equations \eqref{B1asXminor} and \eqref{nuidefn} we obtain
\begin{gather*}
v_{\det(X)} \big(\beta^{2n+1}_{X(m+1,m+1)}\big) =0, \qquad
v_{\det(X)} \big(b^{-1} {\rm d}b\big) = v_{\det(X)} (\nu_i) =-1 \qquad \text{for all} \quad i\geq 0.
\end{gather*}
 All terms in \eqref{betaXasB1andnus} have degree at most one in $b^{-1} {\rm d}b$ since it squares to zero. Because $\deg \nu_i =i+2 \geq 2$, there can be at most $n$ terms of type $\nu_i$ in the expression \eqref{betaXasB1andnus} for $\beta_X^{2n+1}$. We therefore deduce that $v_{\det(X)}\big(\beta^{2n+1}_X\big) \geq -n-1$, which proves~\eqref{denomBetaGeneral}.

 When $X$ is symmetric, the proof of \eqref{denomBetaSymmetric} goes along very similar lines. By Lemma~\ref{lem: nuvanishes}, $\nu_0 =\nu_1=0$ and therefore every non-trivial form $\nu_i$ has degree $\geq 4$. It follows that there can be at most $n$ of them in the expansion \eqref{betaXasB1andnus} for $\beta_X^{4n+1}$ and therefore $v\big(\beta^{4n+1}_X\big) \geq - n-1$.
 \end{proof}

\subsection[Decomposition of type (1,...,1)]
{Decomposition of type $\boldsymbol{(1,\dots,1)}$}

Consider a decomposition of the form $X=LBU$, where $B$ is diagonal, and $L$ (resp.~$U$) is lower (resp.~upper) triangular with 1's on the diagonal. Define $\mathcal{L}$, $\mathcal{B}$, $\mathcal{U}$ using \eqref{CurlyLBUdef}. Since $B$ is diagonal, $\mathcal{B}^2=0$. Suppose that $X$ is symmetric of rank $2n+1\geq 3$, and denote the diagonal entries of $B$ by $b_1,\dots, b_{2n+1}$.
Write $\mathcal{W} = \mathcal{L} + \mathcal{U}$. Using \eqref{betaastraceLBU} and $\mathcal{B}^2=0$ we find that
\begin{gather*}
\beta_X^{4n+1} = \tr (\mathcal{W} + \mathcal{B})^{4n+1} = (4n+1) \tr \big(\mathcal{W} (\mathcal{B} \mathcal{W})^{2n}\big) + \cdots,
\end{gather*}
where $\cdots$ denotes terms involving fewer than $2n$ matrices $\mathcal{B}$ (in some circumstances of interest, these terms vanish for reasons of degree). This uses the fact that $n\geq 1$.
If we write
\begin{gather*}
\Omega_B = \sum_{i=1}^{2n+1} (-1)^i b_i {\rm d}b_1 \wedge \dots \wedge \widehat{{\rm d}b_i} \wedge \dots \wedge {\rm d}b_{2n+1}
\end{gather*}
then one can deduce from the definition of the trace that the leading term of $\beta_X^{4n+1}$ is
\begin{gather} \label{PulloutEigenvalues} \tr \left( \mathcal{W} (\mathcal{B} \mathcal{W})^{2n}\right) = \frac{ 1}{\det(B)} \Omega_B \wedge \bigg( \sum_{\gamma} \mathcal{W}_{1, \gamma(1)} \wedge \dots \wedge \mathcal{W}_{2n+1, \gamma(2n+1)} \bigg),
\end{gather}
where the sum ranges over all $(2n)! = (2n+1)!/(2n+1)$ permutations $\gamma$ of $1,\dots, 2n+1$ modulo cyclic permutations.

\section{Canonical differential forms associated to graphs}
We define canonical differential forms associated to graphs via their Laplacian matrix
 and derive some first properties. In this section, the forms will be viewed as meromorphic functions on projective spaces (i.e., before performing any blow-ups).

\subsection{Canonical graph forms}
 For any finite set $S$, let $\Pro^S= \Pro\big(\Q^S\big)$ denote the projective space over $\Q$ of dimension $|S|-1$ with projective coordinates $x_s$ for $s\in S$.
 Let $G$ be a connected graph.

\begin{defn}The \emph{graph hypersurface} $X_G \subset \Pro^{E_G}$ is defined \cite{BEK} to be the zero locus of the homogeneous polynomial $\Psi_G$.

Define the open coordinate simplex $\sigma_G \subset \Pro^{E_G}(\R)$ to be
\begin{gather*}
 \sigma_G = \{ (x_e)_{e \in E_G}\colon x_e >0 \} .
 \end{gather*}
\end{defn}

The polynomial $\Psi_G$ is positive on $\sigma_G$
 since by Theorem~\ref{thm: MatrixTree} it
 is a non-trivial sum of monomials with positive coefficients. Therefore
 \begin{gather*}
 \sigma_G \cap X_G = \varnothing.
 \end{gather*}

 Let $\Lambda_G$ be any choice of Laplacian matrix. Its coefficients are elements of
 \begin{gather*}
 R_G^0= \Q \big[(x_e)_{e\in E_G}, \Psi_G^{-1}\big]
 \end{gather*}
 and $\Lambda_G \in \GL_{h_G} \big(R_G^0\big)$ is invertible.
 Let $R_G^{\bullet}= \Omega^{\bullet} \big( \mathrm{Spec}\big(R_G^0\big)\big)$
 be the K\"ahler differentials on the affine hypersurface complement $\mathbb{A}^{E_G} \backslash \big(X_G \cap \mathbb{A}^{E_G}\big)$.

 \begin{defn} For every integer $k\geq 1$, define
\begin{gather*} 
\omega^{4k+1}_G = \beta_{\Lambda_G}^{4k+1} \in R_G^{4k+1} .
\end{gather*}
Recall that this equals $ \mathrm{tr} \big( \big( \Lambda_G^{-1} {\rm d} \Lambda_G\big)^{4k+1} \big)$.
\end{defn}

The general properties stated in Section~\ref{sect: defbetabasic} imply the following.
\begin{thm} \label{cor: omegaprojective} The differential forms $\omega^{4k+1}_G$ are well-defined, and give rise for all $k\geq 1$ to closed, projective differential forms
\begin{gather*}
\omega^{4k+1}_G \in \Omega^{4k+1}\big( \Pro^{|E_G|-1} \backslash X_G\big)
\end{gather*}
whose singularities lie along the graph hypersurface, where they have a pole of order at most $k+1$. In particular, they are smooth on the open simplex $\sigma_G$.
\end{thm}

\begin{proof}
The invariance of $\beta^{4k+1}$ (Lemma~\ref{lem: BasicPropertiesBeta}) implies that $\omega^{4k+1}_G$ is independent of the choice of bases which go into defining the Laplacian matrix $\Lambda_G$. The fact that $\omega^{4k+1}_G$ is closed follows from Lemma~\ref{lem: BasicPropertiesBeta}$(v)$.
 Since $\det(\Lambda_G)$ is by definition the graph polynomial $\Psi_G$, it is immediate from the definition of $\omega^{4k+1}_G$ and the formula for the inverse of a matrix in terms of its adjugate that
\begin{gather*}
\omega^{4k+1}_G = \frac{ N_G }{\Psi_G^{4k+1}} \qquad \text{for some} \quad
N_G \in \Omega^{4k+1}\big(\Q[x_e, e\in E_G]\big),
\end{gather*}
 where $N_G$ is a polynomial form of degree $(4k+1)h_G$. In particular, $\omega^{4k+1}_G$ is homogeneous of degree $0$. The order of the pole is given by \eqref{denomBetaSymmetric}.
 The projectivity of $\omega^{4k+1}_G$ follows from vanishing under contraction with the Euler vector field:
 \begin{gather*}
 \bigg(\sum_{e\in E_G} x_e \frac{\partial}{\partial x_e}\bigg) \omega^{4k+1}_{G}(x_e) = \frac{\partial}{\partial \lambda} \omega^{4k+1}_G( \lambda x_e) = \frac{\partial}{\partial \lambda} \beta^{4k+1}_{\lambda \Lambda_G} = \frac{\partial}{\partial \lambda} \beta^{4k+1}_{\Lambda_G}=0,
 \end{gather*}
 where the penultimate equality is Lemma~\ref{lem: projinv}.
\end{proof}

Note that since $\Lambda_G$ is symmetric, the forms $\beta^{4n+3}_{\Lambda_G}$ vanish for all $n\geq 0$. If $G$ has connected components $G_1,\dots, G_n$ then using Lemma~\ref{lem: BasicPropertiesBeta}$(vi)$, we have
\begin{gather*}
\omega^{4k+1}_{G} = \sum_{i=1}^n \omega^{4k+1}_{G_i}
\end{gather*}
since $\Lambda_G = \bigoplus_{i=1}^n \Lambda_{G_i}$ with respect to the decomposition $H_1(G;\Z) \cong \bigoplus_{i} H_1(G_i;\Z)$.

\begin{Example} \label{example: W3form} For $G= W_3$, the wheel with 3 spokes, Example~\ref{examples: smallbetas} gives
\begin{gather*}
\omega^5_{W_3} = 10 \frac{\Omega_{W_3}} {\Psi_{W_3}^2},
\end{gather*}
where $\Omega_{W_3} = \sum_{i=1}^6 (-1)^i x_i {\rm d}x_1 \cdots \widehat{{\rm d}x_i} \cdots {\rm d}x_6$. It is the Feynman differential form which computes the residue in dimensional regularisation in massless $\phi^4$ theory. In general, this is not true: the forms $\omega^{4k+1}_G$ have complicated numerators, which are strongly reminiscent of the kinds of numerators occurring in a gauge theory \cite{Golz}. It would be very interesting to interpret the canonical forms $\omega^{4k+1}_G$ more generally in terms of a suitable quantum field theory, or conversely, interpret the integrands which arise in the parametric representation of quantum electrodynamics, for instance, as matrix-valued differential forms in the spirit of Section~\ref{sect: defbetabasic}.
\end{Example}

\begin{rem} 
More generally, for any exact sequence $(S)$ Section~\ref{sect: Matroids} we may define
\begin{gather} \label{omega(S)}
\omega_S^{4k+1} = \beta^{4k+1}_{\Lambda_B},
\end{gather}
where the Laplacian matrix $\Lambda_B$ is relative to a choice of basis $B$ of $H$. The latter depends on the basis $B$ only up to the transformation \eqref{LambdaTransformationMatroids}, and since the forms $\beta^{4k+1}_{\Lambda_B}$ are invariant (Lemma~\ref{lem: BasicPropertiesBeta}), it follows that $\omega_S^{4k+1}$ is well-defined.
As a consequence, for any regular matroid~$M$, we may define a form
\begin{gather*} 
\omega_M^{4k+1},
\end{gather*}
which does not depend on the choice of representation of the matroid.
\end{rem}

\subsection{First properties} The forms $\omega_G^{4k+1}$ are invariant under automorphisms.

\begin{lem} \label{lem: autoinvariant} Consider any automorphism $\pi$ of a graph $G$. It induces a map $\pi^*\colon R^0_G \cong R^0_G$ which permutes the edge variables via $\pi^* x_e = x_{\pi(e)}$. Then
\begin{gather*}
\omega^{4k+1}_G = \pi^* \omega^{4k+1}_G .
\end{gather*}
\end{lem}
\begin{proof} The automorphism $\pi$ induces an automorphism $P$ of $H_1(G;\Q)$ and hence acts on the graph Laplacian via the formula $\pi^* \Lambda_{G} = P^{T} \Lambda_{G} P$.
The statement follows from the invariance of $\beta_{\Lambda_{G}}$ (Lemma~\ref{lem: BasicPropertiesBeta}).
\end{proof}

 The forms $\omega_{\bullet}^{4k+1}$ are compatible with contractions in the following sense.
First of all, if $\gamma$ is a subset of the set of edges of $G$, consider the linear subspace
 \begin{gather*}
 L_{\gamma} \subset \Pro^{E_G}
 \end{gather*}
 defined by the vanishing of the edge coordinates $x_e$ for all $ e\in E_{\gamma}$. It is canonically isomorphic to $\Pro^{E_{G/\gamma}}$.
 A basic property of graph polynomials with respect to contraction of edges implies that $\Psi_G$ vanishes along $L_{\gamma}$ if $h_{\gamma}>0$, but in the case $h_{\gamma}=0$, its restriction to $L_{\gamma}$ satisfies $\Psi_G\big|_{L_{\gamma}} = \Psi_{G/\gamma}$. Thus $L_{\gamma}$ is contained in the graph hypersurface $X_{G}$ if $h_{\gamma} >0$ but otherwise if $ h_{\gamma}=0$ one has
 \begin{gather*}
 L_{\gamma} \cap X_G = X_{G/\gamma},
 \end{gather*}
 via the canonical identification $L_{\gamma} = \Pro^{E_{G/\gamma}}$.
\begin{prop} \label{prop: restrictomega} Let $\gamma\subset E_G$ such that $h_{\gamma}=0$, i.e., $\gamma$ is a forest. Then
\begin{gather*}
\omega^{4k+1}_{G} \big|_{L_{\gamma}} = \omega^{4k+1}_{G/\gamma},
\end{gather*}
as meromorphic forms on $L_{\gamma} =\Pro^{E_{G/\gamma}}$. They are regular on the open complement of the graph hypersurface $L_{\gamma} \backslash (L_{\gamma} \cap X_{G}) = \Pro^{E_{G/\gamma}} \backslash X_{G/\gamma}$.
\end{prop}
\begin{proof} Since $\omega_G^{4k+1}$ is regular at the generic point of $L_{\gamma}$, and likewise for $L_I$ for all $I \subset E_{\gamma}$, the statement for a general forest $\gamma$ can be proved by contracting one edge in $\gamma$ at a time. We can thus assume that $\gamma$ consists of a single edge $e$. Since in this case $L_{\gamma} $ is the hyperplane defined by $x_e=0$, it suffices to show that
\begin{gather} \label{inproof: omegarestricte}
\omega^{4k+1}_G \big|_{x_e=0} = \omega^{4k+1}_{G/e}.
\end{gather}
By assumption, $e$ has distinct endpoints, and therefore contraction of the edge $e$ defines an~isomorphism $H_1(G;\Z) \cong H_1(G/e;\Z)$. By definition of the graph Laplacian matrix,
 $\Lambda_{G/e} = \Lambda_{G} \big|_{x_e=0}$ from which \eqref{inproof: omegarestricte} immediately follows.
 \end{proof}

The restriction of $\omega^{4k+1}_{G}$ to a linear subspace $L_{\gamma}$, where $h_{\gamma}>0$, is not defined. This is because~$L_{\gamma}$ is contained in $X_G$, along which $\omega^{4k+1}_{G}$ may have poles.
 \subsection{Further graph-theoretic properties}

\subsubsection{Duality and deletion of edges }

 \begin{lem}[duality] \label{lem: dual}
 Let $G$ be a graph and $\widecheck{G}$ the dual $($cographic$)$ matroid. Then
 \begin{gather*}
 \omega^{4k+1}_{\widecheck{G}} = i^* \omega^{4k+1}_{G}
 \end{gather*}
 for all $k \geq 1$,
 where $i$ is the involution $i\colon x_e \mapsto x_e^{-1} $. This relation holds, in particular, if $G$ is a~planar graph and $\widecheck{G}$ a planar dual.
 \end{lem}

 \begin{proof} This holds more generally for the form \eqref{omega(S)} associated to an exact sequence and its dual, by \eqref{DdecompSdualS}. The latter, together with Lemma~\ref{lem: BasicPropertiesBeta}, implies that
 \begin{gather*}
 \omega^{4k+1}_D = \omega^{4k+1}_{\Lambda_S} + i^* \omega^{4k+1}_{\Lambda^{-1}_{S^{\vee}}} = \omega^{4k+1}_{\Lambda_S} - i^* \omega^{4k+1}_{\Lambda_{S^{\vee}}}.\end{gather*}
 The form $ \omega^{4k+1}_D$ vanishes for $k\geq 1$.
 In particular, the statement holds for any regular matroid~$M$ and its dual $M^{\vee}$, and in particular for graphs, whose matroids are regular.
 \end{proof}

 \begin{cor}[deletion of edges]
 Let $G$ be a graph. Then
 \begin{gather*}
 \omega^{4k+1}_{G\backslash e} = \big( i_e^* \omega^{4k+1}_{G}\big)\big|_{x_e=0},
 \end{gather*}
 where $i_e (x_f) = x_f$ if $f\neq e$ and $i_e(x_e)= x_e^{-1}$. Informally, $\omega^{4k+1}_{G\backslash e}$ is the coefficient of $x_e^n $ in
 $\omega^{4k+1}_G$ of highest degree $n$.
 \end{cor}
 \begin{proof} Deletion of an edge is dual to contraction of the correponding edge in the dual matroid. The statement then follows from the previous lemma and
\eqref{inproof: omegarestricte}.
 \end{proof}

 \subsubsection{Series-parallel operations (dividing and doubling edges)}

\begin{lem}[series] \label{lem: series}
 Let $G'$ denote the graph obtained from $G$ by replacing an edge $e$ with two edges $e'$, $e''$ in series $($subdividing $e$ with a two-valent vertex$)$. Then
\begin{gather*}
 \omega^{4k+1}_{G'} = s_{e}^* \omega^{4k+1}_{G},
 \end{gather*}
where $s_e\colon R^{\bullet}_G \rightarrow R^{\bullet}_{G'}$ is the map
\begin{gather} \label{sedef}
s_e x_f = \begin{cases}
x_f &\text{if} \ f \neq e,
\\
 x_{e'}+x_{e''}& \text{if} \ f =e.
 \end{cases}
\end{gather}
\end{lem}
\begin{proof}
A representative for the graph Laplacian matrix $\Lambda_{G'}$ is obtained from $\Lambda_G$ by repla\-cing~$x_e$ with $x_{e'}+ x_{e''}$ from which the result immediately follows.
\end{proof}

\begin{lem}[parallel] \label{lem: parallel}
 Let $G'$ denote the graph obtained from $G$ by replacing an edge $e$ with two edges $e'$, $e''$ in parallel $($duplicate the edge $e)$. Then
\begin{gather*}
 \omega^{4k+1}_{G'} = p_e^* \omega^{4k+1}_{G}
 \end{gather*}
for all $k\geq 1$,
where $p_e= i s_e i$ is the map
\begin{gather} \label{pedef}
p_e x_f = \begin{cases}
x_f &\text{if} \ f \neq e,
\\
\big(x^{-1}_{e'}+x^{-1}_{e''}\big)^{-1} &\text{if}\ f =e .
\end{cases}
\end{gather}
\end{lem}
\begin{proof}
Let $\widecheck{G}$ be the matroid dual to $G$. Contracting an edge on $G$ corresponds to deleting an~edge in $\widecheck{G}$ and vice versa.
 Since subdividing and duplicating edges are uniquely characterised in terms of contractions and deletions, one verifies that subdivision of an edge $e\in G$ is dual to the operation of duplicating the edge $e \in \widecheck{G}$. It follows from Lemmas~\ref{lem: dual} and \ref{lem: series} that $\omega^{4k+1}_{G'} = (-1)^2 p_e^* \omega^{4k+1}_{G} $, where $p_e^* =i^* s_e^* i^*$, which leads to the stated formula for $p_e^*$. \end{proof}

\begin{rem}
For $k=0$, the form $\beta^1$ is not projectively invariant and the relation needs to be modified:
$\omega^1_{G'} = p_e^* \omega^1_{G} + {\rm d}\log(x_{e'}+ x_{e''}) $. It is equivalent to the formula
$\Psi_{G'} = (x_{e'}+x_{e''} ) p_e^* \Psi_G$ (e.g., \cite[Lemma 18]{PeriodsFeynman}) via $\omega^1_G = {\rm d} \log \Psi_G$.
\end{rem}

Feynman integrals are known to satisfy a whole range of graph-theoretic identities \cite{BroadhurstKreimer, PeriodsFeynman, Schnetz}, and one can ask whether these identities hold on the level of the forms $\omega^{4k+1}_G$. Here we mention just two of the most simple ones.
\begin{lem} \label{lem: 1vertexjoin} Let $G$ be a $1$-vertex join of $G_1$ and $G_2$. Then
\begin{gather*}
\omega^{4k+1}_G = \omega^{4k+1}_{G_1} +\omega^{4k+1}_{G_2} .
\end{gather*}
\end{lem}
\begin{proof} Since $H_1(G;\Z) = H_1(G_1;\Z) \oplus H_1(G_2;\Z)$, it follows from Lemma~\ref{lem: BasicPropertiesBeta}$(vi)$ that $\Lambda_G = \Lambda_{G_1} \oplus \Lambda_{G_2}$ with respect to $\Q^{E_G} = \Q^{E_{G_1}} \oplus \Q^{E_{G_2}}$.
\end{proof}

\begin{lem} Let $G$ and $G'$ be any two graphs with a pair of distinguished vertices $\{v_1,v_2\}$ and $\{v'_1,v'_2\}$. There are two ways of joining these graphs together by gluing either $v_i$ with $v_i'$ $($or $v_i$ with $v'_{3-i})$ for $i=1,2$ to obtain two $2$-vertex joins $G_1$ and $G_2$. Their canonical differential forms are equal:
$\omega^{4k+1}_{G_1} = \omega^{4k+1}_{G_2}.$
\end{lem}
\begin{proof} By Whitney, the matroids associated to $G_1$ and $G_2$ are isomorphic, so $\Lambda_{G_1}$ is equivalent to $\Lambda_{G_2}$.
\end{proof}

\begin{rem} \label{rem: 2VJ} The operation in the lemma is not to be confused with the 2-vertex join $G_1: G_2$, for which we assume in addition that $\{v_1,v_2\}$ (respectively $\{v_1',v_2'\}$) are connected by an edge~$e$ (resp.~$e'$). It is defined by joining together $G_1$, $G_2$ by identifying $v_1=v_1'$ and $v_2=v_2'$ and deleting the edges $e$, $e'$.
\end{rem}

\subsection{The Hopf algebra of canonical differential forms} Let us write
$ \Omega^0_{\can}=\Z$, generated by the constant form $1$ of degree zero.

\begin{defn} Let $\Omega^{\bullet}_{\mathrm{can}} = \bigoplus_{{\rm d}\geq 0} \Omega^d_{\can}$ denote the graded exterior algebra over $\Z$ generated by symbols $\beta^{4k+1}$ for $k\geq 1$. We can equip $\Omega^{\bullet}_{\can}$ with a coproduct
\begin{gather*}
\Delta\colon\ \Omega^{\bullet}_{\can} \To \Omega^{\bullet}_{\mathrm{can}} \otimes_{\Z} \Omega^{\bullet}_{\can}
\end{gather*}
such that each generator $\beta^{4k+1}$ is primitive:
 $\Delta \beta^{4k+1} = \beta^{4k+1} \otimes 1 + 1 \otimes \beta^{4k+1}$.
\end{defn}
Note that the coproduct is the same as that defined on the infinite general linear group \eqref{betaprimitive}.
An element $\omega \in \Omega^n_{\can}$ is primitive if and only if $n= 4k+1$ for some $k\geq 1$ and $\omega$ is proportional to $\beta^{4k+1}$.

\begin{Example} 
The smallest degrees $k$ for which $\Omega^k_{\can}$ is non-zero are
\begin{gather*}
0,\quad 5,\quad 9,\quad 13,\quad 14,\quad 17,\quad 18,\quad 21.
\end{gather*}
The space $\Omega^{22}_{\can}$ has rank 2, generated by $\beta^{5} \wedge \beta^{17}$ and $\beta^{9} \wedge \beta^{13}$. One has, for example,
\mbox{$\Delta_{\can} \big(\beta^5 \wedge \beta^9\big) = 1 \otimes \big(\beta^5 \wedge \beta^9\big) + \beta^5 \otimes \beta^9 - \beta^9 \otimes \beta^5 + \big(\beta^5 \wedge \beta^9\big) \otimes 1$}.
 \end{Example}

 Any element $\omega \in \Omega^{k}_{\can}$ defines a universal differential $k$-form which to any connected graph~$G$ assigns the projective differential form
\begin{gather*}
G \mapsto \omega_G \in \Omega^k \big(\Pro^{E_G} \backslash X_G\big).
\end{gather*}
It automatically vanishes on any graph with $k$ edges or fewer since there are no projective invariant differential forms of degree $k$ in $\leq k$ variables.
By Lemma~\ref{lem: autoinvariant} any canonical form $\omega$ is invariant under automorphisms of~$G$. A canonical form $\omega$ satisfies the functoriality properties which are deduced from those for primitive canonical forms by taking exterior products (for example, Proposition~\ref{prop: restrictomega} holds verbatim for any $\omega \in \Omega^{\bullet}_{\can}$). We leave the statements to the reader.

\begin{defn} Every canonical form defines universal cohomology classes in the cohomology of graph hypersurface complements. For all $\omega \in \Omega^{k}_{\can}$, we obtain a class
\begin{gather*}
[\omega_G] \in H_{{\rm d}R}^k \big(\Pro^{E_G} \backslash X_G\big)
\end{gather*}
in algebraic de Rham cohomology \cite{Grothendieck}, for every graph $G$.
\end{defn}

 \begin{rem} \label{rem: degreexe}
 Let $\omega$ be a canonical form of degree $k$. Suppose that $G$ satisfies $e_G = k+1$.
 Suppose that the order of the pole in the denominators of $\omega_G$ and $\omega_{\widecheck{G}}$ are bounded by $n$ (such an $n$ depends only on $\omega$ by Theorem~\ref{thm: optimaldenom}). The projective invariance of $\omega$, together with Lemma~\ref{lem: dual}, which implies that $\omega_G =i^*\big(\omega_{\widecheck{G}}\big)$, gives
 \begin{gather*}
 \omega_G = \frac{P_G}{\Psi_G^{n}} \Omega_G, \qquad \text{where} \quad \Omega_G= \sum_{i} (-1)^i x_i \,{\rm d}x_1 \wedge \cdots \wedge\widehat{{\rm d}x_i} \wedge \cdots \wedge {\rm d}x_{e_G},
 \end{gather*}
 where $P_G$ is a polynomial in $\Q[x_e]$ of degree at most $n-1$ in each variable $x_e$.
 \end{rem}

\subsection{Vanishing properties}
We now consider the case of most interest, namely when the dimension of the simplex $\sigma_G$ equals the degree of the form $\omega_G$, i.e.,
\begin{gather*}
e_G = \deg (\omega_G)+1 .
\end{gather*}

\begin{prop} \label{prop: vanishing} \samepage Let $\omega \in \Omega_{\can}^{k}$ of degree $k$. Then for any graph $G$ with $k+1$ edges the form~$\omega_G$ vanishes
 if one of the following holds:
\begin{enumerate}\itemsep=0pt
\item[$(i)$] $G$ has a vertex of degree $\leq 2$,

\item[$(ii)$] $G$ has a multiple edge,

\item[$(iii)$] $G$ has a tadpole,

\item[$(iv)$] $G$ is one-vertex reducible $($can be disconnected by deleting a vertex$)$,

\item[$(v)$] $G$ has a bridge $($can be disconnected by deleting an edge$)$. Thus in this situation, $\omega_G$~va\-nishes unless $G$ is core or ``$1$-particle irreducible''.
\end{enumerate}
\end{prop}

\begin{proof} In the cases $(i)$ and $(ii)$, $G$ is obtained from a graph $G'$ with $k$ edges by either duplicating or subdividing an edge $e$. Then, by Lemmas~\ref{lem: series} and~\ref{lem: parallel},
\begin{gather*}
\omega_G = f^* \omega_{G'},
\end{gather*}
where $f= s_e$ \eqref{sedef} in the case $(i)$ and $f=p_e$ \eqref{pedef} in the case $(ii)$. The differential form $\omega_{G'}$ is projective of degree $k$ in $k$ variables and therefore $\omega_{G'}$ vanishes, as does $\omega_G$.
The statement $(iii)$ is a special case of $(iv)$.
Suppose that $G$ is a one-vertex join of two graphs $G_1$ and $G_2$. Using Sweedler's notation we can write
\begin{gather*}
\Delta_{\can} \omega = \sum \omega' \otimes \omega'' .
\end{gather*}
Then by Lemma~\ref{lem: 1vertexjoin} and multiplicativity of the coproduct we have:
\begin{gather*}
\omega_G = \sum \omega'_{G_1} \wedge \omega''_{G_2},
\end{gather*}
 where each term satisfies $\omega' \in \Omega_{\can}^{k_1}$ and $\omega'' \in \Omega_{\can}^{k_2}$ for some $k_1+ k_2 = k$. Since $e_{G_1} + e_{G_2} = k+1$ we must have $e_{G_i} \leq k_i$ for some $i=1,2$, which implies that $\omega_{G_i}$ vanishes for the same reasons as above. Therefore $\omega_G$ is zero.

When $G$ has a bridge $e$, let $G_1$, $G_2$ denote the two connected components of $G\backslash e$. In this situation $\Lambda_G = \Lambda_{G_1} \oplus \Lambda_{G_2}$ as in Lemma~\ref{lem: 1vertexjoin}, and the proof proceeds as for a one-vertex join~$(iv)$ except that the equality
$e_{G_1} +e_{G_2} = k$ holds.
\end{proof}

\begin{cor} 
Let $\omega \in \Omega^n_{\can}$ be of degree $n$ and suppose that $G$ is a connected graph with $e_G=n+1$ edges and $h_G$ loops. Then $\omega_G$ vanishes unless
\begin{gather*} 
h_G \geq \frac{e_G}{3} + 1.
\end{gather*}
If $G$ is not three regular, then $\omega_G$ vanishes unless
 $h_G > \frac{e_G}{3} +1$.
 \end{cor}
\begin{proof}
Let $d = 2e_G/ v_G$ be the average degree of the vertices in $G$. By the previous proposition, $\omega_G$ vanishes unless every vertex in $G$ has degree $\geq 3$. Therefore $d\geq 3$ with equality if and only if $G$ is three-regular. We deduce that
\begin{gather*}
h_G -1 = e_G - v_G \geq e_G - \frac{2}{d} e_G = \frac{d-2}{d} e_G
\end{gather*}
from which the statement follows.
\end{proof}

\subsection{Variants} Since there are several possible formulations of Laplacian matrices associated to graphs, it is natural to ask if the associated invariant forms lead to the same differential forms. We show that they do.

\begin{lem} \label{lem: betaviaLG} Let $L_G$ be a matrix \eqref{LGdefn}. Then, for all $k\geq 1$,
\begin{gather*}
\beta^{4k+1}_{L_G} = \beta^{4k+1}_{\Lambda_G} .
\end{gather*}
\end{lem}
\begin{proof}
From Lemmas \ref{lem: DGdirectsum} and \ref{lem: BasicPropertiesBeta}$(vi)$, we have
\begin{gather*}
\beta_{D_G}^n = \beta_{\Lambda_G}^n + \beta^n_{L_G^{-1}}.
\end{gather*}
Let $n>1$. Then $\beta_{D_G}^n=0$, and Lemma~\ref{lem: BasicPropertiesBeta}$(ii)$ implies that
$\beta^n_{L_G} = (-1)^{n+1} \beta^n_{\Lambda_G}$.
\end{proof}

We now turn to the graph matrix defined in Section~\ref{sect: GraphMatrix}.

\begin{prop} Let $M_G$ be any choice of graph matrix. Then for all $k\geq 1$,
\begin{gather*}
 \beta^{4k+1}_{M_G} = \beta^{4k+1}_{\Lambda_G} .
 \end{gather*}
\end{prop}
\begin{proof}
By Lemma~\ref{lemMGLBU} we may write $M_G= L B U $,
where $L$, $B$, $U$ are block lower triangular, diagonal and upper triangular respectively.
Using the notation of Section~\ref{sect: DecompBlockMatrix} we set $\mathcal{L} = B^{-1} L^{-1} {\rm d}L. B$, $\mathcal{B} = B^{-1} {\rm d}B$, and
$\mathcal{U} = {\rm d}U . U^{-1}$, where
\begin{gather*}
{\rm d} L= \left(\begin{array}{c|c}
 0 & 0 \\ \hline
 \varepsilon_G {\rm d} D_G^{-1^{\vphantom{1}}} & 0 \\
\end{array}\right)\!, \qquad
{\rm d} B = \left(\begin{array}{c|c}
 {\rm d} D_G & 0 \\ \hline
 0 & {\rm d} L_G \\
\end{array}\right) \!, \qquad
{\rm d} U = \left(\begin{array}{c|c}
 0 & - {\rm d} \big(D^{-1}_G\big) \varepsilon^T_G \\ \hline
 0 & 0 \\
\end{array}\right)\!.
\end{gather*}
 Since $D_G$ is diagonal, ${\rm d}D^{-1}_G . {\rm d}D_G=0$ and
${\rm d}L. {\rm d}B = {\rm d}B. {\rm d}U = {\rm d}L . B. {\rm d}U = 0$. From this we deduce that
\begin{gather*}
 \mathcal{L} \mathcal{B} = \mathcal{B} \mathcal{U} = \mathcal{L} \mathcal{U}=0.
 \end{gather*}
Since also $\mathcal{L}^2= \mathcal{U}^2 =0$ we deduce that
\begin{gather*}
 \left( \mathcal{L} + \mathcal{B} + \mathcal{U}\right)^n = \mathcal{B}^n + \mathcal{B}^{n-1} \mathcal{L} + \mathcal{U} \mathcal{B}^{n-1} + \mathcal{U} \mathcal{B}^{n-2} \mathcal{L}.
 \end{gather*}
By cyclicity, the traces of all terms on the right-hand side vanish except for the first, and therefore
$ \tr (\omega_{M_G}^n) = \tr ( \mathcal{B}^n)$.
By Lemma~\ref{lem: BasicPropertiesBeta}$(vi)$ we deduce that
\begin{gather*} 
\beta^n_{M_G} = \beta^n_{D_G} + \beta_{L_G}^n .
\end{gather*}
The term $ \beta^n_{D_G}$ vanishes for $n>1$ and we conclude using the previous lemma.
\end{proof}

The previous proposition leads to closed formulae for the canonical differential forms $\omega_G$ in terms of graph polynomials and their ``Dodgson'' variants (Definition \ref{defn: Dodgson}).
If we define $\eta_{G}$ to be the $(E_G \times E_G)$ square matrix
 \begin{gather*}
 (\eta_{G})_{ij} = \bigg(\frac{\Psi_G^{ij}}{\Psi_G} \, {\rm d} x_j\bigg), \qquad 1 \leq i \leq j \leq E_G
 \end{gather*}
then by writing the inverse of a matrix in terms of its adjugate matrix, we have
\begin{gather*} 
\mu_{M_G} =
\begin{pmatrix}
\eta_{G} & 0 \\ 0 & 0
\end{pmatrix}
\end{gather*}
in block matrix notation. From this we deduce:

\begin{cor}The canonical form is given by
\begin{gather*}
\omega^{4k+1}_G = \tr \big(\eta_G^{4k+1}\big).
\end{gather*}
As a consequence, it can be written as a polynomial in $\frac{\Psi^{ij}_G}{\Psi_G}$ and ${\rm d}x_j$.
\end{cor}
From this one can write down a closed formula for $\omega^{4k+1}_G$ as a sum over permutations involving products of Dodgson polynomials. For example,
\begin{gather*} 
\beta^5_{M_G} = 10\! \sum_{I \subset E_G} \sum_{\sigma \in \mathrm{Dih}(I)} \frac{\Psi_G^{i_{\sigma_1} i_{\sigma_2}} }{\Psi_G} \frac{\Psi_G^{i_{\sigma_2} i_{\sigma_3}}}{\Psi_G} \frac{\Psi_G^{i_{\sigma_3} i_{\sigma_4}}}{\Psi_G} \frac{\Psi_G^{i_{\sigma_4} i_{\sigma_5}}}{\Psi_G} \frac{\Psi_G^{i_{\sigma_5} i_{\sigma_1}}}{\Psi_G} {\rm d}x_{i_{\sigma_1}} \cdots {\rm d}x_{i_{\sigma_5}},
\end{gather*}
where the sum is over all subsets $I=(i_1,\dots, i_5) \subseteq E_G$, and $\mathrm{Dih}(I) \cong \Sigma_5 /D_{10}$ is the set of dihedral orderings of $I$ (the twelve ways of writing the elements of $I$ around the vertices of pentagon, up to dihedral symmetries). This formula easily generalises, but is of limited practical use because of the sheer number of terms.

\begin{rem}
Using condensation identities (e.g., \cite[Sections~2.4--2.5]{PeriodsFeynman}) which are based on results of Dodgson and Leibniz, we can show that
\begin{gather*}
\beta^5_{M_G} = 10 \sum_{I \subset E_G} \bigg(\frac{\Psi_G^{i_1 i_2 i_3 , i_1i_4i_5}}{\Psi_G}
\frac{\Psi_G^{i_2 i_4, i_3i_5}}{\Psi_G}- \frac{\Psi_G^{i_1 i_3 i_5 , i_1i_2i_4}}{\Psi_G} \frac{\Psi_G^{i_2 i_3, i_4i_5}}{\Psi_G} \bigg) {\rm d}x_{i_1} \cdots {\rm d}x_{i_5},
\end{gather*}
which gives the optimal power of $\Psi_G$ in the denominator (Theorem~\ref{thm: optimaldenom}).
This phenomenon is very reminiscent of the cancellations which occur in the parametric formulation of quantum electrodynamics \cite{Golz} and suggests a matrix formulation of the latter. It also suggests a possible formulation of canonical graph forms using generalised Gaussian integrals.
\end{rem}

\section{Algebraic compactification of the space of metric graphs}
We construct an algebraic compactification of the space of metric graphs by blowing up, and define an algebraic differential form upon it to be an infinite collection of differential forms of~the same degree which satisfy certain compatibilities.
We then prove that the pull-backs of~canonical forms along the blow up satisfy all these compatibilities.

\subsection{Reminders on linear blow ups in projective space}
For any subset of edges $I \subset E_G$, recall that $L_I \subset \Pro^{E_G}$ denotes the linear space defined by the vanishing of the coordinates $x_e$ for all $e\in I$.

Consider subsets $\mathcal{B}_G\subset 2^{E_G}$ of sets of edges of $G$ with the properties:
\begin{enumerate}\itemsep=0pt
\item[$(i)$] $E_G \in \mathcal{B}_G$,

\item[$(ii)$] $I_1, I_2 \in \mathcal{B}_G \ \Longrightarrow \ I_1 \cup I_2 \in \mathcal{B}_G$.
\end{enumerate}
Furthermore, we require the assignment $G\mapsto \mathcal{B}_G$ to satisfy various properties including $\mathcal{B}_{\gamma} = \{I\in \mathcal{B}_{G} \colon I\subset E_\gamma\}$ for all subgraphs $\gamma \subset G$, and a similar property for quotients $G/\gamma$, for which we refer to~\cite[Section~5.1]{Cosmic}.
Examples of interest satisfying all the required properties include~$\mathcal{B}^{\mathrm{core}}_G$, consisting of all core subgraphs (the minimal case of relevance), or $\mathcal{B}^{\mathrm{all}}_G$ consisting of all subgraphs (the maximal case). We shall fix some such family of $\mathcal{B}_{G}$ once and for all. For the present application to canonical graph forms, $\mathcal{B}^{\mathrm{core}}_G$ suffices, but one can imagine situations where one should take $\mathcal{B}^{\mathrm{all}}_G$, for instance if one were to consider differential forms with a more complicated polar locus. We shall simply take $\mathcal{B}_G=\mathcal{B}^{\mathrm{core}}_G$ from now on.

 For any graph $G$, let
\begin{gather} \label{piGdefn}
\pi_G\colon\ P^{G} \To \Pro^{E_G}
\end{gather}
denote its iterated blow-up along linear subspaces $L_{\gamma}$ corresponding to $\gamma \in \mathcal{B}_G$ in increasing order of dimension \cite{BEK}, \cite[Definition 6.3]{Cosmic}. It does not depend on any choices.
 It is equipped with a divisor $D \subset P^{G}$
 \begin{gather*}
 D = \pi_G^{-1}\bigg(\bigcup_{e\in E_G} L_{e}\bigg),
 \end{gather*}
which is the total inverse image of the coordinate hyperplanes. Its irreducible components are of two types: the strict transforms $D_e$ of coordinate hyperplanes $x_e=0$, which are in one-to-one correspondence with the edges of $G$, and the inverse images of $L_{\gamma}$, for every $\gamma \in \mathcal{B}_G$ with $|\gamma| \geq 2$, which we denote by $D_{\gamma}$. Let
 \begin{gather*}
 \widetilde{\sigma}_G = \overline{\pi_G^{-1} ( \sigma_G )}
 \end{gather*}
 denote the closure, in the analytic topology, of the inverse image of the open coordinate simplex~$\sigma_G$. It is a compact manifold with corners which we have in the past called the Feynman polytope. The following theorem was first proved in \cite{BEK} for primitive-divergent graphs (for more general Feynman graphs, including those with arbitrary kinematics and masses, see \cite[Theorem~5.1]{Cosmic}).

\begin{thm} \label{thm: NormalCrossing} The divisor $D\subset P^G$ is simple normal crossing. Every irreducible component is canonically isomorphic to a space of the same type:
\begin{gather*}
D_e = P^{G/ e} \qquad \text{and} \qquad D_{\gamma} \cong P^{\gamma} \times P^{G/\gamma} .
\end{gather*}
The strict transform $Y_G \subset P^G$ of the graph hypersurface $X_G \subset \Pro^{E_G}$ does not meet $\widetilde{\sigma}_G$. Its intersection with the divisor $D$ satisfies:
\begin{gather*}
Y_G \cap D_e \cong Y_{G/e} \qquad \text{and} \qquad
Y_G \cap D_{\gamma} \cong \big(P^{\gamma} \times Y_{G/\gamma}\big) \cup \big(Y_{\gamma} \times P^{G/\gamma}\big) .
\end{gather*}
\end{thm}
In particular, the complements of the strict transform of the graph hypersurface in each boundary component $D_{\gamma}$ satisfy the product structure:
\begin{gather} \label{Dgammaproduct}
D_{\gamma} \backslash (D_{\gamma} \cap Y_{G}) \cong \big( P^{\gamma} \backslash Y_{\gamma} \big) \times \big(P^{G/\gamma} \backslash Y_{G/\gamma} \big) .
\end{gather}
This product structure is fundamental to both the existence of the renormalisation group \cite{BrownKreimer} and also the coaction principle \cite{Cosmic}. We call the maps
\begin{align*}
P^{G/e} \cong D_e & \Longhookrightarrow P^{G},
\\
P^{\gamma} \times P^{G/\gamma} \cong D_{\gamma} &\Longhookrightarrow P^{G}
\end{align*}
\emph{face maps}, since they induce inclusions of faces on the polytope $\widetilde{\sigma}_G$.
It is clear that the assignment $G\mapsto \big(P^G, D\big)$ is clearly functorial in $G$ with respect to graph isomorphisms.

\subsection[Differentials on the total space P Tot]
{Differentials on the total space $\boldsymbol{P^{\mathrm{Tot}}}$}
If $G$ has several connected components $G= \bigcup_{i=1}^n G_i$, let us set $P^G = P^{G_1} \times \dots \times P^{G_n}$.

Let us define the total space $P^{\mathrm{Tot}}$ to be the collection of schemes $\big(P^G\big)_G$ as $G$ ranges over all graphs, together with morphisms
\begin{align}
i_{e}\colon \ &\hphantom{P^{\gamma} \times{}\, }P^{G/e}\To P^G,\nonumber
\\
i_{\gamma}\colon \ &P^{\gamma} \times P^{G/\gamma} \To P^G\label{facemaps}
\end{align}
by taking products of face maps for every connected component of $G$. Every isomorphism $\tau\colon G\cong G'$ induces an isomorphism
\begin{gather}
\tau\colon\ P^{G} \cong P^{G'} . \label{tauiso}
\end{gather}
If $G$ has connected components $G_1,\dots, G_n$, define
\begin{gather*}
\widetilde{\sigma}_G = \prod_{i=1}^n \widetilde{\sigma}_{G_i}.
\end{gather*}
 An orientation on $G$ is equivalent to an orientation of each $\sigma_{G_i}$ and hence $\widetilde{\sigma}_G$.

\begin{defn} \label{defn: formsonPtot} Define a \emph{primitive algebraic differential form} $\{\widetilde{\omega}\}$ of degree $k$ on $P^{\mathrm{Tot}}$ to be a~collection of differential forms $\widetilde{\omega}_G$, for every $G$, such that:
\begin{enumerate}\itemsep=0pt

\item For all $G$, the form $\widetilde{\omega}_G$ is projective and meromorphic on $P^G$ of degree $k$, and its restriction to $\widetilde{\sigma}_G$ is smooth (i.e., its poles lie away from $\widetilde{\sigma}_G$).

\item Its restriction along face maps \eqref{facemaps} satisfies the compatibilities:
\begin{gather*}
i_e^* \widetilde{\omega}_G = \widetilde{\omega}_{G/e},
\\
i_{\gamma}^* \widetilde{\omega}_{G} = \widetilde{\omega}_{\gamma} \wedge 1 + 1\wedge \widetilde{\omega}_{G/\gamma},
\end{gather*}
where, by abuse, $\widetilde{\omega}_{\gamma} $ denotes the pull-back along the projection onto the first component $P^{\gamma} \times P^{G/\gamma} \rightarrow P^{\gamma}$, and similarly for $\widetilde{\omega}_{G/\gamma}$. The collection of forms $\widetilde{\omega}$ is
 also required to be compatible with isomorphisms \eqref{tauiso}:
\begin{gather*}
\tau^* \widetilde{\omega}_{G'} = \widetilde{\omega}_{G} \qquad \text{for all}\quad \tau\colon\ G \cong G'.
\end{gather*}
\end{enumerate}
An algebraic differential form $\{\widetilde{\omega}\}$ of degree $k$ on $P^{\mathrm{Tot}}$ is then defined to be an exterior product of primitive forms. Note that this will affect the formula for the restriction $i_{\gamma}^*$, but all other properties remain unchanged.
\end{defn}

The differential is defined component-wise: ${\rm d} \{\widetilde{\omega}\} = \{ {\rm d} \widetilde{\omega}_G\}_G$. One can clearly define various sheaves of differentials on $P^{\mathrm{Tot}}$, but the above ``global'' definition is adequate for our purposes.
An algebraic differential form restricts to a smooth form
$\widetilde{\omega}_G\big|_{ \widetilde{\sigma}_G} $ of degree $k$
on the polytope $\widetilde{\sigma}_G$, for every $G$.

\begin{rem} Instead of $P^{\mathrm{Tot}}$ we may also consider the collection of schemes $\big(P^G\big)_G$, where $G$ ranges over all graphs of bounded genus $\leq g$, equipped with the face maps. In this case, the topological space given by the collection of $\widetilde{\sigma}_G$, together with the identifications induced by face maps and automorphisms, is closely related to the quotient of the bordification~\cite{BordificationOuterSpace} of Outer space $\mathcal{O}_g$ by the action of $\mathrm{Out}(F_g)$. \end{rem}

\subsection{Canonical forms along exceptional divisors}
Let $\omega \in \Omega^n_{\can}$ be a canonical form.
Denote the exceptional divisor of \eqref{piGdefn} by $\mathcal{E} \subseteq D \subseteq P^G$ and define
\begin{gather*} 
\widetilde{\omega}_G \in \Omega^{n} \big(P^G \backslash (\mathcal{E} \cup Y_G)\big)
\end{gather*}
to be the smooth differential form $ \pi^*_G (\omega_G)$ for any connected $G$, where $\pi_G$ is the blow-up \eqref{piGdefn}. It could \emph{a priori} have poles along components of the exceptional locus $\mathcal{E}$. In fact, this is never the case, even if $G$ has subgraphs $\gamma$ which are called ``divergent'' in physics terminology
 (meaning that they satisfy $h_{\gamma} \geq 2 e_{\gamma}$).

\begin{thm} \label{thm: nopolesonD}
The form $\widetilde{\omega}_G$ has no poles along the divisor $D$ and therefore extends to a smooth form on $P^G \ \backslash \ Y_G$, i.e.,
\begin{gather*}
\widetilde{\omega}_G \quad \in \quad \Omega^{n} \left(P^G \ \backslash \ Y_G \right).
\end{gather*}
Its restrictions to irreducible boundary components of $D$ satisfy
\begin{gather*}
\widetilde{\omega}_G\big|_{D_e} = \widetilde{\omega}_{G/e}
\end{gather*}
if $D_e$ is the strict transform of the hyperplane $L_e$ corresponding to a single edge $e$ of $G$, and in the case when $D_{\gamma}$ is an exceptional component corresponding to a core subgraph $\gamma \subset G$, satisfy
\begin{gather} \label{RestrictomegaToDgammaAndCoproduct}
 \widetilde{\omega}_G\big|_{D_{\gamma}} = \sum \widetilde{\omega}'_{\gamma}\wedge \widetilde{\omega}''_{G/\gamma},
 \end{gather}
where $\Delta_{\can} \omega = \sum \omega' \otimes \omega''$ in Sweedler notation.
The forms on the right-hand side of this formula are viewed on $D_{\gamma} \backslash (D_{\gamma} \cap Y_{G})$ via the isomorphism \eqref{Dgammaproduct}.
\end{thm}

\begin{proof} It is enough to prove the statement for $\omega = \beta^{4k+1}$ a primitive form in $\Omega^{4k+1}_{\can}$. The fact that $\widetilde{\omega}^{4k+1}_G$ has no poles along an irreducible component of the form $D_e$, and the formula for its restriction,
are a consequence of Proposition~\ref{prop: restrictomega}. Now consider the case of an exceptional divisor $D_{\gamma}$, where $\gamma \subsetneq G$ is a core subgraph.
Local affine coordinates in a neighbourhood of $D_{\gamma} \cong P^{\gamma} \times P^{G/\gamma}$ \big(or, to be more precise, of $D_{\gamma}\backslash (D_{\gamma} \cap \mathcal{E}')$, where $\mathcal{E'}$ consists of all components of $\mathcal{E}$ not equal to $D_{\gamma}$, which is isomorphic to an open affine subset of $\Pro^{E_{\gamma}} \times \Pro^{E_{G/\gamma}}\big)$ are given by replacing $x_{e}$ with $x_{e} z$ for all $e\in E_{\gamma}$ \cite[Section~5.3]{Cosmic} and setting some $x_{e_0}=1$ for $e_0 \in E_{\gamma}$. In~these coordinates, the locus $D_{\gamma}$ is given by the equation $z=0$.

There is a decomposition of the homology
$H_1(G; \Z) \cong H_1(\gamma; \Z) \oplus H_1(G/\gamma;\Z)$
which is obtained by splitting the exact sequence
\begin{gather*}
0 \To H_1(\gamma;\Z) \To H_1(G;\Z) \To H_1(G/\gamma; \Z) \To 0 .
\end{gather*}
With respect to a suitable basis of this decomposition, the graph Laplacian matrix, in the local affine coordinates described above, can be written in block form
\begin{gather*}
\Lambda_G = \begin{pmatrix} z \Lambda_{\gamma} & z B \\ z C & D \end{pmatrix}\!, \qquad \text{where} \quad D \equiv \Lambda_{G/\gamma} \pmod{z}
\end{gather*}
and $\Lambda_{\gamma}$, $B$, $C$, $D$ are matrices whose entries are polynomials in the $x_e$, for $e \in E_G$.

We can therefore write the graph Laplacian in the form
\begin{gather*}
 \Lambda_G = \Lambda U, \qquad \text{where we set} \quad \Lambda = \begin{pmatrix} z \Lambda_{\gamma} & 0 \\ 0 & \Lambda_{G/\gamma} \end{pmatrix}\!,
 \end{gather*}
where the matrix $U$ is defined by $U= \Lambda^{-1}\Lambda_G$. It satisfies
\begin{gather*}
U= \begin{pmatrix} z^{-1} \Lambda^{-1}_{\gamma} & 0 \\ 0 & \Lambda^{-1}_{G/\gamma} \end{pmatrix} \begin{pmatrix} z \Lambda_{\gamma} & zB \\ zC & D \end{pmatrix} \equiv \begin{pmatrix} 1 & \Lambda_{\gamma}^{-1}B \\ 0 & 1 \end{pmatrix} \pmod{z} .
\end{gather*}
In particular, the entries of $\Lambda_G$, $\Lambda$ and $U$ have no poles at $z=0$. Since $\det(U) \equiv 1 \pmod{z}$, the inverse matrix $U^{-1}$ has entries which have no poles at $z=0$, and can be expressed as formal power series in $z$ whose coefficients are rational functions in the $x_e$, for $e \in E_G$. We have
\begin{gather*}
 \Lambda_G^{-1} {\rm d} \Lambda_G = U^{-1} \Lambda^{-1} ({\rm d} \Lambda) U + U^{-1} {\rm d}U
 \end{gather*}
and hence
\begin{gather*}
U \big(\Lambda_G^{-1} {\rm d}\Lambda_G\big) U^{-1} =\Lambda^{-1} {\rm d}\Lambda + {\rm d}U.U^{-1}.
\end{gather*}
We wish to compute
\begin{gather*}
 \beta_{\Lambda_G}^{4k+1} = \tr\big( U\big(\Lambda_G^{-1} {\rm d}\Lambda_G\big)^{4k+1}U^{-1} \big) = \tr\big(\big(\Lambda^{-1} {\rm d}\Lambda + {\rm d} U .U^{-1}\big)^{4k+1} \big) .
 \end{gather*}
Now observe that the matrix
\begin{gather*}
 \Lambda^{-1}{\rm d}\Lambda = \frac{{\rm d}z}{z} \begin{pmatrix} 1 & 0 \\ 0 & 0 \end{pmatrix} + \begin{pmatrix} \Lambda_{\gamma}^{-1} {\rm d} \Lambda_{\gamma} & 0 \\ 0 &
 \Lambda_{G/\gamma}^{-1} {\rm d} \Lambda_{G/\gamma}
 \end{pmatrix}
 \end{gather*}
is block diagonal, and furthermore, multiplying it by any matrix whose entries are rational functions in $z$ and which vanishes at $z=0$ leads to a matrix whose entries have no poles at $z=0$ and which vanishes along $z=0$. By an earlier computation, ${\rm d}U$, and hence ${\rm d}U. U^{-1}$, is strictly block upper triangular modulo terms which vanish along $z=0$.
 It follows that any product of the matrices $ \Lambda^{-1} {\rm d}\Lambda$ and ${\rm d}U. U^{-1}$ involving at least one factor of the form ${\rm d}U. U^{-1}$ is strictly block upper triangular modulo terms which vanish along $z=0$, and therefore has vanishing trace at $z=0$. We deduce that
\begin{gather*}
\beta_{\Lambda_G}^{4k+1} = \beta_{\Lambda}^{4k+1} + \text{terms vanishing at } z=0 .
\end{gather*}
Since $k\geq 1$, Lemmas~\ref{lem: BasicPropertiesBeta}$(vi)$ and~\ref{lem: projinv} imply that
 \begin{gather*}
 \beta^{4k+1}_{\Lambda} = \beta^{4k+1}_{z \Lambda_{\gamma}} + \beta^{4k+1}_{\Lambda_{G/\gamma}} = \beta^{4k+1}_{\Lambda_{\gamma}} + \beta^{4k+1}_{\Lambda_{G/\gamma}}.
 \end{gather*}
 In particular, $\beta^{4k+1}_{\Lambda}$ and hence $\beta^{4k+1}_{\Lambda_G}$ have no poles at $z=0$, and we conclude that
 \begin{gather*}
 \beta_{\Lambda_G}^{4k+1} \Big|_{z=0} = \beta^{4k+1}_{\Lambda_{\gamma}} + \beta^{4k+1}_{\Lambda_{G/\gamma}}.
 \end{gather*}
 Since this calculation holds in every local affine chart, we deduce that
 \begin{gather*}
 \widetilde{\beta}^{4k+1}_{\Lambda_G}= \widetilde{\beta}_{\Lambda_\gamma}^{4k+1} \wedge 1 + 1\wedge \widetilde{\beta}_{\Lambda_{G/\gamma}}^{4k+1}.
 \end{gather*}
 Since $\Delta_{\can} \beta ^{4k+1} = \beta^{4k+1} \otimes 1+ 1 \otimes \beta^{4k+1}$, this proves \eqref{RestrictomegaToDgammaAndCoproduct}. The case of a general element in~$\Omega_{\can}$ follows from the multiplicativity of the coproduct.
\end{proof}

\begin{rem}
Note that the previous theorem gives another way to derive the asymptotic ``factorisation'' formula $\Psi_G \sim \Psi_{\gamma} \Psi_{G/\gamma}$ which lies behind \eqref{Dgammaproduct}, by inspecting the determinant of the matrix $\Lambda$ which occurs in the proof.
\end{rem}

Note that the core subgraphs $\gamma$ which occur in the previous theorem are not necessarily connected.

\begin{cor} For every canonical form $\omega \in \Omega^{n}_{\can}$, the collection $\{
\widetilde{\omega}_G\}_G$ defines an algebraic differential form of degree $n$ in the sense of Definition~$\ref{defn: formsonPtot}$.
\end{cor}
In this paper we will consider forms with poles along graph hypersurfaces only, even though the Definition \ref{defn: formsonPtot} allows more general polar loci in principle.

\subsection{Canonical cohomology classes} \label{sect: AbsCohomClass}
We deduce the existence of universal compatible families of closed differential forms, and hence cohomology classes, on the complements of graph hypersurfaces.
\begin{defn} For every $\omega \in \Omega^{k}_{\can}$ we may define canonical (absolute) cohomology classes for every graph $G$:
\begin{gather*}
[\widetilde{\omega}_G]^{\mathrm{abs}} \in H_{{\rm d}R}^k\big(P^G \backslash Y^G\big).
\end{gather*}
\end{defn}
They satisfy a number of compatibilities including invariance under automorphisms and functoriality with respect to restriction to faces of the divisor $D$, which are cohomological versions of Definition~\ref{defn: formsonPtot}.
As a consequence, these classes are deduced from the graph hypersurface complement of the complete graph $K_n$, for $n$ sufficiently large, by restriction (since every graph is deduced from a complete graph by deleting edges).
Examples suggest that $[\widetilde{\omega}_G]^{\mathrm{abs}}$ is often zero.

\section{Canonical graph integrals and Stokes' formula}
We study integrals of canonical forms over coordinate simplices $\sigma_G$, which are always finite. We then apply Stokes' theorem to the Feynman polytope to deduce relations between canonical integrals.

\subsection{Integrals of canonical differential forms}
Let $\{\widetilde{\omega}\}$ be a closed algebraic differential form of degree $k$ as in Definition~\ref{defn: formsonPtot}.

\begin{defn}
Let $(G,\eta)$ be an oriented graph with $k+1$ edges. Define
\begin{gather*}
I_{(G,\eta)}\big(\{\widetilde{\omega}\}\big) = \int_{\widetilde{\sigma}_G} \widetilde{\omega}_G ,
\end{gather*}
where the orientation on $\widetilde{\sigma}_G$ is induced by the orientation $\eta$ on the edges of $G$.
Since $\widetilde{\omega}_G$ is smooth and the domain $\widetilde{\sigma}_G$ is compact, the integral is finite.
\end{defn}

\begin{lem} The integral $I$ is well-defined on the equivalence class $[G,\eta]$ and is thus defined on the level of the graph complex $\GC_N$, for any $N$ even.
\end{lem}
\begin{proof} Reversing orientations changes the sign:
\begin{gather*}
I_{(G,-\eta)} \big(\{\widetilde{\omega}\}\big) = - I_{(G,\eta)} \big(\{\widetilde{\omega}\}\big).
\end{gather*}
Furthermore, if $\tau\colon G \cong G$ is an automorphism of $G$, then
\begin{gather*}
I_{(G,\eta)} \big(\{\widetilde{\omega}\}\big) = I_{(G,\tau(\eta))} \big(\{\widetilde{\omega}\}\big)
\end{gather*}
by the functoriality property $\tau^* \widetilde{\omega}_G = \widetilde{\omega}_G$ which follows from Lemma~\ref{lem: autoinvariant}.
\end{proof}

 From now on we drop the orientation in the notation for $G$, and assume that all graphs are implicity oriented.
 We now let $\omega \in \Omega^k_{\can}$ be a canonical differential form.

\begin{cor}If $G$ has $k+1$ edges, the canonical integral equals
\begin{gather} \label{canintegralfinite} I_G\big(\{\widetilde{\omega}\}\big) = \int_{\sigma_G} \omega_G
\end{gather}
and is finite. It vanishes if any of the following are true: $G$ has a tadpole or a bridge, $G$ has a~vertex of degree $\leq 2$, $G$ has multiple edges, or $G$ is one-vertex reducible.
\end{cor}
\begin{proof} By Theorem~\ref{thm: nopolesonD}, $\widetilde{\omega}_G$ is a differential form in the sense of Definition~\ref{defn: formsonPtot} and so the
 canonical integral converges. It can be written as an integral over the open simplex $\sigma_G$ because the complement $\widetilde{\sigma}_G\backslash \sigma_G$ has Lebesgue measure zero. The vanishing statement is a consequence of Proposition~\ref{prop: vanishing}.
\end{proof}

It follows from duality properties (Lemma~\ref{lem: dualityandLG}) of canonical forms that $I_{G}(\{\omega\}) = I_{G^{\vee}}(\{\omega\})$ if $G$ and $G^{\vee}$ are planar graphs dual to each other.

In physics parlance, a graph $G$ is called divergent if $\deg_2 G \leq 0$, i.e., $2 h_G \geq e_G$.

\begin{lem} \label{lem: intGvanishes}
Suppose that $\omega$ is primitive $($e.g., $\omega$ is a generator of the form $\beta^{4k+1})$. Then the integral \eqref{canintegralfinite}
vanishes unless
\begin{gather*}
e_G = 2h_G ,
\end{gather*}
or equivalently, $\deg_2 G =0$.
 \end{lem}
\begin{proof}
 Since $\omega$ is primitive, and $\Lambda_G$ is a $h_G \times h_G$ matrix, Proposition~\ref{prop: betavanishes} implies that
\begin{gather*}
\omega_{G} = 0 \qquad \text{unless}\quad \deg \omega_{G} < 2h_{G} .
\end{gather*}
For the integral to be defined, $\deg \omega_{G} = e_{G}-1$ and therefore
$ e_{G} - 2 h_{G} \leq 0$. Now by Lemma~\ref{lem: betaviaLG}, we may write $\omega_G = \beta^{4k+1}_{L_G}$, where $L_G$ is the matrix \eqref{LGdefn} of size $v_G-1$, where $v_G$ is the number of vertices of $G$. By Proposition~\ref{prop: betavanishes},
\begin{gather*}
\omega_{G} = 0 \qquad \text{unless}\quad \deg \omega_{G} < 2(v_{G}-1) .
\end{gather*}
Using $v_G= e_G-h_G+1$ and the fact that $\deg \omega_{G} = e_{G}-1$ we conclude that $\omega_G$ vanishes unless $e_G\geq 2h_G$. This shows that $\omega_G$ vanishes unless $e_G = 2h_G$.
\end{proof}

As a result, integrals of \emph{primitive} forms will only detect elements in the zeroth homology of the graph complex $\GC_2$, which motivates the second part of Conjecture~\ref{conj: Mainconj} (namely equation~\eqref{FreePrimOmega} and the remark which follows it).

Classes in higher homology groups can in principle be detected by integrals of canonical forms which are not primitive.

\subsection{Relations from Stokes' theorem}
Stokes' theorem implies the following relation between graph integrals. It combines
the diffe\-ren\-tial in a graph complex with the coproduct both on graphs and on differential forms.

\begin{thm} Let $\omega \in \Omega^k_{\mathrm{can}}$ be a canonical form of degree $k$. Write its coproduct in the form
$\Delta_{\can} \omega = \sum_i \omega'_i \otimes \omega''_i$.
 For any graph $G$ with $k+2$ edges,
\begin{gather} \label{Stokes}
0 = \sum_{e \in E_G} \int_{\sigma_{G/e}} \omega_{G/e} + \sum_i \sum_{\gamma \subset G} \int_{\sigma_{\gamma}} (\omega_i')_{\gamma} \times \int_{\sigma_{G/\gamma}} (\omega_i'')_{G/\gamma},
\end{gather}
where the sum is over all core subgraphs $\gamma \subsetneq G$, such that
$e_{\gamma}= \deg \omega'_i+1$
and the orientation on $\sigma_{\Gamma}$, for $\Gamma \in \{G,\gamma, G/\gamma\}$, is induced by any fixed orientation on $G$.
 \end{thm}

\begin{proof} Here and later, we shall often write $\omega$ instead of $\omega_G$ to keep the notations uncluttered.
Applying Stokes' formula to the compact polytope $\widetilde{\sigma}_G$ gives
\begin{gather*}
0 = \int_{\widetilde{\sigma}_G} {\rm d} \widetilde{\omega} = \int_{\partial \widetilde{\sigma}_G} \widetilde{\omega} .
\end{gather*}
By Theorem~\ref{thm: NormalCrossing}, the boundary $\partial \widetilde{\sigma}_G$ is a union of facets $\widetilde{\sigma}_{G/e}$, where $e \in E_G$ is an edge, and
$\widetilde{\sigma}_{\gamma} \times \widetilde{\sigma}_{G/\gamma}$, where $\gamma \subset G$ is a core subgraph. Thus we obtain
\begin{gather*}
0 = \sum_{e \in E(G)} \int_{\widetilde{\sigma}_{G/e}} \widetilde{\omega}\big|_{\widetilde{\sigma}_{G/e}} + \sum_{\gamma \subset G} \int_{\widetilde{\sigma}_{\gamma} \times \widetilde{\sigma}_{G/\gamma}} \widetilde{\omega}\big|_{\widetilde{\sigma}_{\gamma} \times \widetilde{\sigma}_{G/\gamma}} .
\end{gather*}
By Theorem~\ref{thm: nopolesonD}, we have
\begin{gather*}
 \widetilde{\omega}\big|_{\widetilde{\sigma}_{\gamma} \times \widetilde{\sigma}_{G/\gamma}} = \sum_{i} \widetilde{\omega}_i'\big|_{\widetilde{\sigma}_{\gamma}} \wedge \widetilde{\omega}_i''\big|_{\widetilde{\sigma}_{G/\gamma}} .
 \end{gather*}
 Since $\widetilde{\sigma}_{\gamma}$ has dimension $e_{\gamma}-1$, the restriction of the holomorphic form $\widetilde{\omega}'_i$ to it vanishes unless $\deg \widetilde{\omega}'_i \leq e_{\gamma}-1$.
Similarly, $\deg \widetilde{\omega}''_i \leq e_{G/\gamma}-1$ is also required for non-vanishing of the differential form $ \widetilde{\omega}''_i $, and hence
\begin{gather*} \deg \omega = \deg \widetilde{\omega}'_i+ \deg \widetilde{\omega}''_i \ \leq \ e_{\gamma} +e_{G/\gamma}-2 = e_{G}-2 . \end{gather*}
Since this is an equality, we deduce that $e_{\gamma}= \deg \omega'_i+1$.
\end{proof}

The quadratic terms in the right-hand side of \eqref{Stokes} include:
\begin{gather} \label{1timesomega}
\int_{\sigma_{\gamma}} 1 \times \int_{\sigma_{G/\gamma}} \omega_{G/\gamma}
\end{gather}
whenever $G$ contains a core 1-edge subgraph $\gamma$, i.e., a tadpole. If $G$ has no tadpoles the terms~\eqref{1timesomega} never occur. Similarly, the quadratic terms in \eqref{Stokes} also include
\begin{gather} \label{omegatimes1}
\int_{\sigma_{\gamma}} \omega_{\gamma} \times \int_{\sigma_{G/\gamma}} 1
\end{gather}
whenever $\gamma\subset G$ is a core subgraph and $G/\gamma$ has a single edge. In this situation $\gamma = G \backslash e$ for $e$ an edge in $G$. Thus these terms can be rewritten in the form
\begin{gather*}
 \sum_{e\in E_G} \int_{\sigma_{G\backslash e}} \omega_{G\backslash e}
 \end{gather*}
since by Proposition~\ref{prop: vanishing}$(v)$ such an integral vanishes unless $G\backslash e$ is core.

\begin{cor}
If $G$ has no tadpoles we may rewrite \eqref{Stokes} in the form
\begin{gather} \label{Stokesv2} 0 = \sum_{e\in E_G} \bigg(\int_{\sigma_{G/e}} \omega_{G/e} + \int_{\sigma_{G\backslash e}} \omega_{G\backslash e} \bigg) + \sum_{\gamma \subset G} \int_{\sigma_{\gamma} \times \sigma_{G/\gamma}} \Delta_{\can}' \omega,\end{gather}
where $\Delta'_{\can} \omega= \Delta_{\can} \omega - 1\otimes \omega - \omega \otimes 1$ is the reduced coproduct on $\Omega_{\can}$.
\end{cor}
 \begin{rem} \label{rem: Termsvanish} It can often happen that terms in the formula \eqref{Stokesv2} vanish.
 The terms \eqref{omegatimes1} vanish if, for example, for every edge $e$ of $G$, the graph $G\backslash e$ has a vertex of valency $\leq 2$. The latter is guaranteed if $G$ has no two vertices of valency $\geq 4$ which are connected by an edge.

Likewise, the quadratic terms where $\omega_i'$ and $\omega''_i$ are non-trivial (have degree $>0$)
\begin{gather} \label{Quadterms} \int_{\sigma_{\gamma}} \omega_i' \times \int_{\sigma_{G/\gamma}} \omega_i''
\end{gather}
 often vanish. For example, if $\omega =\omega^{4m+1} \wedge \omega^{4n+1} $ is the wedge product of two primitive forms,
 then because $\omega_i'$ and $\omega_i''$ are both primitive, Lemma~\ref{lem: intGvanishes} implies that \eqref{Quadterms} vanishes
 unless $\deg_2 \gamma =\deg_2 G/\gamma =0$, and in particular, $\deg_2 G = 0$.

Further vanishing criteria can be obtained by combining Lemma~\ref{lem: intGvanishes} with the fact that if a~graph $\Gamma$ satisfies $3h_{\Gamma} -e_{\Gamma} \leq 2$ then it has a vertex of valency $\leq 2$ and thus vanishes in $\GC_2$.
 \end{rem}

 \subsection{Detecting graph homology classes}
 Using the formula \eqref{Stokesv2}, one can deduce the existence of non-vanishing homology classes in the graph complex from the
 non-vanishing of canonical integrals. A simple case is as follows.

 \begin{cor} \label{cor: IGnonzero}
 Suppose that $G\in \GC_2$ of degree $0$ is closed $({\rm d}G =0)$ and homogeneous of edge degree $e$. Let $\omega \in \Omega^{e-1}_{\can}$ be a primitive canonical form of degree $e-1$. If the canonical integral is non-vanishing:
 \begin{gather*}
 I_{G}(\omega) = \int_{\sigma_G} \omega_G \neq 0
 \end{gather*}
 then the class $[G]\in H_0(\GC_2)$ is non-zero.
 \end{cor}
 \begin{proof}
 Suppose that $G = {\rm d}X$, where $X$ is a linear combination of graphs in $\GC_2$ of degree $1$. Applying formula \eqref{Stokesv2} to $X$ implies that
 \begin{gather*}
 0 = \int_{{\rm d}X} \omega + \int_{\delta X} \omega.
 \end{gather*}
 By Lemma~\ref{lem: intGvanishes}, the restriction of $\omega$ to $\delta X$ vanishes, since $\deg_2 \delta X= \deg_2 X+1 >0$. We~there\-fore deduce that $0 = \int_{{\rm d}X} \omega=I_G(\omega)$, a contradiction.
 \end{proof}
 The proof implies that if $\omega \in \Omega_{\can}$ is primitive, and $X \in \GC_2$ has degree $\deg_2 X=1$ in the graph complex with edge-grading $\deg(\omega)+2$, then we have:
 \begin{gather} \label{IomegaRelations}
 \int_{{\rm d}X} \omega=0.
 \end{gather}
 See Section~\ref{sect: Examples} for examples of relations between canonical integrals obtained in this way.

There exist more elaborate versions of Corollary~\ref{cor: IGnonzero} involving diagram chases around the graph complex.
We describe a basic mechanism in the next paragraph.

\subsection{Applications} \label{sect: precip} The following argument shows how cohomology classes may appear in unexpected degrees in the graph complex.

\begin{lem} 
Suppose that $X \in \GC_2$ of degree $\deg_2 X = n$ such that ${\rm d}X = \delta X=0$, and let $\omega \in \Omega_{\can}$ with $\deg \omega = e(X)-1$. Either
\begin{gather*}
[X] \in H_n(\GC_2)
\end{gather*}
is non-zero, or there exists $X' \in \GC_2$ satisfying ${\rm d}X' = \delta X' =0$ of degree $\deg_2 X' = n+2$ with
\begin{gather*}
\int_{X'} \omega \equiv \int_{X} \omega \mod \text{\rm (products of canonical integrals)}.
\end{gather*}
If $n\geq 0$ and $\omega$ is a linear combination of products of at most two primitive canonical forms, then in fact
\begin{gather*}
\int_{X'} \omega = \int_{X} \omega.
\end{gather*}
\end{lem}
\begin{proof}
If the homology class $[X]$ were to vanish, then there exists $Y \in \GC_2$ such that ${\rm d}Y=X$. Set $X' = - \delta Y$. Since $\delta^2=0$ we have $\delta X'=0$. Using ${\rm d} \delta = - \delta {\rm d}$ we also deduce that ${\rm d}X'=0$.
Now apply \eqref{Stokesv2} to $Y$ to obtain
\begin{gather*}
0 = I_{{\rm d}Y}(\omega) + I_{\delta Y} (\omega) + I_{\Delta' Y}(\Delta'_{\can} \omega),
\end{gather*}
 which implies that $I_{X'}(\omega) = I_X(\omega) + I_{\Delta' Y}(\Delta'_{\can} \omega) $. The term $ I_{\Delta' Y}(\Delta'_{\can} \omega) $ is a linear combination of products of canonical integrals of factors of $\omega$, which proves the first statement.
 For the second, note that the degree of $Y$ equals $n+1\geq 1$, and so $ I_{\Delta' Y}(\Delta'_{\can} \omega) $ vanishes when $\omega$ is a linear combination of products of two primitive forms $\omega^{4i_1+1} \wedge \omega^{4i_2+1}$, by Remark~\ref{rem: Termsvanish}.
 \end{proof}

\begin{cor} \label{cor: detectclassv2} Suppose that $X \in \GC_2$ of degree $\deg_2 X = n\geq0 $ such that ${\rm d}X = \delta X=0$, and let $\omega \in \Omega_{\can}$
of degree $\deg \omega = e(X)-1$
be a linear combination of products of at most $2$ primitive canonical forms, such that
\begin{gather*}
 \int_{X} \omega \neq 0.
 \end{gather*}
Then there exists an $m\geq 0$, and an element $X_m \in \GC_2$ satisfying ${\rm d}X_m = \delta X_m=0$, such that its homology class $[X_m] \in H_{n+2m}(\GC_2)$ is non-zero and
\begin{gather*}
\int_{X_m} \omega = \int_{X} \omega.
\end{gather*}
\end{cor}
\begin{proof}
Apply the previous lemma repeatedly to $X_0=X$ to obtain a sequence $X=X_0$, $X_1,\dots,X_m $ of elements satisfying ${\rm d} X_i = \delta X_i=0$ and such that $I_{X_i}(\omega) = I_X(\omega)$. Since the loop number of $X_i$ decreases by 1 at each step, this process terminates after a finite number of steps by Proposition~\ref{prop: vanishing}, and the last one in the sequence must be a non-trivial homology class.
\end{proof}

In Section~\ref{sect: omega59} we apply the corollary to an element $X = \delta G$, where $G$ represents a class in~$H_0(\GC_2)$ with non-trivial coproduct.

\begin{rem} In \cite{SpectralSequenceGC2} it is shown that
$H(\GC_0, {\rm d}+ \delta) \cong \bigoplus_{n\geq 1} \Q [W_{2n+1}] $ is generated by the wheel classes. Thus, for any homogeneous $Z\in \GC_0$ such that ${\rm d}Z = \delta Z =0$ which is not proportional to a wheel class, there exists $X$ such that $({\rm d}+ \delta) X=Z$. By applying \eqref{Stokesv2} we deduce that
\begin{gather*}
I_{Z}(\omega) = I_{\Delta' (X)}( \Delta'_{\can}(\omega)),
\end{gather*}
for
 any canonical differential form $\omega \in \Omega^{e_Z-1}_{\can}$. In particular, canonical integrals of any such $Z$ are trivial modulo products of lower order canonical integrals.
\end{rem}

\section{Outer motive and canonical motivic periods of graphs}

\subsection{A motive associated to the graph complex}
For any connected oriented graph $G$, one can define the
graph motive \cite{BEK, Cosmic}
\begin{gather*}
\mathrm{mot}_G = H^{e_G-1} \big( P^G \backslash Y_G, D \backslash (D \cap Y_G)\big)
\end{gather*}
which is to be viewed in a category $\mathcal{H}_{\Q}$ of realisations over $\Q$ (see, for example, \cite{NotesMot, deligneP1}).
 If $G$ has connected components $G_1,\dots, G_n$, define $\mathrm{mot}_G$ to be $\bigotimes_{i=1}^n \mathrm{mot}_{G_i}$. The objects $\mathrm{mot}_G$ are equipped with face maps \cite{Cosmic}
\begin{align}
i_e\colon\ &\qquad\quad\ \, \mathrm{mot}_{G/e}\To \mathrm{mot}_G, \nonumber
\\
 i_{\gamma}\colon\ &\mathrm{mot}_{\gamma} \otimes \mathrm{mot}_{G/\gamma}\To \mathrm{mot}_G, \label{motface}
\end{align}
as well as maps induced by isomorphisms $\tau\colon G \cong G'$ which we denote by:
\begin{gather} \label{mottau}
\tau\colon\ \mathrm{mot}_{G'} \cong \mathrm{mot}_G.
\end{gather}
Note that the face maps increase the cohomological degree by one and correspond to boundary maps in cohomology.
Define the ind-motive of all graphs (resp.~of bounded genus) to be a limit of the graph motives with respect to \eqref{motface} and \eqref{mottau}:
\begin{gather} \label{TotalMotive}
\mathrm{Mot_{Graphs}} = \bigoplus_G \mathrm{mot}_G /{\sim}, \qquad
\mathrm{Mot^{\leq g}_{Graphs}} = \bigoplus_{h_G \leq g} \mathrm{mot}_G /{\sim}.
\end{gather}
For the second line of \eqref{motface}, this means that the images of any two face maps
\begin{gather*}
 i\colon\ \mathrm{mot}_{g} \otimes \mathrm{mot}_{h} \To \mathrm{mot}_G \qquad \text{and} \qquad
 i'\colon\ \mathrm{mot}_{g} \otimes \mathrm{mot}_{h} \To \mathrm{mot}_{G'}
 \end{gather*}
are identified, and we take $\sim$ to be the equivalence relation generated by this together with $i_e$ (which is actually a special case of $i_{\gamma}$, since $\mathrm{mot}_e$ is the trivial object), and $\tau$. Since for any two graphs $g,h$ one can insert $g$ into a vertex of $h$ to obtain a graph $G$ such that $g \leq G$ and $G/g= h$, we deduce a product
\begin{gather*}
\mathrm{Mot^{\leq a}_{Graphs}}\otimes \mathrm{Mot^{\leq b}_{Graphs}} \To \mathrm{Mot^{\leq a+b}_{Graphs}} ,
\end{gather*}
which is canonical, by definition of $\sim$. A similar product exists on $\mathrm{Mot_{Graphs}}$ by dropping the restriction on loop numbers.
The object $\mathrm{Mot_{Graphs}^{\leq g}}$ could be viewed as a motive associated to Outer space $\mathcal{O}_g$. Note that, in reality, we are actually interested in the smallest quotients of~\eqref{TotalMotive} whose dual Betti realisation contains the projective limit of the Betti classes $[\widetilde{\sigma}_G]$ defined presently, but for simplicity we will say nothing more about this.

\subsection{Motivic period integrals}
If $G$ is equipped with an orientation, the Feynman polytope defines by Theorem~\ref{thm: NormalCrossing} a canonical Betti homology class
\begin{gather*}
[\widetilde{\sigma}_G] \in (\mathrm{mot}_G)^{\vee}_B
\end{gather*}
which satisfies the following properties with respect to face maps:
\begin{gather*}
(i^{\vee}_e)^{B} [\widetilde{\sigma}_G] = [\widetilde{\sigma}_{G/e}],
\\
 (i^{\vee}_{\gamma})^{B} [\widetilde{\sigma}_G] =
 [\widetilde{\sigma}_{\gamma}] \otimes [\widetilde{\sigma}_{G/\gamma}]
\end{gather*}
induced by the boundary map applied to graph polytopes, where $\widetilde{\sigma}_{G/e}$, $\widetilde{\sigma}_{\gamma}\times \widetilde{\sigma}_{G/\gamma}$ are given the induced orientations. Furthermore, isomorphisms $\tau\colon G \cong G'$ induce
\begin{gather*}
 (\tau^{\vee})^B [\widetilde{\sigma}_G] = [\widetilde{\sigma}_{G'}],
 \end{gather*}
where $\tau$ is compatible with the orientations on $G$, $G'$ (in the case when it reverses orientations, the previous formula has a minus sign).

Now let $\omega\in \Omega^k_{\can}$ be a canonical differential form of degree $k$, and suppose that $G$ is an oriented graph with $k+1$ edges.
By Theorem~\ref{thm: nopolesonD}, the form $\widetilde{\omega}_G$ has no poles along $D\subset P^G$, and therefore its restriction to $D$ vanishes, because $D$ is of dimension $<k$.
 It therefore defines a relative cohomology class
\begin{gather*}
[\widetilde{\omega}_G] \in (\mathrm{mot}_G)_{{\rm d}R}
\end{gather*}
whose image under the natural map $(\mathrm{mot}_G)_{{\rm d}R} \rightarrow H_{{\rm d}R}^{e_G-1} \big( P^G \backslash Y_G\big)$ is the absolute class $[\widetilde{\omega}_G]^{\mathrm{abs}}$ defined in Section~\ref{sect: AbsCohomClass}.
\begin{defn} Let $G$ be an oriented graph with $e_G= k+1$ edges.
Define the motivic canonical integral to be the ``motivic period'' \cite{NotesMot}
 \begin{gather*} 
 I_G^{\mm}(\omega) = [\mathrm{mot}_G, [\widetilde{\sigma}_G], [\widetilde{\omega}_G]]^{\mm},
 \end{gather*}
 where the orientation on $\widetilde{\sigma}_G$ is given by that of $G$.
 \end{defn}

The canonical integral $I_G(\omega)$ can be retrieved from its motivic version by applying the period homomorphism~\cite{NotesMot}, i.e., $I_G(\omega) =\operatorname{per} I^{\mm}_G(\omega)$.

\begin{lem}
The motivic period $I^{\mm}_G(\omega)$ only depends on the class of $G$ in $\GC_N$.
 \end{lem}
\begin{proof}
Reversing the orientation of $G$ reverses the sign of $[\widetilde{\sigma}_G]$ and hence of $I^{\mm}_G(\omega)$. Functoriality with respect to isomorphisms:
\begin{gather*}
I^{\mm}_G(\omega) = I^{\mm}_{\tau(G)} (\omega)
\end{gather*}
 follows from the formalism of motivic periods and the fact (Lemma~\ref{lem: autoinvariant}) that $\omega$ is invariant with respect to automorphisms. Finally, it follows from Proposition~\ref{prop: betavanishes} that $I^{\mm}_G(\omega)$ vanishes if $G$ has a two-valent vertex, since $\omega_G$ and hence $\widetilde{\omega}_G$ already vanishes.
\end{proof}

 It is undoubtedly true that $I^{\mm}_G(\omega)$ and $I^{\mm}_{G^{\vee}}(\omega)$ are equal when $G$ is a~planar graph and $G^{\vee}$ a~planar dual, but the argument is more delicate.

 \subsection{Cosmic Galois group and Outer space}
 In \cite{Cosmic}, the cosmic Galois group (a name first suggested by Cartier) was defined to be the quotient of the (de Rham) Tannaka group of the category $\mathcal{H}_{\Q}$ which acts on the system\footnote{To be more precise, on the system consisting of the smallest quotient objects $\mathrm{mot}_G\rightarrow {}_{\sigma}\mathrm{mot}_G$ with the property that $[\widetilde{\sigma}_G] \in \big(\mathrm{mot}^B_G\big)^{\vee}$ is in the image of $\big({}_{\sigma}\mathrm{mot}^B_G\big)^{\vee}$.}
 of objects $\mathrm{mot}_G$. It is a pro-algebraic group over $\Q$ which acts on the de Rham vector spaces $\mathrm{mot}^{{\rm d}R}_G$ in such a~way that it respects the (de Rham versions of) the face maps \eqref{motface} and \eqref{mottau}. In particular, it acts on $\mathrm{Mot_{Graphs}}$ and
 $\mathrm{Mot_{Graphs}^{\leq g}}$
 and respects the relations betwen motivic periods
 \begin{gather}
I^{\mm}_{G/e}(\omega) = I^{\mm}_{G}\big( i_e^{{\rm d}R} \omega\big),\nonumber
\\
 I^{\mm}_{\gamma}(\omega) I^{\mm}_{G/ \gamma}(\omega') = I^{\mm}_{G}\big(i_{\gamma}^{{\rm d}R} (\omega\otimes \omega')\big) \nonumber,
 \\
 I^{\mm}_{G}( \omega) = I^{\mm}_{G'}\big(\tau^{{\rm d}R} \omega\big),\label{motivicfacerels}
 \end{gather}
 where $\omega$, $\omega'$, $\omega'' $ are de Rham cohomology classes of the appropriate degrees and $\tau\colon G\cong G'$. It~must be emphasized that the maps $i^{{\rm d}R}_e$, $i^{{\rm d}R}_{\gamma} $ increase cohomological degree, and are not to be confused with the restriction maps which go into Definition~\ref{defn: formsonPtot}.

\subsection{Motivic Stokes formula}
The motivic periods $I^{\mm}_G(\omega)$, where $\omega$ is a canonical form, vanish in all the situations listed in~Proposition~\ref{prop: vanishing}.

\begin{thm} The motivic version of \eqref{Stokes} holds. If $\omega$ is a canonical form of degree $k$, and $G$ has $k+2$ edges, then:
\begin{gather} \label{motStokes}
0 = \sum_{e \in E(G)} I^{\mm}_{G/e}(\omega) + \sum_i \sum_{\gamma \subset G} I^{\mm}_{\gamma}(\omega'_i) I^{\mm}_{G/\gamma}(\omega''_i),
 \end{gather}
where the second sum is over all core subgraphs $\gamma \subset G$ with $e_{\gamma} = \omega'_i-1$ and $\Delta_{\can} \omega = \sum_i \omega'_i \otimes \omega''_i$ is the coproduct applied to $\omega$.
\end{thm}

\begin{proof} The proof using Stokes' formula is valid in the context of motivic periods since it can be expressed in terms of face maps via the long relative sequence of cohomology.
Concretely, one deduces from this the relation
\begin{gather} \label{CohomologicalStokes}
0= \sum_e i_e([\widetilde{\omega}_{G/e}]) + \sum_i \sum_{\gamma \subset G} i_{\gamma}\big(\big[\widetilde{\omega}'_i\big|_{\gamma}\big] \otimes \big[\widetilde{\omega}''_i\big|_{G/\gamma}\big]\big) \end{gather}
and the identity then follows from the relations \eqref{motivicfacerels}.
\end{proof}

\subsection{A question about the motivic Galois action} \label{sect: QuestionGalois}
 A canonical differential form $\omega$ of degree $k$ defines a \emph{collection of classes}
\begin{gather*}
[\widetilde{\omega}_G] \in \mathrm{mot}^{{\rm d}R}_G
\end{gather*}
 for all $G$ with $e_G=k+1$ edges.
 This collection satisfies the properties that it
 \begin{itemize}\itemsep=0pt
 \item is functorial with respect to isomorphisms of graphs,
 \item vanishes on graphs satisfying the conditions of Proposition~\ref{prop: vanishing},
 \item satisfies the cohomological relations \eqref{CohomologicalStokes}.
 \end{itemize}
 The cosmic Galois group respects all these properties.
 Consider the $\Q$-vector space $\mathsf{H}_{\can}$ gene\-rated by the images of the classes
 $[\widetilde{\omega}_G]$, for all $G$, under the de Rham versions of the maps \eqref{motface} and \eqref{mottau}. It can be viewed
as a $\Q$-subspace of the de Rham total motive~\eqref{TotalMotive}
\begin{gather*}
 \mathsf{H}_{\can} \subset \mathrm{Mot^{{\rm d}R}_{Graphs}}.
 \end{gather*} Since the cosmic Galois group acts on this inductive limit (i.e., it respects the maps \eqref{motface} and~\eqref{mottau}),
it is natural to ask if it preserves the space $\mathsf{H}_{\can}$. If so, it would be very interesting to know how it acts upon it.

 The examples in Section~\ref{sect: Examples} seem to suggest, for example, that
 the cosmic Galois group preserves the subspace generated by the canonical classes $1$ and $\omega^5$.

 Note that we do not suggest that the cosmic Galois group necessarily acts directly on $\Omega_{\can}$: it is conceivable that a canonical form $\omega$ gives rise to algebraically independent motivic periods~$I^{\mm}_G(\omega)$ and $ I^{\mm}_{G'}(\omega)$ with entirely different Galois actions.

 \subsubsection{Relation to Feynman integrals}
 For any oriented graph $G$ with edges numbered from $1,\dots, n$, let us write
\begin{gather} \label{OmegaGprojform}
\Omega_G = \sum^n_{i=1} (-1)^i \alpha_i {\rm d}\alpha_1 \wedge \dots \wedge \widehat{{\rm d}\alpha_i} \wedge \dots \wedge {\rm d}\alpha_n.
\end{gather}
 Quantum field theory provides, for every $G$ with $\deg_2 G=0$ which has no subgraphs $\gamma$ of degree $\deg_2 \gamma<0$, a canonical differential form
 \begin{gather*} 
 \omega^{\mathrm{Feyn}}_G = \frac{\Omega_G}{\Psi_G^2}
 \end{gather*}
 which by \cite{BEK} defines a form $\widetilde{\omega}^{\mathrm{Feyn}}_G=\pi_G^* \omega^{\mathrm{Feyn}}_G$ whose class is
 \begin{gather*}
 \big[\widetilde{\omega}^{\mathrm{Feyn}}_G\big] \in \mathrm{mot}^{{\rm d}R}_G.
 \end{gather*}
 The period integrals of these classes, called Feynman residues:
 \begin{gather*}
 I^{\mathrm{Feyn}}_{G} = \int_{\sigma_G} \omega^{\mathrm{Feyn}}_G = \int_{\sigma_G} \frac{\Omega_G}{\Psi_G^2}
 \end{gather*}
 have been studied intensely (see \cite{Schnetz} for a survey of known results).
The set of all such classes generates under the maps \eqref{motface} and \eqref{mottau} a $\Q$-vector space $\mathsf{H}_{\mathrm{Feyn}}$.
It can be viewed as a subspace
\begin{gather*}
 \mathsf{H}_{\mathrm{Feyn}} \ \subset \ \mathrm{Mot^{{\rm d}R}_{Graphs}}.
 \end{gather*}
The examples of classes $\omega \in \mathsf{\Omega}_{\can}$ considered in Section~\ref{sect: Examples} seem to be contained in $\mathsf{H}_{\mathrm{Feyn}}$. For example, one can express $W_4$ as a minor of $W_5$ by contracting one edge and deleting another. Therefore by applying the two corresponding face maps,
we can view the degree 7 class $\big[\widetilde{\omega}^{\mathrm{Feyn}}_{W_4}\big] \in \mathrm{mot}^{\mathrm{dR}}_{W_4}$ as a class of degree 9 in $\mathrm{mot}^{\mathrm{dR}}_{W_{5}}$. We expect that it is proportional to the class of the canonical form $[\widetilde{\omega}^9]$ (see discussion in Section~\ref{sect: Examples}).
In practice, this means that canonical integrals seem to reduce to Feynman residues by integration-by-parts identities, at least for graphs of small loop order, i.e., $ \mathsf{H}_{\mathrm{can}} $ appears to be contained in $ \mathsf{H}_{\mathrm{Feyn}} $ at low orders. It would be very interesting to know if this is always the case, and to understand in more detail the relationship between the spaces $\mathsf{H}_{\mathrm{can}}$, $\mathsf{H}_{\mathrm{Feyn}}$ and $\mathsf{H}_{\mathrm{can}} \cap \mathsf{H}_{\mathrm{Feyn}} $.

\section{Examples} \label{sect: Examples}
In the following examples, we will orient our graphs so that the integrals of canonical forms are non-negative.
In each example, the first step in computing a canonical integral is to compute the integrand in parametric form using its definition and some of the tricks described in earlier sections (notably a suitable $LBU$ decomposition). For the first few examples,
the integrals themselves can then be computed directly using the algorithm of \cite{Massless, PeriodsFeynman} which has been implemented in \cite{Bogner,Panzer}; the later ones require the more powerful approach of \cite{BorinskySchnetz}. The fact that the latter method is applicable uses Remark~\ref{rem: degreexe}, as pointed out by Schnetz.

\subsection[The form omega5]
{The form $\boldsymbol{\omega^5}$}
The canonical form of degree $5$ was computed in Example~\ref{examples: smallbetas}. It is non-vanishing only on the wheel with 3 spokes, the unique graph of degree zero at 3 loops in $\GC_2$ (all other graphs with 3 loops and 6 edges have a doubled edge or two-valent vertex). The form $\omega^5_{W_3}$ was computed in Examples~\ref{example: W3} and~\ref{example: W3form} and satisfies
\begin{gather*}
\omega^5_{W_3} = 10 \omega^{\mathrm{Feyn}}_{W_3}.
\end{gather*}
Its canonical integral is thus proportional to the Feynman residue and gives
\begin{gather*}
I_{W_3}\big(\omega^5\big) = 10 I^{\mathrm{Feyn}}_{W_3} = 60 \zeta(3).
\end{gather*}
Since the (de Rham) Galois conjugates of the motivic version of $\zeta(3)$ are $1$ and itself, this example provides some possible evidence in favour of
Section~\ref{sect: QuestionGalois}.

\subsection[The form omega9]
{The form $\boldsymbol{\omega^9}$}
Let $G$ be the wheel with 5 spokes, and let $S_5\subset E_{W_5}$ denote its five inner spoke edges. With the notation \eqref{OmegaGprojform}, one can compute:
\begin{gather*}
\omega^{9}_{W_5} = 18 \bigg( \frac{ 1}{\Psi_{W_5}^2} + 12 \frac{\prod_{e\in S_5} x_e }{\Psi_{W_5}^3} \bigg) \Omega_{W_5}.
\end{gather*}
The corresponding canonical integral is
\begin{gather*}
I_{W_5}\big(\omega^9\big) = 1260 \zeta(5).
\end{gather*}
It can be computed using the software implementations mentioned at the beginning of this section. The integral of the first term
\begin{gather*}
\omega^{\mathrm{Feyn}}_{W_5} = \frac{\Omega_{W_5}}{ \Psi_{W_5}^{2} }
\end{gather*} is convergent and proportional to $\zeta(7)$, which is the Feynman residue of $W_5$. Thus the canonical integral $I_{W_5}\big(\omega^9\big)$ has ``weight drop'', and indeed one checks that $\big[\widetilde{\omega}^9_{W_5}\big]^{\mathrm{abs}}$ vanishes.
 Hodge-theoretic considerations \cite[Section~7.5, Example~9.7]{Cosmic} imply that this integral is related via face maps $\iota_{\gamma}$ to periods of minors of $W_5$.
 Concretely, the integrand $\omega^{9}_{W_5}$ is exact, and so it would be interesting, by a double application of Stokes' formula (or for instance by \cite[Proposition 37]{PeriodsFeynman}) to relate it explicitly to the Feynman period of
the wheel with \emph{four} spokes $W_4$, which is a minor of $W_5$ obtained by contracting one edge and deleting another.

The Feynman residue of the latter \cite{BroadhurstKreimer, Schnetz} is
\begin{gather*}
I^{\mathrm{Feyn}}_{W_4}= \int_{\sigma_{W_4}} \omega_{W_4}^{\mathrm{Feyn}} = \int_{\sigma_{W_4}}\frac{\Omega_{W_4}}{\Psi_{W_4}^2} = 20 \zeta(5) .
\end{gather*}
 This suggests that the cohomology class $[\widetilde{\omega}^9_{W_5}]$ is in the image of $\mathsf{H}^{\mathrm{Feyn}}$ (Section~\ref{sect: QuestionGalois}).
The same comment applies to the graph $Z_5$ in the figure below.

\begin{figure}[h]
\centering
\includegraphics[width=11cm]{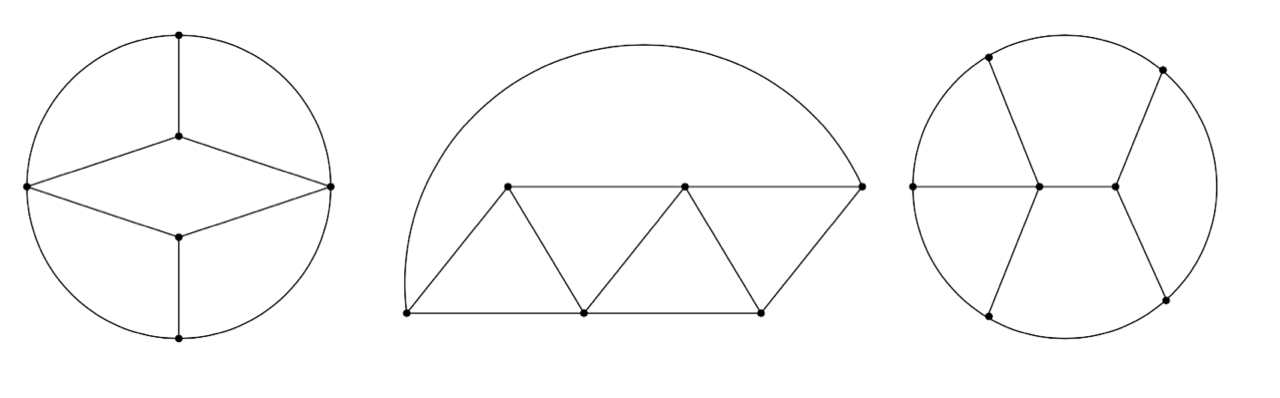}
\put(-60,0){ $X_5$}\put(-310,0){ $T_5=W_3:W_3$}\put(-175,0){ $Z_5$}
\caption{Two five-loop graphs with 10 edges (left), and a five loop graph with 11 edges (right).}\label{figureTZX}
\end{figure}

The form $\omega^9$ pairs with a number of other graphs with 10 edges and $5$ loops. Two are depicted in Figure \ref{figureTZX}: a graph $T_5$ which is a two-vertex join of $W_3$ with itself (Remark~\ref{rem: 2VJ}), and the zig-zag graph $Z_5$. One calculates, with some effort, that
\begin{gather*}
I_{T_5} \big(\omega^9\big)= 0 \qquad \text{and} \qquad I_{Z_5} \big(\omega^9\big)= 630 \zeta(5).
\end{gather*}
Interestingly, $\omega^9_{T_5}$ is not identically zero, although its integral vanishes. These results are consistent with the formula~\eqref{Stokes}. Indeed, one verifies that the graph homology class $[T_5]$ is zero, and that with suitable orientations,
\begin{gather*}
 {\rm d} X_5 = 2 Z_5 - W_5,
 \end{gather*}
 where $X_5$ is the graph depicted in Figure~\ref{figureTZX} on the far right. This identity implies the following relation between homology classes
 \begin{gather*}
 [W_5] = 2 [Z_5] \in H_0(\GC_2).
 \end{gather*}
 By the motivic version of \eqref{IomegaRelations} it also implies that
 \begin{gather} \label{W5Z5motivicrelation}
 I^{\mm}_{W_5}\big(\omega^9\big) = 2 I^{\mm}_{Z_5} \big(\omega^9\big).
\end{gather}
Thus we see that \eqref{Stokes} transfers information in a non-trivial way between
different graphs. The motivic version \eqref{motStokes} implies an explicit constraint on the action of the cosmic Galois group: Galois conjugates of motivic Feynman periods of the different graphs $Z_5$ and $W_5$ are constrained by the relation \eqref{W5Z5motivicrelation}.

\subsection[The form omega5 ^ omega9]
{The form $\boldsymbol{\omega^5 \wedge \omega^9}$}\label{sect: omega59}
Recall that it follows from \eqref{WillwacherGRT} and \eqref{MTZinjection} that there exists an element $\xi_{3,5} \in \GC_2$ with $16$~edges, $8$~loops, of degree zero, which satisfies ${\rm d} \xi_{3,5}=0$ and is dual to $[\sigma_3, \sigma_5]$.
Since the antisymmetrized Connes--Kreimer coproduct is dual to the Lie algebra structure on graph cohomology Section~\ref{sect: FurtherStructures}, it follows that $\Delta' \xi_{3,5} = W_3 \otimes W_5 - W_5 \otimes W_3$ plus possible extra terms involving graphs with tadpoles or vertices of degree $\leq 2$ whose canonical integrals vanish by Proposition~\ref{prop: vanishing}, and which we can ignore.

 Apply equation \eqref{Stokesv2} to $\xi_{3,5}$ and $\omega= \omega^5 \wedge \omega^9$ together with the above computations for the wheel integrals to deduce that
\begin{gather*}
\int_{\delta \xi_{3,5}} \omega^5 \wedge \omega^9 \in \Q^{\times} \zeta(3) \zeta(5),
\end{gather*}
where $\delta \xi_{3,5} \in \GC_2$ has edge grading 15, and loop grading $7$.
Since ${\rm d} \xi_{3,5}=0$ we deduce that ${\rm d} (\delta \xi_{3,5})= 0$, and we may apply Corollary~\ref{cor: detectclassv2} with $X= \delta \xi_{3,5}$ to deduce the existence of a~non-trivial class in either $H_1(\GC_2)$ or $H_3(\GC_2)$ with the same canonical integral. The computer calculations mentioned in the introduction show that $H_1(\GC_2)$ vanishes at 7 loops, and hence
\begin{cor}
There exists an element
$\Xi_{3,5} \in \GC_2$ at $15$ edges, and $6$ loops with the property that ${\rm d} \Xi_{3,5}=0$
 such that
 \begin{gather*}
 I_{\Xi_{3,5}}\big(\omega^5 \wedge \omega^9\big) = \zeta(3) \zeta(5).
 \end{gather*}
Its homology class is non-zero:
\begin{gather*}
0 \neq [\Xi_{3,5}] \in H_3(\GC_2).
\end{gather*}
 \end{cor}

 Similar arguments by applying \eqref{Stokesv2} along the lines of Section~\ref{sect: precip} can be used to compute
other examples of non-trivial pairings between canonical forms and graph homology (see Table~\ref{tableForms}). Note the similarity between this argument and that of~\cite{SpectralSequenceGC2}, except for the additional role played by the Lie coalgebra structure.

\subsubsection[The complete graph K6]{The complete graph $\boldsymbol{K_6}$}

Recall from Example~\ref{Example: Kn} that the Laplacian $L_{K_6}$ of the complete graph $K_6$ corresponds to the generic symmetric matrix of rank~$5$.
 One verifies that the canonical form $\omega^5 \wedge \omega^9$ is proportional to the invariant volume form. One can subsequently deduce from this, with the help of a~computer, that
 \begin{gather*}
 \omega_{K_6}^5 \wedge \omega^9_{K_6} = \frac{9!}{8} \frac{\prod_{e \in E_{K_6}} x_e }{\Psi_{K_6}^3} \Omega_{K_6} ,
 \end{gather*}
from which it is obvious that the associated canonical integral is positive and hence non-zero.
 Schnetz, using the method of \cite{BorinskySchnetz}, has computed
 \begin{gather*}
 I_{K_6} \big(\omega^5 \wedge \omega^9\big) = \frac{9!}{16}\bigg(360 \zeta(3,5)+690 \zeta(3)\zeta(5)- \frac{29 \pi^8}{315}\bigg)=1708.1901\dots.
 \end{gather*}
 The multiple zeta value $\zeta(3,5) =\sum_{1\leq n_1<n_2}\frac{1}{n_1^3n_2^5}$ is expected to be transcendental over the $\Q$-algebra generated by odd zeta values.
 It would be very interesting to relate this integral, via Stokes' formula and face maps, to the Feynman residue of the complete bipartite graph $K_{3,4}$, as one has the following identity:
 \begin{gather*}
 I_{K_6} \big(\omega^5 \wedge \omega^9\big) = \frac{9!}{16}\bigg(15 \zeta(3)\zeta(5)-\frac{25}{96} \int_{\sigma_{K_{3,4}}} \omega^{\mathrm{Feyn}}_{K_{3,4}} \bigg)
 \end{gather*}
 (the Feynman residue for $K_{3,4}$ is called $P_{6,4}$ in~\cite{Schnetz}). It would be very interesting to interpret this identity by relating it to the
 Borel regulator~\cite{Siegel}.

\subsection{Further wheels}
 For the wheel with seven spokes,
we check that
\begin{gather*}
\omega^{13}_{W_7} = 26 \big(1 + 60 Y+ 360 Y^2 \big) \frac{\Omega_{W_7}}{\Psi_{W_7}^2}, \qquad \text{where}\quad Y= \frac{\prod_{e\in S_7} x_e }{\Psi_{W_7}}
\end{gather*}
 and $S_7\subset E_{W_7}$ denotes the internal spokes of $W_7$. Its canonical integral is evidently positive and hence non-zero.
Schnetz has confirmed using \cite{BorinskySchnetz} that
\begin{gather*}
I_{W_7}\big(\omega^{13}\big)= 24024\zeta(7).
\end{gather*}
 In general, one can write a graph Laplacian for wheel matrices explicitly as in \cite[formula~(11.3)]{BEK} and use formula \eqref{PulloutEigenvalues} to compute the canonical forms to leading order. We can easily deduce that, for example
\begin{gather*}
\omega^{4n+1}_{W_{2n+1}} \equiv (8n+2) \omega^{\mathrm{Feyn}}_{W_{2n+1}}
\bigg(\!\!\!\!\!\mod \prod_{e \in S_{2n+1} }x_e \bigg) .
\end{gather*}
We expect that $I^{\mm}_{W_{2n+1}}\big(\omega^{4n+1}\big)$ is a non-zero rational multiple of the motivic odd zeta value $\zeta^{\mathfrak{m}}(2n+1)$ of weight $2n+1$.
Computations to appear in the forthcoming preprint~\cite{BorinskySchnetz} suggest that the rational coefficient is given by
\begin{gather*}
I_{W_{2n+1}}(\omega^{4n+1}) \overset{?}{=} (2n+1) \binom{4n+2}{2n+1}\zeta(2n+1) .
\end{gather*}

\begin{rem}
The above examples suggest considering the following family of period integrals. For any odd wheel $W_{2n+1}$, with $n\geq 1$, consider
\begin{gather*}
I_n^{(k)} = \int_{\sigma_{W_{2n+1}}} \bigg(\frac{\prod_{e\in S_{2n+1}} x_e}{\Psi_{W_{2n+1}}} \bigg)^{\!k} \frac{\Omega_{W_{2n+1}}}{\Psi^2_{W_{2n+1}}}
\end{gather*}
 for all $k \geq 0$, where $S_{2n+1}$ denotes the internal spokes of $W_{2n+1}$. A standard Picard--Fuchs argument implies that they satisfy recurrence relations in $k$.
 It is shown in \cite{BorinskySchnetz} using Gegenbauer polynomial techniques that
 \begin{gather*}
 I^{(k)}_n= \frac{2}{(2k+2)!} \binom{4n}{2n} \sum_{m=1}^{\infty} \frac{ \prod_{\ell=1}^k \big(m^2-\ell^2\big) }{m^{4n-1}}
 \end{gather*}
which, by expanding the product in the previous expression, is a sum of odd single zetas with weights from $4n-2k-1$ to $4n-1$.
 \end{rem}

\subsection{Summary}
 The following table summarizes the canonical integrals which are known or conjectured at present (right three columns) and compares them with Feynman residues of primitive divergent graphs (left two columns).
\begin{gather*}
\setlength{\tabcolsep}{4.5pt}\def\arraystretch{1.3}
\begin{array}{c|c||c|c|c} \hline
\text{Graph } G & \text{Feynman } I_G^{\mathrm{Feyn}} & \text{Graph } G & \omega\in \Omega_{\can} &\text{Canonical } I_G(\omega) \\ \hline %
W_3 & 6 \zeta(3) & W_3 & \omega^5 & 60 \zeta(3) \\
 W_4 & 20 \zeta(5) & W_5 & \omega^9 & 1260 \zeta(5) \\
W_5 & 70 \zeta(7) & W_7 & \omega^{13} & 24024 \zeta(7) \\
\vdots & \vdots & \vdots & \vdots &\vdots
\\
 W_{n+1} & \binom{2n}{n} \zeta(2n-1) & W_{2n+1} & \omega^{4n+1} & (2n+1) \binom{4n+2}{2n+1} \zeta(2n+1) ? \\[2pt] \hline
 W_3:W_4 & 120 \zeta(3) \zeta(5) & \Xi_{3,5} & \omega^{5} \wedge \omega^9 & \zeta(3) \zeta(5) \\ \hline
 K_{3,4} & * \zeta(3) \zeta(5) + * P_{3,5} & K_6 & \omega^{5} \wedge \omega^9 & * \zeta(3) \zeta(5) + *P_{3,5} \\ \hline
\end{array}
\end{gather*}
In the above table, the quantity
\begin{gather*}
P_{3,5} = \zeta(3,5) - \frac{29}{12} \zeta(8)
\end{gather*}
is a multiple zeta value of weight 8 and an asterisk denotes a known rational number.
The fact that the canonical integral for $K_6$ produces the same period $P_{3,5}$ strongly suggests that it can be reduced to the Feynman residue of $K_{3,4}$ (and to products of the Feynman residues for~$W_3$,~$W_4$) by application of Stokes' formula and relations in the cohomology of graph hypersurface complements. Note that $K_{3,4}$ is not a minor of $K_6$, but the Feynman residue for~$K_{3,4}$ has weight drop (lower than expected transcendental weight) which suggests that $P_{3,5}$ is a~generalised Feynman period of a~graph which is a~common minor of both $K_{3,4}$ and $K_6$, which might explain its simultaneous appearance in $I^{\mathrm{Feyn}}_{K_{3,4}}$ and $I_{K_6}\big(\omega^5 \wedge \omega^9\big)$.

\subsection*{Acknowledgements} This project has received funding from the European Research Council (ERC) under the European Union's Horizon 2020 research and innovation programme (grant agreement no.~724638). Many thanks to M.~Chan (who spotted a mistake in an earlier proof of Theorem~\ref{thm: nopolesonD}), S.~Galatius (who pointed out that~\eqref{mu2nvanishes} follows from the Amitsur--Levitzki theorem), S.~Payne, G.~Segal for discussions and especially R.~Hain and K.~Vogtmann, of which the present project is an offshoot of joint work. I am very grateful to O.~Schnetz, for computing the above examples of canonical integrals, M.~Borinsky, for sharing his computations of Euler characteristics, C.~Dupont and the referees for many helpful comments and corrections.

\pdfbookmark[1]{References}{ref}
\LastPageEnding

\end{document}